\documentclass{amsart}
\usepackage{amsmath,amssymb,verbatim,mathtools}
\usepackage{enumerate}
\newtheorem{theorem}{Theorem}[section]
\newtheorem{lemma}[theorem]{Lemma}
\newtheorem{proposition}[theorem]{Proposition}
\newtheorem{corollary}[theorem]{Corollary}

\theoremstyle{definition}

\theoremstyle{remark}
\newtheorem{remark}{Remark}[section]
\numberwithin{equation}{section}

\DeclareMathOperator{\sech}{sech}

\DeclareMathOperator{\spann}{span}
\DeclareMathOperator{\supp}{supp}
\DeclareMathOperator{\diag}{diag}

\newcommand{\R}{\mathbb{R}}
\newcommand{\C}{\mathbb{C}}

\newcommand{\Z}{\mathbb{Z}}
\newcommand{\la}{\langle}
\newcommand{\ra}{\rangle}
\newcommand{\pd}{\partial}
\newcommand{\eps}{\epsilon}
\newcommand{\mM}{\mathcal{M}}
\newcommand{\mY}{\mathcal{Y}}
\newcommand{\tb}{\tilde{\beta}}
\newcommand{\tg}{\tilde{g}}
\newcommand{\tu}{\tilde{u}}
\newcommand{\tchi}{\tilde{\chi}}
\newcommand{\mF}{\mathcal{F}}
\newcommand{\mL}{\mathcal{L}}
\newcommand{\wmL}{\widetilde{\mathcal{L}}}
\newcommand{\calX}{\mathcal{X}}
\renewcommand{\a}{\alpha} 
\renewcommand{\k}{\kappa}
\begin{document}
\title{Spectral stability of resonant Y-solitons for the KP-II equation}
\author{Tetsu Mizumachi}
\thanks{The author was supported by JSPS KAKENHI Grant Numbers JP21K03328
and JP24H00185.}
\address{Division of Mathematical and Information Sciences\\
Hiroshima University, 739 - 8521 Japan}
\email{tetsum@hiroshima-u.ac.jp}
\begin{abstract}
  The KP-II equation was introduced by Kadomtsev and Petviashvili to
  explain the transverse stability of line solitary waves in shallow
  water. Using Darboux transformations, we prove the spectral
  stability of Y-shaped line solitons in an exponentially weighted
  space that captures the amount of waves receding behind the line
  solitons. We show that the spectrum of the linearized operator
  consists of the continuous spectrum of the zero background together
  with three residual spectral domains whose boundaries contain
  continuous spectra associated with the constituent line solitons.

Exploiting the analytic dependence of the pairings with adjoint eigenfunctions
on the spectral parameters, we prove that the orthogonality
conditions associated with all adjoint eigenfunctions can be reduced
to orthogonality conditions against three families of adjoint
continuous eigenfunctions.

Moreover, we investigate linear stability of Y-solitons in a subspace
satisfying orthogonality conditions against adjoint continuous
eigenfunctions associated with outgoing modulations along each line
soliton constituting the Y-soliton.
\end{abstract}
\maketitle
\section{Introduction}
\label{sec:intro}
The KP-II equation
\begin{equation}\label{eq:KPII}
\partial_x(4\pd_tu+\pd_x^3u+3\pd_x(u^2))+3\partial_y^2u=0
\quad\text{for $t>0$ and $(x,y)\in \R^2$,}
\end{equation}
is a generalization to two spatial dimensions of the KdV equation
\begin{equation}
  \label{eq:KdV}
4\pd_tu+\pd_x^3u+3\pd_x(u^2)=0\,,
\end{equation}
and has been derived to explain the transverse stability of solitary wave solutions
to the KdV equation when the surface tension is weak or absent.
See \cite{KP} for the derivation of \eqref{eq:KPII}.
The Cauchy problem of \eqref{eq:KPII} has been studied by
\cite{Bourgain,Hadac,HHK,IM,Tak,TT,Tz,Ukai} and by \cite{MST}
on the background of line solitons.
The spectral stability of $1$-line solitons on $\R^2$ was proved by \cite{APS}.
Stability and phase shifts of $1$-line solitons on $\R^2$ was studied by
\cite{Miz15,Miz18,Miz19} in exponentially and polynomially weighted spaces.
Recently, a $1$-codimensional stable manifold on which modulating $1$-line solitons have no phase shift has been constructed by \cite{Pompili}.
Linear stability of elastic $2$-line solitons has been studied by \cite{Miz23}
and spectral stability of periodic solutions was studied by \cite{HLP}.
Space-time decay of solutions for the KP-II equation as well as the KP-I equation in regions that do not include solitons or lumps has been studied by
\cite{MMPP24}. They extend  virial identities obtained by
\cite{deBouard-Martel}. The KP-II equation is a $2+1$-dimensional
integrable system and the inverse scattering method for solutions around
line soliton solutions have been studied by
\cite{VA,BPP12,BPP,Wudc1,Wudc2,Wudc3}.
\par
In this paper, we consider spectral and linear stability of resonant line
solitons of Y-shape. Multi-line solutions of the KP-II
equation can be expressed as
\begin{equation}
  \label{eq:exact}
u(t,x,y)=2\pd_x^2\log \tau(t,x,y)\,,  
\end{equation}
\begin{gather}
  \label{eq:tau-linesol}
\tau=K\Theta B\,,\enskip
K=
\begin{pmatrix}
1 &  \ldots & 1 \\
\k_1 & \ldots & \k_M \\
\vdots  &  & \vdots \\
\k_1^{N-1}  & \ldots & \k_M^{N-1}
\end{pmatrix}\,,\enskip
\Theta=\diag(e^{\theta_1},\cdots,e^{\theta_M})\,,
\\ \notag
\k_1<\k_2<\cdots<\k_M\,,\quad 
\theta_j=\k_jx+\k_j^2-\k_j^3t\,,
\end{gather}
and $B$ is an $M\times N$ matrix whose $N$-th minors are nonnegative.
See, e.g. \cite{CK,Hirota,Nim-Fr}.
In particular, if $M=2$, $N=1$ and $\tau(t,x,y)=e^{\theta_1}+e^{\theta_2}$,
then we have a solution of the $1$-line soliton
$$ u(t,x,y)=\frac12(\k_2-\k_1)^2\sech^2\frac{\theta_2-\theta_1}{2}\,.$$
Since
$\theta_2-\theta_1=(\k_2-\k_1)\{x+(\k_1+\k_2)y-(\k_1^2+\k_1\k_2+\k_2^2)t\}$,
the line soliton propagates toward $(1,\k_1+\k_2)$.
\par
Let
\begin{equation}
  \label{eq:tau-Y}
u_{\mY}=2\pd_x^2\log\tau_{\mY}\,,\quad \tau_{\mY}=e^{\theta_1}+e^{\theta_2}+e^{\theta_3}\,,
\end{equation}
that is $M=3$, $N=1$ and $B=\begin{pmatrix}  1 & 1 &1\end{pmatrix}$
in \eqref{eq:tau-linesol}.
In the region on the $xy$-plane where one of the exponential functions $e^{\theta_i}$
$(1\le i\le 3)$ is dominant,
we have $u=2\pd_x^2\log \tau\simeq 2\pd_x^2\theta_i=0$.
For $\pm y\gg 1$, line solitons of $u=2\pd_x^2\log\tau_{\mY}$ appear
in the region $\theta_i\simeq \theta_j$ $(i\ne j)$
if two of exponential terms $e^{\theta_i}$ and $e^{\theta_j}$ in \eqref{eq:tau-Y} are of the
same order and much larger than the other.
The peak of line soliton is located along a line $\theta_1=\theta_3 $ for $y\gg 1$
and along lines $\theta_1=\theta_2$ and $\theta_2=\theta_3$ for $y\ll -1$.
We call each line soliton as $[1,3]$-soliton,
$[1,2]$-soliton and $[2,3]$-soliton.
The asymptotics of $u_{\mY}$ is given by
\begin{equation}
  \label{eq:u1-asymptotics}
u_{\mY}(t,x,y)\simeq \left\{
  \begin{gathered}  
u_{13}  \quad\text{for $y\gg1$,}
\\ u_{12}+u_{23} \quad\text{for $y\ll-1$,}
\end{gathered}\right.
\end{equation}
where
\begin{gather}
  \label{eq:ujk}
  u_{jk}=2\pd_x^2\log(e^{\theta_j}+e^{\theta_k})=2c_{jk}\sech^2\sqrt{c_{jk}}z_{jk}\,,
  \\ \notag
z_{jk}=x+2a_{jk}y-\omega_{jk}t\,,\enskip a_{jk}=\frac{\k_j+\k_k}{2}\,,
\\ c_{jk}=\frac{(\k_k-\k_j)^2}{4}\,,\enskip
\omega_{jk}=\k_j^2+\k_j\k_k+\k_k^2\,.
\end{gather}
\par
Let
\begin{equation}
  \label{eq:b1-b2Y}
b_1=-(\k_1\k_2+\k_2\k_3+\k_3\k_1)\,,\; b_2=\k_1+\k_2+\k_3\,,\;
X=x-b_1t\,,\; Y=y-b_2t\,.
\end{equation}
Then $b_1+2a_{jk}b_2=\omega_{jk}$ and 
\begin{gather*}
\theta_k-\theta_j=(\k_k-\k_j)(X+2a_{jk}Y)
=\sqrt{c_{ij}}z_{ij}
\quad\text{for $1\le j<k\le 3$,}  
\end{gather*}
and \eqref{eq:tau-Y} is the stationary solution in the
moving coordinate $(X,Y)$. 
\par
By the change of variables
$(\k_1,\k_2,\k_3)\mapsto (-\k_1,-\k_2,-\k_3$) and $y\mapsto -y$,
\eqref{eq:tau-Y} is transformed into another Y-shaped soliton
$$u=2\pd_x^2\log\tau\,,\quad
\tau=\sum_{1\le j<k\le 3}e^{\theta_j+\theta_k}\,,
$$
which have $[1,2]$-soliton and $[2,3]$-soliton for $y\gg1$ and 
$[1,3]$-soliton for $y\ll -1$.
\par
The Y-shaped soliton is fundamental because it is a prototype of
generic multi-line solitons which have web patterns in the middle.
Moreover, the Y-shaped solitons can be used to explain Mach reflection.
See e.g. \cite{LYK,KY,RHB22}.
\par
In this paper, we will prove the spectral stability of $u_{\mY}$
in an exponentially weighted space $\calX$ defined by \eqref{eq:defcalX1}.
We introduce the weighted space in order to measure how waves move out from areas
around Y-solitons following the idea of the pioneering work \cite{PW}
that proves asymptotic stability of KdV solitons.
Indeed, for $1$-line solitons, these weighted estimates
make it possible to measure interactions
between modulating line solitons and the remainder part of solutions
(\cite{Miz15,Miz18,Miz19}). 
We expect our result would be the first step to justify modulation analysis of
perturbed Y-solitons (\cite{HanKodamaLi,RHB22}).
\par
For linearized operators around line solitons, there are two types of
continuous eigenfunctions. One is plane wave solutions that behave like
solutions to the linearized equation around $0$ at $Y\simeq\pm\infty$
and the other is continuous eigenfunctions that are localized
as $x\to\infty$ but do not decay in the transverse direction.
Let $\mL_0$ be the linearized operator of \eqref{eq:KPII}
around $0$, $\mL_\mY$ be the linearized operator of \eqref{eq:KPII}
around \eqref{eq:tau-Y} and $\mL_{jk}$ be a linearized operator of \eqref{eq:KPII}
around $u_{jk}$ for $(j,k)=(1,3)$, $(1,2)$, $(2,3)$. 
We remark that continuous eigenmodes assocaited with $\sigma(\mL_0)$ are
distorted plane waves and that continuous eigenmodes assocaited with
$\sigma(\mL_{jk})\setminus\sigma(\mL_0)$ decay as $z_{jk}\to\infty$
but do not decay in the transverse direction.
\par
In the present paper, we prove that
$\sigma(\mL_\mY)=\sigma(\mL_0)\cup\overline{\sigma_r(\mL_\mY)}$ and that
$\sigma(\mL_{jk})\setminus\sigma(\mL_0)$
$(1\le j<k\le3)$ are on the boundaries of $\sigma_r(\mL_\mY)$. 

Utilizing the analyticity of the secular conditions with respect to
the spectral parameter, we prove that if a function is orthogonal to
the adjoint continuous eigenfunctions corresponding to three families
of outgoing modulations along the $[1,2]$- and $[2,3]$-solitons (or to
three families of incoming modulations along the $[1,3]$- and
$[2,3]$-solitons), then it is orthogonal to every adjoint
eigenfunction, thanks to Lusin--Privalov-type theorems
(see Lemma~\ref{lem:orth-condition-identity} and
Corollary~\ref{cor:orth-condition-identity}).
\par
Our plan of the present paper is as follows.
In Section~\ref{sec:resonant}, we give explicit formulas of continuous eigenfunctions
for $\mL_\mY$ and $\mL_\mY^*$ by using Jost functions and dual Jost functions.
We compute the group velocity of continuous eigenfunctions that are localized
in the $x$-direction and find that modulations along $[1,2]$-soliton are incoming,
that modulations along $[2,3]$-soliton are outgoing, and that we have both outgoing and incoming modulations
for $[1,3]$-soliton.
We state our main result at the end of Section~\ref{sec:resonant}.
\par
In Section~\ref{sec:[i,j]-stability},
we collect spectral stability results of $1$-line solitons posed in oblique directions.
In Section~\ref{sec:spectral}, we will prove that if
$\lambda\in\sigma(\mL_\mY)$ and
$\lambda\in\cap_{1\le j<k\le 3}\rho(\mL_{jk})\cap\rho(\mL_0)$, then
$\lambda$ is either an eigenvalue of $\mL_\mY$ or $\mL_\mY^*$
by using the spectral stability of $[j,k]$-solitons.
We find that there are three domains of adjoint eigenvalues of $\mL_\mY$,
with $\sigma(\mL_{jk})\setminus\sigma(\mL_0)$ $(1\le j<k\le 3)$
lying on their respective boundaries.
Furthermore, we will prove that $u\in\calX$ is orthogonal to the
adjoint eigenfunctions if $u$ is orthogonal to adjoint
continuous eigenfunctions that correspond to outgoing modulations
along $[1,2]$-soliton and $[1,3]$-soliton.
\par
In Section~\ref{sec:nonexistence}, we will prove that
there is no eigenvalue for $\mL_\mY$ and $\mL_\mY^*$ in a half plane
$\{\lambda\in\C\mid \Re\lambda>-\delta\}$, where $\delta>0$.
The eigenvalue problem of $\mL_\mY$ is a $1+1$-dimensional linear equation
akin to Boussinesq equations.
Note that if a solution of \eqref{eq:KPII} is time independent, then it is
a solution of a $1$-dimensional good Boussinesq equation, where we regard $y$
as a time variable. Combining this property with Darboux transformations
and the linearized modified KP-II equation, we prove non-existence of eigenvalues.
Our method can be regarded as a $2$-dimensional extension of
\cite[Appendix~B]{PW} that
computes the Evans function for KdV and mKdV by using Darboux transformations. 
Miura transformations have been used to study nonlinear stability of
 KdV solitons and KPII $1$-line solitons (\cite{BK,MV,MT,Pompili}).
But we have some difficulty in treating the phase shifts of line-solitons on $\R^2$
via Miura transformations. This is one of our motivations to study
linear stability of line soliton solutions.
\par
In Section~\ref{sec:proof-main}, we investigate linear stability of
Y-soliton on a subspace of $\calX$ satisfying secular conditions.
Let $\wmL$ be a resriction of $\mL_\mY$ to a subspace $\calX_\mY$ satisfying
the secular condition.
Gearhart-Pr\"{u}ss theory (\cite{Gh,Pruss}) tells us that $e^{t\wmL}$
is exponentially stable if 
$\sigma(\wmL)\subset \{\lambda\in\C\mid \Re \lambda\le -b\}$ for a $b>0$.
We find that there is a $b>0$ such that $\lambda-\wmL$ is invertibile
if $\lambda\ne0$ and $\Re\lambda\ge -b$ for a $b>0$.  But it remains
open whether $0\in\rho(\wmL)$.  We prove linear stability
of Y-solitons in a weighted space whose weight function grows slightly
faster than that of $\calX$ as $y\to\infty$.
\bigskip

\section{Dispersion relations of waves along a Y-soliton}
\label{sec:resonant}
In this section, we will introduce explicit formulas of
continuous eigenmodes of $\mL_\mY$ and $\mL_\mY^*$
that are localized in $x$ but do not decay in the transverse direction,
and state our main result at the end.
\par

For the purpose, we recall the definitions of Jost solutions.
We say that $\Phi$ is a Jost solution if $\Phi$ satisfies
\begin{gather}
\label{eq:Jost-eqs}
L\Phi=0\,,\quad B\Phi=0\,,
\\ \notag
L=-\pd_y+\pd_x^2+u\,,\quad
B=4(\pd_t+\pd_x^3)+6u\pd_x+3\left(\pd_xu+\pd_x^{-1}\pd_yu\right)\,,
\end{gather}
and $\Phi^*$ is a dual Jost solution if $\Phi^*$ satisfies
\begin{gather}
\label{eq:dualJost-eqs}
L^*\Phi^*=0\,,\quad B^*\Phi^*=0\,.
\end{gather}

It is well known that the $x$ derivative of a product of a Jost solution and a dual Jost solution
gives a solution of the linearized KP-II equation (the squared eigenfunction method).
\begin{lemma}
\label{lem:prodJdJ}
Let $\Phi$ be a solution of \eqref{eq:Jost-eqs}
and $\Phi^*$ be a solution of \eqref{eq:dualJost-eqs}.
Then $v=\pd_x(\Phi\Phi^*)$ satisfies
\begin{equation}
  \label{eq:linearKP}
\pd_x(4\pd_tv+6\pd_x(uv)+\pd_x^3v)+3\pd_y^2v=0\,,
\end{equation}
and $w=\Phi\Phi^*$ satisfies the adjoint equation
\begin{equation}
  \label{eq:adjLKP}
\pd_x(4\pd_tw+6u\pd_xw+\pd_x^3w)+3\pd_y^2w=0\,.  
\end{equation}
\end{lemma}
Lemma~\ref{lem:prodJdJ} follows from a direct computation.
The Jost solution and the dual Jost solution of the line soliton $u_{\mY}$
can be written as
\begin{gather}
  \label{eq:Jost-sol}
\Phi^1(k)=\Phi^0(k)
\frac{\tau_\Phi(k)}{\tau}\,,\quad 
  \Phi^{1,*}(k)=\Phi^{0,*}(k)\frac{\tau_{\Phi^*}(k)}{\tau}\,,
\\  \label{def:Phi0}
\Phi^0(k)=e^{ikx-k^2y+ik^3t}\,,\quad
\Phi^{0,*}(k)=e^{-ikx+k^2y-ik^3t}\,.
\end{gather}
where $\tau=\tau_{\mY}$ and  $\tau_\Phi$ and $\tau_{\Phi^*}$
are the shifts of the $\tau$-function:
\begin{gather*}
  \tau_\Phi(k)=\sum_{1\le j\le 3}(ik-\k_j)e^{\theta_j}\,,\quad
\tau_{\Phi^*}(k)=\sum_{1\le j\le 3}\frac{e^{\theta_j}}{ik-\k_j}\,.
\end{gather*}
See \cite{CK,Dickey}.  Let
$\Phi^1_j=\Phi^1(-ik_j)$ and
$\Phi_j^{1,*}=i\operatorname{Res}_{k=-ik_j}\Phi^*(k)=1/\tau_{\mY}$.
Then
\begin{equation}
  \label{def:Phi1s}
\left\{\begin{gathered}
\Phi^1_1=\frac{(\k_1-\k_2)e^{\theta_1+\theta_2}+
  (\k_1-\k_3)e^{\theta_1+\theta_3}}{e^{\theta_1}+e^{\theta_2}+e^{\theta_3}}\,,
\\
\Phi^1_2=\frac{(\k_2-\k_1)e^{\theta_1+\theta_2}+
  (\k_2-\k_3)e^{\theta_2+\theta_3}}{e^{\theta_1}+e^{\theta_2}+e^{\theta_3}}\,,
\\
\Phi^1_3=\frac{(\k_3-\k_1)e^{\theta_1+\theta_3}+
  (\k_3-\k_2)e^{\theta_2+\theta_3}}{e^{\theta_1}+e^{\theta_2}+e^{\theta_3}}\,,
\end{gathered}\right.
\end{equation}
\begin{equation}
  \label{eq:4}
  \Phi^1_1+\Phi^1_2+\Phi^1_3=0\,.
\end{equation}
\par

Now we consider continuous eigenmodes that are associated with
$[j,k]$-soliton.
Let $(X,Y)$ be the moving coordinate as \eqref{eq:b1-b2Y}.
Then $u_{\mY}$ is a stationary solution of
\begin{equation}
  \label{eq:KPII''}
4\pd_tu+\pd_X^3u-4b_1\pd_Xu-4b_2\pd_Yu+3\pd_X(u^2)+3\pd_X^{-1}\pd_Y^2=0\,.  
\end{equation}
The linearized equation of \eqref{eq:KPII} around $u_{\mY}$ can
be written as $\pd_t u=\mL_\mY u$, where
\begin{align*}
\mL_\mY=&
-\frac14\pd_X(\pd_X^2-4b_1+6u_{\mY})
 +b_2\pd_Y-\frac{3}{4}\pd_X^{-1}\pd_Y^2
\\=& -\frac14\pd_{z_{jk}}\left(\pd_{z_{jk}}^2-4c_{jk}+6u_{\mY}\right)
     +(b_2-3a_{jk})\pd_y-\frac{3}{4}\pd_{z_{jk}}^{-1}\pd_y^2
\end{align*}
for $1\le j<k\le 3$. Let
$\gamma_{jk}(\eta)=\sqrt{c_{jk}+i\eta}$, 
$\beta_{jk}^\pm(\eta)=a_{jk}\pm\gamma_{jk}(\eta)$ and
\begin{gather}
\label{def:lambdajk}
\lambda_{jk}^\pm(\eta)=i\eta\{b_2-3a_{jk}\pm\gamma_{jk}(\pm\eta)\}\,,
\quad
\tilde{\lambda}_{jk}^\pm(\eta)=i\eta\{b_2-3a_{jk}\mp\gamma_{jk}(\pm\eta)\}\,,
\end{gather}
\begin{gather*}
g_{jk}^+(X,Y,\eta)=e^{-\lambda_{jk}^+(\eta)t}
\pd_X\left\{\Phi^1(-i\beta_{jk}^-(\eta))\Phi^{1,*}_j\right\}\,,
\\ 
g_{jk}^-(X,Y,\eta)=e^{-\lambda_{jk}^-(\eta)t}
\frac{i\eta}{\k_j-\k_k}
\pd_X\left\{\Phi^1_k\Phi^{1,*}(-i\beta_{jk}^+(-\eta))\right\}\,,
\\
g_{jk}^{*,+}(X,Y,\eta)=
e^{\overline{\lambda_{jk}^+(\eta)}t}\frac{i\bar{\eta}}{\k_j-\k_k}
\Phi^1_j\Phi^{1,*}(-i\beta_{jk}^-(-\bar{\eta}))\,,
\\ 
g_{jk}^{*,-}(X,Y,\eta)=e^{\overline{\lambda_{jk}^-(\eta)}t}
\Phi^1(-i\beta_{jk}^+(\bar{\eta}))\Phi^{1,*}_k\,,
\\ 
\tg_{23}^{*,-}(X,Y,\eta)=e^{\overline{\tilde{\lambda}_{23}^-(\eta)}t}
\Phi^1(-i\beta_{23}^-(\bar{\eta}))\Phi^{1,*}_2\,.  
\end{gather*}
Since
\begin{equation}
  \label{eq:tauXY}
\begin{multlined}
\tau=e^{(\theta_{i_1}+\theta_{i_2})/2}
\biggl\{2\cosh\frac{(\k_{i_2}-\k_{i_1})z_{i_1i_2}}{2}
\\ +\exp\left(\frac{(\k_{i_3}-\k_{i_1})z_{i_1i_3}}{2}
+\frac{(\k_{i_3}-\k_{i_2})z_{i_2i_3}}{2}\right)\biggr\}
\end{multlined}  
\end{equation}
for $i_1$, $i_2$, $i_3$ satisfying $\{i_1,i_2,i_3\}=\{1,2,3\}$ and
\begin{align}
\label{eq:btitj}
  & \beta_{jk}^\pm(\eta)x+\beta_{jk}^\pm(\eta)^2y-\beta_{jk}^\pm(\eta)^3t
    -\frac{1}{2}(\theta_i+\theta_j)
  \\ =& \notag
\pm \gamma_{jk}(\eta)z_{jk} + i\eta Y
+i\eta\{b_2-3a_{jk}\mp\gamma_{jk}(\eta)\}t\,,
\end{align}
 we formally have for $1\le j<k\le 3$,
\begin{equation}
  \label{eq:lmij}
\left\{
  \begin{aligned}
& \mL g_{jk}^\pm(\eta)=\lambda_{jk}^\pm(\eta)g_{jk}^\pm(\eta)\,,
\\ &    
\mL^*g_{jk}^{*,\pm}(\eta)=\overline{\lambda_{jk}^\pm(\eta)}g_{jk}^{*,\pm}(\eta)\,,
\quad
\mL^*\tg_{23}^{*,-}(\eta)=\overline{\tilde{\lambda}_{23}^-(\eta)}\tg_{23}^{*,-}(\eta)\,.
  \end{aligned}
\right.
\end{equation}

\begin{remark}
  \label{rem:wavemotion}
For $\eta\simeq0$ and $Y\gg1$,
  \begin{equation*}
e^{\lambda_{13}^\pm(\eta)t} g_{13}^\pm(X,Y,\eta)
\simeq g_{13}^0(z_{13},\pm\eta)
e^{t\Re\lambda_{13}^\pm(\eta)+i\eta\{Y+t\Im\pd_\eta\lambda_{13}^\pm(0)\}}\,,
  \end{equation*}
 and for $\eta\simeq0$ and $Y\ll-1$,
 \begin{gather*}
e^{\lambda_{12}^\pm(\eta)t} g_{12}^\pm(X,Y,\eta)
\simeq g_{12}^0(z_{12},\pm\eta)
e^{t\Re\lambda_{12}^\pm(\eta)+i\eta\{Y+t\Im\pd_\eta\lambda_{12}^\pm(0)\}}
\text{ near $[1,2]$-soliton,}
\\
e^{\lambda_{23}^\pm(\eta)t}g_{23}^\pm(X,Y,\eta)
\simeq g_{23}^0(z_{23},\pm\eta)
e^{t\Re\lambda_{23}^\pm(\eta)+i\eta\{Y+t\Im\pd_\eta\lambda_{23}^\pm(0)\}}
\text{ near $[2,3]$-soliton.}
\end{gather*}
We infer from the above that wave speeds of modulations along
$[j,k]$-solitons are $-\Im\pd_\eta\lambda_{jk}^\pm(0)$.
Since
\begin{gather*}
-\Im\pd_\eta\lambda_{13}^+(0)=\k_1-\k_2<0\,,\quad
-\Im\pd_\eta\lambda_{13}^-(0)=\k_3-\k_2>0\,,
  \\
-\Im\pd_\eta\lambda_{12}^+(0)=\k_1-\k_3<0\,,\quad
-\Im\pd_\eta\lambda_{12}^-(0)=\k_2-\k_3<0\,,
\\
-\Im\pd_\eta\lambda_{23}^+(0)=\k_2-\k_1>0\,,\quad
-\Im\pd_\eta\lambda_{23}^-(0)=\k_3-\k_1>0\,,
\end{gather*}
the waves are bidirectional along $[1,3]$-soliton,
outgoing to $y=-\infty$ along $[1,2]$-soliton and incoming from $y=-\infty$
along $[2,3]$-soliton.
\end{remark}
\begin{remark}
\label{rem:cef}  
We might expect that continuous eigenmodes
$e^{t\lambda_{jk}^\pm(\eta)}g_{jk}^\pm(X,Y,\eta)$ approximate
modulations of $[j,k]$-soliton as is the case of
$1$-solitons (see \cite{Burtsev,Miz15}).
However, $g_{jk}^\pm$ does not provide a good approximation near the other
solitons.
For example, $g_{13}^\pm(\eta)$ diverges along the crest of $[1,2]$-soliton.
Indeed, if $\eta\ne0$ and $z_{12}=0$,
\begin{align*}
g_{13}^\pm(\eta)=
O\left(e^{(\k_3-\k_2)\{\sqrt{c_{13}}-\gamma_{13}(\pm\eta)\}Y}\right)\to\infty
\quad\text{as $Y\to-\infty$.}
\end{align*}
For this reason, we cannot define spectral projections associated with
these continuous eigenmodes.
\end{remark}
Finally, we will present our main results (see Lemma~\ref{lem:spec1},
Propositions~\ref{prop:noev}, \ref{prop:spec-res} and 
Theorem~\ref{thm:main} for the precise statements).
For the purupose, we will introduce several notations.
Let $\calX$ be a Hilbert space equipped with a norm
\begin{equation}
\label{eq:defcalX1}
  \|u\|_\calX= \biggl(\int_{\R^2_+} e^{2\a z_{13}}|u(X,Y)|^2\,dXdY
+ 
\int_{\R^2_-}e^{2\a z_{12}}|u(X,Y)|^2dXdY\,\biggr)^{1/2}\,,
\end{equation}
where $\R^2_\pm=\{(x,y)\mid x\in\R,\,\pm y\ge0\}$, and let
\begin{equation*}
  \la f, g\ra=
\lim_{R\to\infty}\int_{-R}^R\left(\int_\R f(X,Y)\overline{g(X,Y)}\,dX\right)\,dY\,.
\end{equation*}

Let
$\eta_I^{23}=\a(\k_1-\k_3)$ and
\begin{gather*}
\eta_{1j}^*(\a)=\sup\{\eta\ge0\mid
\lim_{X\to\infty}e^{-\a X}g_{1j}^{*,\pm}(X,Y,\eta)=0\}
\quad\text{for $j=2, 3$,}
\\
\eta_{*,-}(\a)=\sup\{\eta_R\ge0\mid
\lim_{X\to\infty}e^{-\a X}g_{23}^{*,-}(X,Y,\eta_R+i\eta_I^{23})=0\}\,,
\\
\eta_{\#,-}(\a)=\sup\{\eta_R\ge0\mid
\lim_{X\to\infty}e^{-\a X}\tg_{23}^{*,-}(X,Y,\eta_R+i\eta_I^{23})=0\}\,.
\end{gather*}
\par
Let $\lambda(\beta)=\prod_{j=1}^3(\k_j-\beta)$
and $\mathcal{D}_1\subset \C$ be a domain enclosed by four curves
\begin{gather*}
\{\lambda_{1j}^+(\eta_R)\mid -\eta_{1j}^*(\a)\le \eta_R\le \eta_{1j}^*(\a)\}
\quad\text{for $j=2, 3$,}
\\
\{\lambda(\k_1-\a+it)\mid
\sqrt{\a(\a+\k_2-\k_1)}\le \pm t\le \sqrt{\a(\a+\k_3-\k_1)}\}\,,
\end{gather*}
$\mathcal{D}_2\subset \C$ be a domain enclosed by two curves
\begin{gather*}
\{\lambda_{12}^-(\eta_R)\mid -\eta_{12}^*(\a) \le \eta_R\le \eta_{12}^*(\a)\}\,,
\\
\{\tilde{\lambda}_{23}^-(\eta_R+i\eta_I^{23})\mid
-\eta_{\#,-}(\a)\le \eta_R\le \eta_{\#,-}(\a)\}\,,
\end{gather*}
and $\mathcal{D}_3\subset \C$ be a domain enclosed by four curves
\begin{gather*}
\{\lambda_{13}^-(\eta_R)\mid -\eta_{13}^*(\a)\le \eta_R\le \eta_{13}^*(\a)\}\,,
\\
\{\lambda_{23}^-(\eta_R+i\eta_I^{23})\mid
-\eta_{*,-}(\a)\le  \eta_R \le \eta_{*,-}(\a)\}\,,
\\
\{-\lambda(\k_3+\a+it\mid \sqrt{\a(\a+\k_3-\k_1)}\le \pm t \le
\sqrt{\a(\a+2\k_3-\k_1-\k_2)}\}\,.
\end{gather*}
See Remarks~\ref{rem:boundary} and \ref{rem:adjointef}.

Now we introduce our result on spectral stability.
\begin{theorem}
    \label{thm:specstability}
Let $\a$ be a sufficiently small positive number. Let $\mL_0$ and $\mL_\mY$ be
closed operators on $\calX$. Then 
$\sigma(\mL_\mY)\subset \{\lambda\in\C\mid\Re\lambda\le0\}$ and
\begin{equation*}
\sigma(\mL)=\sigma(\mL_0)\cup
\overline{(\mathcal{D}_1\cup\mathcal{D}_2\cup\mathcal{D}_3)}\,.
  \end{equation*}
The operator $\mL_\mY$ has no eigenvalue and 
the set $\mathcal{D}_1\cup\mathcal{D}_2\cup\mathcal{D}_3$ consists
of the eigenvalues of $\mL_\mY^*$.
\end{theorem}
\begin{remark}
As $X\to\infty$, $g_{12}^{*,+}(\eta_{12}^*(\a))$ behaves like a plane wave solution of
$\mL_0^*u=\overline{\lambda_{12}^+(\pm\eta_{12}^*(\a))}u$
and $\lambda_{12}^+(\eta_{12}^*(\a))\in\sigma(\mL_0)$.
By the same reason, $\lambda_{1j}^\pm(\pm\eta_{12}^*(\a))$
$(j=2, 3)$, $\lambda_{23}(\pm\eta_*(\a)+i\eta_I^{23})$ and
$\tilde{\lambda}(\pm\eta_\#(\a)+i\eta_I^{23})$ belong to $\sigma(\mL_0)$.
\end{remark}
\begin{remark}
Let $\wmL$ be the restriction of $\mL_\mY$ to the subspace $\calX_\mY$
that is orthogonal to adjoint eigenfunctions associated with 
$\mathcal{D}_1$, $\mathcal{D}_2$ and $\mathcal{D}_3$.
We prove that $\sigma(\wmL)\setminus\{0\}\subset
\{\lambda\in\C\mid \Re\lambda\le -\delta\}$ for a $\delta>0$
(Lemma~\ref{lem:non-degeneracy} and Corollary~\ref{cor:non-degeneracy}).
But it remains open whether $\lambda=0$ is a regular point
or an essential singularity of $(\lambda-\wmL)^{-1}$.
This is because $\lambda=0$ is the intersection of four curves of
continuous eigenvalues that are associated with modulations of $[1,2]$-soliton
and $[1,3]$-soliton, and we cannot conclude $\operatorname{Range}(\wmL)=\calX_\mY$
by using the spectral stability results for oblique $1$-line solitons.
\end{remark}
\par
If we impose an extra weight as $Y\to\infty$, then at most two curves
of the continuous eigenvalues can intersect in a small neighborhood of
$\lambda=0$. At an intersection $\lambda$ of the two curves, 
the corresponding continuous eigenfunctions decay exponentially
either as $Y\to\infty$ or as $Y\to-\infty$,
allowing us to reduce the invertibility
of $\lambda-\wmL$ to that of $\lambda-\mL_{jk}$.
\begin{theorem}
\label{thm:linearstability}
Let $Y_\pm=\max(\pm Y,0)$.
Let $\a$ and $\nu_0$ be sufficiently small positive numbers.
If $\nu\in(0,\nu_0)$, there exist positive constants
$K$ and $b$ such that if $e^{\nu Y_+}f\in\calX$ and
\begin{gather*}
\la f, g_{12}^{*,+}(\cdot)\ra=\la f, g_{12}^{*,-}(\cdot)\ra
=\la f,g_{13}^{*,-}(\cdot+i\nu)\ra=0
\text{ on $[-\eta_0,\eta_0]$}
\end{gather*}
for an $\eta_0>0$, then for every $t\ge0$,
\begin{equation*}
\|e^{\nu Y_+}e^{t\mL_\mY}f\|_\calX\le Ke^{-bt}\|e^{\nu Y_+}f\|_\calX\,.
\end{equation*}
\end{theorem}
\bigskip

\section{Spectral  stability of  $[j,k]$-solitons}
\label{sec:[i,j]-stability}
Let us recall spectral stability results of $u_{jk}$
in an exponentially weighted space $\calX_{jk}:=L^2(\R^2;e^{2\a z_{jk}}dXdY)$
whose weight function increases on the right of $[j,k]$-soliton.
Let
\begin{gather*}
\mL_0=-\frac{1}{4}\pd_X^3+b_1\pd_X+b_2\pd_Y-\frac{3}{4}\pd_X^{-1}\pd_Y^2\,,
\quad \mL_{jk}=\mL_0-\frac{3}{2}\pd_X(u_{jk}\cdot)\,,
\end{gather*}
\begin{multline*}
  \left(\mL_{jk}(\eta)\right)f(z_{jk})= e^{-iy\eta}\mL_{jk}e^{iy\eta}f(z_{jk})
\\=                                         
\frac{1}{4}\left\{-\pd_{z_{jk}}(\pd_{z_{jk}}^2-4c_{jk})+4i(b_2-3a_{jk})\eta
+3\eta^2\pd_{z_{jk}}^{-1}\right\}
-\frac{3}{2}\pd_{z_{jk}}(u_{jk}\cdot)\,,
\end{multline*}
\begin{gather}
\label{eq:def-gij0}
g_{jk}^0(z,\eta)=\frac{\sqrt{c_{jk}}}{2\gamma_{jk}(\eta)}
\pd_z^2\left\{e^{-\gamma_{jk}(\eta)z}\sech\sqrt{c_{jk}}z\right\}\,
  \\
\label{eq:def-gij*0}
g^{0,*}_{jk}(z,\eta)=\frac12\pd_z\left\{
e^{\gamma_{jk}(-\bar{\eta})z}\sech\sqrt{c_{jk}}z\right\}\,.
\end{gather}
Then we have the following.
\begin{lemma}
  \label{lem:evmL0}
\begin{gather}
\label{eq:evmL0-1}
\left\{
  \begin{aligned}
&  \mL_{jk}(\eta)g_{jk}^0(z_{jk},\pm\eta)
=\lambda_{jk}^\pm(\eta)g_{jk}^0(z_{jk},\pm\eta)\,,
  \\
&  \mL_{jk}^*g^{0,*}_{jk}(z_{jk},\pm\eta)
  =\overline{\lambda_{jk}^\pm(\eta)}g^{0,*}_{jk}(z_{jk},\pm\eta)\,,    
  \end{aligned}\right.
\\
\label{eq:evmL0-1'}
\left\{
  \begin{aligned}
&\mL_{jk}(\eta)g_{jk}^0(-z_{jk},\pm\eta)
=\tilde{\lambda}_{jk}^\pm(\eta)g_{jk}^0(-z_{jk},\pm\eta)\,,
\\ &
  \mL_{jk}^*g^{0,*}_{jk}(-z_{jk},\pm\eta)
  =\overline{\tilde{\lambda}_{jk}^\pm(\eta)}g^{0,*}_{jk}(-z_{jk},\pm\eta)\,.
\end{aligned}\right.
\end{gather}
For any $\eta\in\R$,
\begin{equation}
\label{eq:g-g*-int}
\int_\R g_{jk}^0(z,\eta)\overline{g_{jk}^{0,*}(z,\eta)}\,dz=i\eta\,,
\quad
\int_\R g_{jk}^0(z,\eta)\overline{g_{jk}^{0,*}(z,-\eta)}\,dz=0\,.
\end{equation}
\end{lemma}
\begin{proof}
Let $\varphi_c(z)=2c\sech^2\sqrt{c}z$, $\gamma(\eta)=\sqrt{c+i\eta}$ and
\begin{gather*}
\mL_c=
-\frac14\pd_z\left(\pd_z^2-c+6\varphi_c\right)-\frac34\pd_z^{-1}\pd_y^2\,,
\\
g^0(z,c,\eta)=\pd_z^2\left(e^{-\gamma(\eta)z}\sech \sqrt{c}z\right)\,,
\quad g^{0,*}(z,c,\eta)=
\pd_z\left(e^{\gamma(-\bar{\eta})z}\sech \sqrt{c}z\right)\,.
\end{gather*}
Then it follows from \cite[Lemma~2.1]{Miz15} that
\begin{gather*}
\mL_cg^0(\pm z,c,\eta)=\pm i\eta\gamma(\eta)g^0(\pm z,c,\eta)\,,
  \quad
\mL_c^*g^{0,*}(\pm z,c,\eta)=\mp i\bar{\eta}\gamma(-\bar{\eta})  g^{0,*}(\pm z,c,\eta)\,.
\end{gather*}
Since $u_{jk}=\varphi_{c_{jk}}(z_{jk})$,
$b_1+2a_{jk}b_2-3a_{jk}^2=c_{jk}$
and $\mL_{jk}=\mL_{c_{jk}}+(b_2-3a_{jk})\pd_y$,
we have \eqref{eq:evmL0-1} and \eqref{eq:evmL0-1'}.
Eq.~\eqref{eq:g-g*-int} follows from \cite[Eq.~(2.6)]{Miz15}.
\end{proof}
\begin{remark}
\label{rem:gij0}
By \eqref{eq:def-gij0}, \eqref{eq:def-gij*0}
and the fact that $\varphi_c(x)=2c\sech^2\sqrt{c}x$,
\begin{gather*}
g_{jk}^0(z,0)=-\frac{1}{4\sqrt{c_{jk}}}\pd_z\varphi_{c_{jk}}(z)\,,
\quad
\pd_\eta g_{jk}^0(z,0)=-\frac{i}{4}\pd_c\varphi_{c_{jk}}(z)
+\frac{i}{8c_{jk}^{3/2}}\pd_z\varphi_{c_{jk}}(z)\,,
\\
g_{jk}^{0,*}(z,0)=\frac{1}{4\sqrt{c_{jk}}}\varphi_{c_{jk}}(z)\,,
\quad
\pd_{\bar{\eta}}g_{jk}^{0,*}(z,0)=-\frac{i}{4}
\int_{-\infty}^z \pd_c\varphi_{c_{jk}}(z')\,dz'\,.
\end{gather*}
\end{remark}
Let $\eta_{jk}^*(\a)=2(\sqrt{c_{jk}}+\a)\sqrt{\a(\a+2\sqrt{c_{jk}})}$.
Then $\Re\gamma_{jk}(\eta_{jk}^*(\a))=\sqrt{c_{jk}}+\a$ and
$e^{\alpha z}g_{jk}^0(z,\eta)$ is bounded if
$0<\a<2\sqrt{c_{jk}}$ and $\eta\in(-\eta_{jk}^*(\a),\eta_{jk}^*(\a))$.
For $\eta_0\in(0,\eta_{jk}^*(\a))$, we define a spectral
projection
$$P_{jk}(\eta_0):\calX_{jk}\mapsto
\spann\{g_{jk}^0(z_{jk},\pm\eta)e^{iy\eta}\mid \eta\in[-\eta_0,\eta_0]\}$$
associated with $\mL_{jk}$ as in \cite[Lemma~3.1]{Miz15}.
Let
\begin{equation}
  \label{eq:7}
\left\{
\begin{aligned}
& g_{jk,1}^0(z,\eta)=g_{jk}^0(z,\eta)+g_{jk}^0(z,-\eta)\,,
\enskip
g_{jk,2}^0(z,\eta)=\frac{g_{jk}^0(z,\eta)-g_{jk}^0(z,-\eta)}{i\eta}\,,
\\ &
g_{jk,1}^{0,*}(z,\eta)=\frac{g_{jk}^{0,*}(z,\eta)-g_{jk}^{0,*}(z,-\eta)}
{-2i\bar{\eta}}\,,
\enskip
g_{jk,2}^{0,*}(z,\eta)=\frac{g_{jk}^{0,*}(z,\eta)+g_{jk}^{0,*}(z,-\eta)}{2}\,,
\end{aligned}
\right.  
\end{equation}
\begin{gather*}
P_{jk}(\eta_0)f(x,y)=\frac{1}{2\pi}\sum_{\ell=1,2}
\int_{-\eta_0}^{\eta_0} a_\ell(\eta)g_{jk,\ell}^0(z_{jk},\eta)e^{iy\eta}\,d\eta\,,
\\
a_\ell(\eta)=\int_{\R^2}f(x,y)\overline{g^{0,*}_{jk,\ell}(z_{jk,\ell},\eta)}
e^{-iy\eta}\,
dxdy\,.
\end{gather*}
\begin{proposition}[Proposition~6.3 and Remark~6.1 in \cite{Miz23}]
  \label{lem:wmL1-decay}
  Let $1\le j<k\le3$ and $\a\in(0,2\sqrt{c_{jk}})$.
  Let $\mL_0$ and $\mL_{jk}$ be closed operators on $\calX_{jk}$. Then
  \begin{gather*}
\sigma(\mL_{jk})=\sigma(\mL_0)\cup\{\lambda_{jk}^\pm(\eta)\mid
  \eta\in[-\eta_{jk}^*(\a),\eta_{jk}^*(\a)]\}\,,
\\
\sigma(\mL_0)=\overline{
\{ip(\xi+i\a,\eta+2i\a a_{jk})\mid (\xi,\eta)\in\R^2\}}\,,    
  \end{gather*}
where
$p(\xi,\eta)=\frac{1}{4}\left(\xi^3-3\eta^2/\xi +4b_1\xi+4b_2\eta\right)$.
\par
 Moreover,
 $\lambda_{jk}^\pm(\eta)-\mL_{jk}$ is invertible on
 $\left(I-P_{jk}(\eta_0)\right)\calX_{jk}$ if $\eta_0\in(0,\eta_{jk}^*(\a))$
and $\eta\in(-\eta_0,\eta_0)$.
\end{proposition}
\par
Next, we will investigate spectral stability of $\mL_{23}$ in
$\calX_{12}$. Since modulations along $[2,3]$-soliton are
incoming (Remark~\ref{rem:wavemotion}),
we can expect exponential stabity of $\mL_{23}$ on $\calX_{12}$.

\begin{lemma}
  \label{lem:mL-12inX1}
Assume that $\mL_{23}$ is a closed operator on $\calX_{12}$ and that
\begin{equation}
  \label{ass:alpha23}
0<\a<\min\left\{\frac{c_{23}}{\k_3-\k_1},\,\sqrt{c_{12}}\right\}\,.
\end{equation}
Then there exists a positive number $\delta_1$ such that
\begin{equation*}
\sup_{\Re\lambda\ge -\delta_1\a}\|(\lambda-\mL_{23})^{-1}\|_{B(\calX_{12})}<\infty\,.
\end{equation*}
\end{lemma}
We remark that Lemma~\ref{lem:mL-12inX1} is the same  with \cite[Lemma~7.19]{Miz23} if
we let $\k_2=\k_3$ in \cite[Lemma~7.19]{Miz23}.
\par
To begin with, we investigate continuous eigenfunctions which decay exponentially
as $z_{12}\to\infty$.
By the definitions,
\begin{equation}
  \label{eq:X12-decomp}
  \calX_{12}= e^{-y\eta_I^{23}}\calX_{23}=\bigoplus_{\eta=\eta_R+i\eta_I^{23}}
  e^{iy\eta}L^2(\R;e^{2\a z_{23}}dz_{23})\,.
\end{equation}
Let $\eta=\eta_R+i\eta_I^{23}$ and $\eta_R\in\R$. By
Lemma~\ref{lem:evmL0} and \eqref{eq:X12-decomp},
$e^{iy\eta}g_{23}^0(z_{23},\pm\eta)$ and
$e^{iy\eta}g_{23}^0(-z_{23},\pm\eta)$
are continuous eigenfunctions if
$g_{23}^0(z_{23},\pm\eta)$ and $g_{23}^0(-z_{23},\pm\eta)$
belong to $L^2(\R;e^{2\a z_{23}}dX)$, respectively.

Let $\a$ be a constant satisfying \eqref{ass:alpha23} and
\begin{gather*}
    \label{def:e*-'}
\eta_{*,-}(\a)=2(\sqrt{c_{23}}+\a)\sqrt{\a(\a+2\k_3-2a_{12})}\,,
\\ \label{def:esh-'}
  \eta_{\#,-}(\a)=2(\sqrt{c_{23}}-\a)\sqrt{\a(\a+\k_2-\k_1)}  \,.
\end{gather*}
Then $c_{23}+\eta_I>0$,
$\Re\gamma_{23}(\eta_{*,-}(\a)-i\eta_I^{23})=\sqrt{c_{23}}+\a$,
$\Re\gamma_{23}(\eta_{\#,-}(\a)-i\eta_I^{23})=\sqrt{c_{23}}-\a>0$
and
\begin{gather*}
\Re\gamma_{23}(-\eta)<\sqrt{c_{23}}+\a \quad
\text{if $\eta=\eta_R+i\eta_I^{23}$ and $|\eta_R|<\eta_{*,-}(\a)$,}
\\
\Re\gamma_{23}(-\eta)<\sqrt{c_{23}}-\a\quad
\text{if $\eta=\eta_R+i\eta_I^{23}$ and $|\eta_R|<\eta_{\#,-}(\a)$.}
  \end{gather*}
Thus for $\eta=\eta_R+i\eta_I^{23}$,
$e^{iy\eta}g_{23}^0(z,-\eta)$ is a continuous eigenfunction of
$\mL_{23}$ if $|\eta_R|<\eta_{*,-}(\a)$ and
$e^{iy\eta}g_{23}^0(-z,-\eta)$ is a continuous eigenfunction of
$\mL_{23}$ if $|\eta_R|<\eta_{\#,-}(\a)$.
\par
On the other hand,  $e^{iy\eta}g_{23}^0(\pm z,\eta)$
with $\eta=\eta_R+i\eta_I^{23}$ is not a continuous eigenfunction of $\mL_{23}$
since $\Re\gamma_{23}(\eta)\ge \gamma_{23}(i\eta_I^{23})>\sqrt{c_{23}}+\a$
and $e^{\a z_{23}}g_{23}^0(\pm z_{23},\eta)\not\in L^2(\R)$.
\par
Using Darboux transformations
$$\mM_\pm(\eta):=-\pd_{z_{23}}+i\eta\pd_{z_{23}}^{-1}
-2\sqrt{c_{23}}\tanh\sqrt{c_{23}}z_{23}\,,$$
we can prove Lemma~\ref{lem:specL23} below in the same way as
\cite[Propositions~6.3 and 6.6]{Miz23}.
\begin{lemma}
  \label{lem:specL23}
Let $\mL_0$ and $\mL_{23}$ be closed operators on $\calX_{12}$.
Assume \eqref{ass:alpha23}. Then
  \begin{multline*}
\sigma(\mL_{23})=\sigma(\mL_0)\cup \{\lambda_{23}^-(\eta_R+i\eta_I^{23})
    \mid |\eta_R|\le \eta_{*,-}(\a)\}
    \\  \cup\{\tilde{\lambda}_{23}^-(\eta_R+i\eta_I^{23})
    \mid |\eta_R|\le \eta_{\#,-}(\a)\}\,.
  \end{multline*}
\end{lemma}

We can prove that $\{\lambda_{23}^-(\eta_R+i\eta_I^{23})\}$ and
$\{\tilde{\lambda}_{23}^-(\eta_R+i\eta_I^{23})\}$
lie in the stable half plane in the same way as \cite[Lemma~7.20]{Miz23}.
\begin{lemma}
  \label{lem:lambda2-'}
Assume \eqref{ass:alpha23}.
Let $\eta=\eta_R+i\eta_I^{23}$ and $\eta_R\in\R$. Then
\begin{gather*}
  \eta_R\pd_{\eta_R}\Re \lambda_{23}^-(\eta)
=-\eta_R\pd_{\eta_R}\Re \tilde{\lambda}_{23}^-(\eta) <0
\quad\text{for any $\eta_R\in\R\setminus\{0\}$,}
\\
\Re\lambda_{23}^-(\eta_R+i\eta_I^{23})\le  \lambda_{23}^-(i\eta_I^{23})
\le -C\a
\quad\text{for any $\eta_R\in\R$,}
\\
\Re\tilde{\lambda}_{23}^-(\eta_R+i\eta_I^{23})\le
\Re \tilde{\lambda}_{23}^-(\eta_{\#,-}(\a)+i\eta_I^{23})
\le -C\a\quad\text{if $|\eta_R|\le\eta_{\#,-}(\a)$,}
\end{gather*}
where $C$ is a positive constant that does not depend on $\a$.
\end{lemma}
\begin{proof}[Proof of Lemma~\ref{lem:mL-12inX1}]
As in \cite[Lemma~6.1]{Miz23}, we have
$\sup\{\Re\lambda\mid \lambda\in\sigma(\mL_0)\}=\a(c_{12}-\a^2/4)<0$.
Combining this fact with Lemmas~\ref{lem:specL23} and \ref{lem:lambda2-'},
we have Lemma~\ref{lem:mL-12inX1}.
\end{proof}

The two curves of continuous eigenvalues
$\lambda_{13}^-(\eta)$ and $\lambda_{13}^-(\eta)$
meet at $\lambda=0$. To separate these two curves from $\lambda=0$,
we introduce
\begin{equation*}
\calX_{13}(\nu):=L^2(\R^2;e^{2(\a z_{13}+\nu y)}dXdY)\,.
  \end{equation*}
Let $\eta_{13}^*(\a,\nu)=2(\sqrt{c_{13}}+\a)\sqrt{\a(\a+2\sqrt{c_{13}})+\nu}$.
Then we have the following.
\begin{lemma}
  \label{lem:X13nu}
Let $\mL_{13}$ and $\mL_0$ be closed operators on $\calX_{13}(\nu)$.
Assume
\begin{equation}
  \label{ass:X13nu}
0\le\nu<\min\{\sqrt{c_{12}}\a\,,\, c_{13}-(|\k_2-a_{13}|+\a)^2\,,\,
(\k_3-\beta_{*,-})(\beta_{*,-}-\k_1)\}\,.
\end{equation}
Then
 \begin{multline*}
\sigma(\mL_{13})=\sigma(\mL_0)\cup \{\lambda_{13}^+(\eta_R+i\nu)
    \mid |\eta_R|\le \eta_{13}^*(\a,\nu)\}
    \\  \cup\{\lambda_{13}^-(\eta_R+i\nu)
    \mid |\eta_R|\le \eta_{13}^*(\a,-\nu)\}\,,
  \end{multline*}
  \begin{equation*}
\sigma(\mL_0)=\overline{
\left\{ip\left(\xi+i\a,\eta+i(2\a a_{13}+\nu)\right)
\mid (\xi,\eta)\in\R^2\right\}}\,.
  \end{equation*}
\end{lemma}
\begin{proof}
If $\eta=\eta_R+i\nu$ and $|\eta_R|<\eta_{13}^*(\a,\pm\nu)$,
then $|\a-\Re\gamma_{13}(\pm\eta)|<\sqrt{c_{13}}$ and
$e^{\a x}g_{13}(x,\pm\eta)\in L^2(\R)$.
On the other hand, we have $e^{\a x}g_{13}(-x,\pm\eta)\not\in L^2(\R)$ because
$\Re \gamma_{13}(\pm\eta)\ge \sqrt{c_{13}-\nu}> \sqrt{c_{13}}-\a$
by \eqref{ass:X13nu}.
Thus we can prove  Lemma~\ref{lem:X13nu} in the same way as
\cite[Lemma~6.1, Propositions~6.3 and 6.6]{Miz23}.
  \end{proof}

Following the proof of \cite[Lemma~7.20]{Miz23},
we can prove that $\{\lambda_{13}^+(\eta_R+i\nu)\}$ is
in the stable half plane but 
$\lambda_{13}^-(\eta_R+i\nu)$ is in the unstable half plane
if $\eta_R$ is close to $0$.
\begin{lemma}
  \label{lem:lambda13'}
Assume \eqref{ass:X13nu} and that $\nu>0$.
Then for $\eta=\eta_R+i\nu $ with $\eta_R\in\R$,
\begin{gather*}
 \eta_R\pd_{\eta_R}\Re \lambda_{13}^\pm(\eta)<0
\quad\text{for any $\eta_R\in\R\setminus\{0\}$,}
\\
\Re\lambda_{13}^+(\eta)\le  \lambda_{13}^+(i\nu)<0\,,
\\
\lambda_{13}^-(i\nu)>0\,,\quad
\Re \lambda_{13}^-(\pm\eta_{13}^*(\a,-\nu)+i\nu)<0\,.
\end{gather*}
Moreover, $|\lambda_{13}^\pm(\eta_R+i\nu)|$ are even and monotonically increasing for $\eta_R\ge0$.
\end{lemma}
\begin{proof}
By \eqref{eq:b1-b2Y} and \eqref{def:lambdajk},
\begin{gather*}
\lambda_{13}^+(i\nu)=-\nu(\k_2-\k_1+\sqrt{c_{13}-\nu}-\sqrt{c_{13}})\,,
\\
\lambda_{13}^-(i\nu)=\nu(\k_3-\k_2+\sqrt{c_{13}+\nu}-\sqrt{c_{13}})\,.
\end{gather*}
Since $\k_1<\k_2<\k_3$ and
$\sqrt{c_{13}}-\sqrt{c_{13}-\nu}<\sqrt{c_{13}}-|\k_2-a_{13}|\le \k_2-\k_1$,
we have
$\lambda_{13}^+(i\nu)<0$ and $\lambda_{13}^-(i\nu)>0$.
\par
Let
\begin{equation*}
\gamma_R^\nu(\eta_R)=\Re\sqrt{c_{13}-\nu+i\eta_R}\,,\quad
\gamma_I^\nu(\eta_R)=\Im\sqrt{c_{13}-\nu+i\eta_R}\,.
\end{equation*}
We have $\gamma_R^{\pm\nu}(\eta_R)>0$ for $\eta_R\in\R$
and $\eta_R\gamma_I^{\pm\nu}(\eta_R)>0$ for $\eta_R\in\R\setminus\{0\}$. Since
\begin{gather*}
\pd_{\eta_R}\gamma_R^\nu=
\frac{\gamma_I^\nu}{2\sqrt{(c_{13}-\nu)^2+\eta_R^2}}\,,
\quad
\pd_{\eta_R}\gamma_I^\nu=
\frac{\gamma_R^\nu}{2\sqrt{(c_{13}-\nu)^2+\eta_R^2}}>0\,,
\\
\pd_{\eta_R}\Re\lambda_{13}^\pm(\eta)=
\mp\left(1\pm\frac{\nu}{2\sqrt{(c_{13}\mp\nu)^2+\eta_R^2}}\right)\gamma_I^{\pm\nu}(\pm\eta_R)
-\eta_R\pd_{\eta_R}\gamma_I^{\pm\nu}(\pm\eta_R)\,,
\end{gather*}
we have $\eta_R\pd_{\eta_R}\Re \lambda_{13}^\pm(\eta)<0$ for $\eta_R\in\R\setminus\{0\}$.
\par
Since 
$\gamma_R^{-\nu}(\eta_{13}^*(\a,-\nu))
=\sqrt{c_{13}}+\a$ and
$\eta_{13}^*(\a,-\nu)=
2(\sqrt{c_{13}}+\a)\gamma_I^{-\nu}(\eta_{13}^*(\a,-\nu))$,
\begin{multline*}
\Re\lambda_{13}^-(\pm\eta_{13}^*(\a,-\nu)+i\nu)  
= \nu(\k_3-\k_2+\a)
-2(\sqrt{c_{13}}+\a)
\left(\a^2+2\a\sqrt{c_{13}}-\nu\right)
\\=
-2\a^3-3\a(2\sqrt{c_{13}}\a-\nu)-4\sqrt{c_{13}}
\left(\sqrt{c_{13}}\a
-\frac{\k_3-a_{12}}{\k_3-\k_1}\nu\right)<0\,.
\end{multline*}

In view of \eqref{def:lambdajk}, it is clear that $|\lambda_{13}^\pm(\eta_R+i\nu)|$ are even.
Moreover, $|\lambda_{13}^\pm(\eta_R+i\nu)|$ are monotonically increasing for $\eta_R\ge0$
because $\pm(b_2-3a_{13}\pm\gamma_R^{\pm \nu}(0))>0$, and
$\Re\gamma_R^{\pm\nu}(\eta_R)-\gamma_R^{\pm\nu}(0)$ and $\gamma_I^{\pm\nu}(\eta_R)$
are non-negative and monotonically increasing for $\eta_R\ge0$.
Thus we complete the proof.
\end{proof}

\bigskip

\section{Spectral stability of $u_\mY$}
\label{sec:spectral}
Let $\calX^\a(\nu)$ be a Hilbert space with an inner product
\begin{multline}
  \label{eq:defcalX}
  (u,v)_{\calX^\a(\nu)}= \biggl(\int_{\R^2_+} e^{2\a(X+2a_{13}Y)+2\nu Y}u(X,Y)\overline{v(X,Y)}\,dXdY
\\ + 
\int_{\R^2_-}e^{2\a(X+2a_{12}Y)}u(X,Y)\overline{v(X,Y)}\,dXdY\biggr)^{1/2}\,,
\end{multline}
and $\| u\|_{\calX^\a(\nu)}=(u,u)_{\calX^\a(\nu)}^{1/2}$. 
Let $\calX^\a=\calX^\a(0)$.
Note that $\calX^\a(\nu)=\calX_{12}\cap\calX_{13}(\nu)$, that
and $\calX^{-\a}(-\nu)=(\calX^\a(\nu))^*=\calX_{12}^*+\calX_{13}(\nu)^*$
provided $\a>0$ and $\nu\ge \a(\k_2-\k_3)$.
We have $\calX(\nu_2)\subset \calX(\nu_1)$ if $\nu_2\ge\nu_1\ge0$.
We will omit the superscript $\a$ if there is no confusion.

\par
Let $\mL_{\mY}$ be a closed operator on $\calX(\nu)$.
In this section, we will determine the continuous spectrum of $\mL_\mY$,
find eigenvalues of  $\mL_{\mY}^*$ and
prove that the $f$ is orthogonal to the eigenspaces of $\mL_\mY^*$
if and only if $f$ is orthogonal to the continuous eigenspaces of $\mL_\mY^*$
corresponding to $\sigma(\mL_{jk})\setminus\sigma(\mL_0)$ $(1\le j<k\le3)$.
\par

To begin with,
we will show that $\lambda-\mL_0$ is invertible on $\calX(\nu)$ if
$\lambda$ is in the right half plane of $\C$.
\begin{lemma}
  \label{lem:intersection}
Let $\a>0$, $(\k_2-\k_3)\a\le \nu\le (\k_2-\k_1)\a$ and
$0<\delta_0<\min\{c_{12},c_{23}\}-\frac{\a^2}{4}$. Let
$\mL_0$ be a closed operator on $\calX$ and
\begin{gather*}
D_0=\overline{\left\{ip(\xi+i\a,\eta+2i\a a)\mid
(\xi,\eta)\in\R^2\,,\; a_{12}\le a\le a_{13}+\nu/(2\a)\right\}}\,.
\end{gather*}
Then
$\sigma(\mL_0)=D_0\subset \{\lambda\in\C\mid \Re\lambda< -\delta_0\a\}$.
There exists a $C>0$ such that for every $f\in \calX$ and 
$\lambda\in\C$ satisfying $\Re \lambda\ge -\delta_0\a$,
\begin{equation*}
\|(\lambda-\mL_0)^{-1}f\|_{\calX(\nu)} \le C\|f\|_{\calX(\nu)}\,.
  \end{equation*}
\end{lemma}
\begin{proof}[Proof of Lemma~\ref{lem:intersection}]
We may assume that $f\in C_0^\infty(\R^2)$ since
$C_0^\infty(\R^2)$ is dense in $\calX$.
Let $\mL_{0,\a}(a)=e^{\a(x+2ay)}\mL_0e^{-\a(x+2ay)}$.
We have $\widehat{\mL_{0,\a}(a)}=ip(\xi+i\a,\eta+2ia\a)$
and
\begin{align}
\label{eq:Imp}
  \Im p(\xi+i\a,\eta+2ia\a)=&\frac{\a}{4}
  \left\{3\xi^2-\a^2+3\frac{(\eta-2a\xi)^2}{\xi^2+\a^2}+
4(b_1+2b_2a-3a^2)\right\}
 \\ \notag
\ge & \a(b_1+2b_2a-3a^2)-\frac{\a^3}{4}\,.
\end{align}
Since $b_1+2a_{jk}b_2-3a_{jk}^2=c_{jk}$ and $a_{12}<b_2/3<a_{23}$,
\begin{equation}
 \label{eq:Imp-bound}
 \Im p(\xi+i\a,\eta+2ia\a)\ge
\a\left(\min\{c_{12},c_{23}\}-\frac{\a^2}{4}\right)
>\delta_0\a\quad
\text{for $a\in[a_{12},a_{23}]$.}
\end{equation}
Since $a_{12}\le a_{13}+\nu/(2\a)\le a_{23}$, we have
$\lambda\not\in D_0$ provided $\Re\lambda\ge -\delta_0\a$.
\par
Let $a_{12}\le a_1\le a\le a_2\le a_{23}$, $F_a(x,y)=e^{\a(x+2ay)}f(x,y)$
and let $\lambda\not\in D_0$.
If $u$ is a solution of $(\lambda-\mL_0)u=f$ in $e^{-\a(x+2ay)} L^2(\R^2)$,
\begin{equation}
  \label{eq:2}
\left(\lambda-\mL_{0,\a}(a)\right)U_a=F_a\,,  
\end{equation}
where $U_a=e^{\a(x+2ay)}u$ and
\begin{equation*}
  \widehat{U}_a=
  \frac{\widehat{F_a}(\xi,\eta)}{\lambda-ip(\xi+i\a,\eta+2ia\a)}\,,
  \quad \widehat{F_a}=\hat{f}(\xi+i\a,\eta+2ia\a)\,.
\end{equation*}
By \eqref{eq:Imp-bound}, $(\lambda-ip(\xi+i\a,\eta+2ia\a))^{-1}$
is analytic and bounded for $\eta\in\C$ with
$2\a(a_1-a)\le \Im\eta\le 2\a(a_2-a)$.
Hence it follows from \cite[Theorem~IX.13]{RS2} 
that $U_a(x,y)=e^{2\a(a-a_1)y}U_{a_1}(x,y)\in L^2(\R^2)$ for $a\in[a_1,a_2]$.
By the Plancherel theorem, 
\begin{multline*}
\|e^{\a(x+2ay)}u\|_{L^2(\R^2)}= \|\widehat{U_a}\|_{L^2(\R^2)}
\\ \le                                    
 \sup_{(\xi,\eta)\in\R^2}\frac{1}{|\lambda-i p(\xi+i\a,\eta+2i\a a)|}
  \|\widehat{F_a}\|_{L^2(\R^2)}\,.
\end{multline*}
Using the above for $a=a_{12}=:a_1$ and $a=a_{13}+\nu/(2\a)=:a_2$, we have
\begin{align*}
\|u\|_{\calX(\nu)}\le & \|U_{a_1}\|_{L^2(\R^2)}+\|U_{a_2}\|_{L^2(\R^2)}
\\ \le & C\a^{-1}(\|F_{a_1}\|_{L^2(\R^2)}+\|F_{a_2}\|_{L^2(\R^2)})
\le 2C\a^{-1}\|F\|_{\calX(\nu)}\,,
\end{align*}
where $C$ is a constant depending only of $\delta_0$.
Thus we prove $\sigma(\mL_0)\subset D_0$.
\par
Finally, we will prove that $D_0=\sigma(\mL_0)$.
For the purpose, it suffices to show that
$\bar{\lambda}-\mL_0^*$ is not invertible on
$\calX^\a(\nu)^*=\calX_{13}(\nu)^*+\calX_{12}^*$ for $\lambda\in\operatorname{int}(D_0)$.
\par
Suppose that $\lambda\in\operatorname{int}(D_0)$.
Then there exist $(\xi,\eta)\in R^2$ and $a\in[a_{12},a_{13}+\nu/(2\a)]$
such that $\lambda=ip(\xi+i\a,\eta+2i\a a)$.
Let  $\chi_0(y)$  be a smooth non-negative function such that
\begin{equation}
  \label{eq:chi0}
\chi_0(y)=
\begin{cases}
  & 1\quad\text{if $ y\le 1$,}
\\ & 0\quad\text{if $y\ge2$.}
\end{cases}
\end{equation}
Let $\chi_{0,R}(x,y)=\chi_0\left(\sqrt{x^2+y^2}/(16R)\right)$ and
$u_0^R=\chi_{0,R}(x,y)e^{ix(\xi-i\a)+iy(\eta-2i\a a)}$.
Since $x+2ay\le \max\{z_{12},z_{13}+\nu Y\}$ for $a\in[a_{12},a_{13}+\nu/(2\a)]$,
\begin{equation}
\label{eq:9}
\|u_0^R\|_{\calX(\nu)^*}=O(R)\,.
\end{equation}
On the other hand,
    \begin{equation}
      \label{eq:10}
\|(\bar{\lambda}-\mL_0^*)u_0^R\|_{\calX(\nu)^*}=O(1)      
    \end{equation}
because $(\bar{\lambda}-\mL_0^*)e^{ix(\xi-i\a)+iy(\eta-2i\a a)}=0$ and
$[\mL_0,\chi_{0,R}]=O(1/R)$ by \cite[(C.6)]{Miz23}.
Combining \eqref{eq:9} and \eqref{eq:10}, we have $\bar{\lambda}\in \sigma(\mL_0)$.
Thus we complete the proof.
\end{proof}
\par

Using \eqref{eq:u1-asymptotics}, 
we will show that continuous eigenvalues for $\mL_{13}$ in
$\calX_{13}(\nu)$ and continuous eigenvalues for $\mL_{12}$ and
$\mL_{23}$ in $\calX_{12}$ belong to $\sigma(\mL_{\mY})$.
\begin{lemma}
  \label{lem:spec-pre}
Assume \eqref{ass:alpha23} and \eqref{ass:X13nu}.
Let $\mL_0$ and $\mL_{\mY}$ be closed operators on $\calX(\nu)$,
$\mL_{13}$ be a closed operator on $\calX_{13}(\nu)$ and
$\mL_{12}$ and $\mL_{23}$ be closed operators on $\calX_{12}$.
Then $\sigma(\mL_0)\cup\sigma(\mL_{13})\cup\sigma(\mL_{12})\cup\sigma(\mL_{23})\subset
\sigma(\mL_{\mY})$.
\end{lemma}

\begin{proof}[Proof of Lemma~\ref{lem:spec-pre}]
First, we will prove that $\sigma(\mL_0)\subset \sigma(\mL)$.
Let $a_{12}<a<a_{13}+\nu/(2\a)$, $\lambda=ip(\xi+i\a,\eta+2ia \a)$ and
\begin{equation*}
k=\frac12\left(\xi-i\a-\frac{i\eta+2a\a}{\xi-i\a}\right)\,,
\quad
\ell=\frac12\left(-\xi+i\a-\frac{i\eta+2a\a}{\xi-i\a}\right)\,.
\end{equation*}
Then $k-\ell=\xi-i\a$, $\ell^2-k^2=i\eta+2a\a$, $k^3-\ell^3=i\bar{\lambda}$
and it follows from Lemma~\ref{lem:prodJdJ} that
\begin{equation*}
\mL_\mY^*\left(\Phi^1(k)\Phi^{1,*}(\ell)\right)=\bar{\lambda}
\left(\Phi^1(k)\Phi^{1,*}(\ell)\right)\,.
\end{equation*}
Since $u^R=\chi_{0,R}(x,y)\Phi^1(k)\Phi^{1,*}(\ell)$ satisfies
\begin{equation*}
\|u^R\|_{\calX(\nu)^*}=O(R)\quad\text{and}\quad \|(\bar{\lambda}-\mL_\mY^*)u^R\|_{\calX(\nu)^*}=O(1)\,,
\end{equation*}
we have $\lambda\in \sigma(\mL_\mY)$. 
Thus we prove $\sigma(\mL_0)\subset \sigma(\mL_\mY)$.
\par

By the definitions,
\begin{align*}
u_{\mY}-u_{13}=&
2\pd_x^2\log\left(1+\frac{e^{\theta_2}}{e^{\theta_1}+e^{\theta_3}}\right)             
= O\left(\min\{e^{\theta_2-\theta_1},e^{\theta_2-\theta_3}\}\right)\,.
\end{align*}
Since
\begin{equation}
  \label{eq:2/tau}
\left\{  
\begin{aligned}
& \theta_2-\theta_1=2\sqrt{c_{12}}z_{13}-(\k_2-\k_1)(\k_3-\k_2)Y\,,
\\ &
\theta_3-\theta_2=2\sqrt{c_{23}}z_{13}+(\k_2-\k_1)(\k_3-\k_2)Y\,,
\end{aligned}\right.
\end{equation}
\begin{equation}
  \label{eq:Y-13}
  |u_{\mY}-u_{13}|\lesssim \left\{
    \begin{aligned}
& e^{-2\sqrt{c_{23}}z_{13}-(\k_2-\k_1)(\k_3-\k_2)Y}\quad\text{if $Y>0$ and $z_{13}\ge0$,}
 \\
& e^{2\sqrt{c_{12}}z_{13}-(\k_2-\k_1)(\k_3-\k_2)Y}\quad\text{if $Y>0$ and $z_{13}\le0$.}
    \end{aligned}\right.
\end{equation}
Similarly,
\begin{gather}
  \label{eq:Y-23}
u_{\mY}-u_{23}=O\left(e^{-2\sqrt{c_{12}}z_{13}+(\k_2-\k_1)(\k_3-\k_2)Y}\right)
  \quad\text{if $Y<0$ and $z_{13}\ge0$,}
  \\ \label{eq:Y-12}
u_{\mY}-u_{12}=O\left(e^{2\sqrt{c_{23}}z_{13}+(\k_2-\k_1)(\k_3-\k_2)Y}\right)
  \quad\text{if $Y<0$ and $z_{13}\le0$.}
\end{gather}
Let $\eta=\eta_R+i\nu$, $|\eta_R|<\eta_{13}^*(\a,\pm\nu)$,
$\tchi_{0,R}(y)=\chi_0(y/(3R))-\chi_0(y/R)$ and
\begin{equation*}
  g_{13}^{\pm,R}(\eta)=\tchi_{0,R}(y)g_{13}^0(z_{13},\pm\eta)e^{iy\eta}\,.
\end{equation*}
Then $\|g_{13}^{\pm,R}(\eta)\|_{\calX(\nu)}=O(R^{1/2})$.  On the other hand, it
follows from Lemma~\ref{lem:evmL0} and \eqref{eq:Y-13} that
\begin{equation*}
  \|(\lambda_{13}^\pm(\eta)-\mL_{\mY}) g_{13}^{\pm,R}(\eta)\|_{\calX(\nu)}
  =O(R^{-1/2})\,.
\end{equation*}
Thus we have $\lambda_{13}^\pm(\eta)\in\sigma(\mL_{\mY})$
provided $|\eta_R|\le \eta_{13}^*(\a,\pm\nu)$.
Combining the above with Proposition~\ref{lem:wmL1-decay}
and Lemma~\ref{lem:X13nu}, we have $\sigma(\mL_{13})\subset \sigma(\mL_{\mY})$.
\par
Similarly, we can prove
$\sigma(\mL_{12})\cup \sigma(\mL_{23})\subset \sigma(\mL_{\mY})$ by 
using \eqref{eq:Y-23}, \eqref{eq:Y-12},  Proposition~\ref{lem:wmL1-decay} and
Lemma~\ref{lem:specL23}.
\end{proof}

If $\lambda\in \rho(\mL_0)\cap\rho(\mL_{12})\cap\rho(\mL_{13})\cap\rho(\mL_{23})$, then
$\lambda\in\rho(\mL_{\mY})$ unless $\lambda$ is an eigenvalue of $\mL_{\mY}$ or $\mL_\mY^*$.
More precisely, we have the following.
\begin{lemma}
  \label{lem:spec1}
Assume \eqref{ass:alpha23} and \eqref{ass:X13nu}.
Let $\mL_{13}$ be a closed operator on $\calX_{13}(\nu)$,
$\mL_{12}$ and $\mL_{23}$ be closed operators on $\calX_{12}$,
and $\mL_0$ and $\mL_{\mY}$ be closed operators on $\calX(\nu)$.
If $\lambda\in \rho(\mL_0)\cap \rho(\mL_{12})\cap \rho(\mL_{13})\cap \rho(\mL_{23})$,
then $\lambda\in\sigma_p(\mL_\mY)$ or
$\lambda-\mL_\mY$ has a bounded inverse on
${}^\perp\!\ker(\lambda-\mL_\mY)^*$.
\end{lemma}
\begin{proof}[Proof of Lemma~\ref{lem:spec1}]
Let $\chi_0$ be as \eqref{eq:chi0}.
For $(i,j)=(1,2)$, $(1,3)$, $(2,3)$ and $R>0$, let
\begin{gather*}
  \chi_{ij,R}(x,y)=\left\{
  \begin{aligned}
&\chi_0\left(\tfrac{|z_{ij}|}{R}\right)
\left(1-\chi_{0,R}(x,y)\right) & \text{ if $(-1)^{i+j}y>0$,}
\\ & 0  & \text{ if $(-1)^{i+j}y<0$,}
  \end{aligned}\right.
\\
\chi_{1,R}(x,y)=1-\chi_{0,R}(x,y)-\sum_{1\le i<j\le 3}\chi_{ij,R}(x,y)\,.
\end{gather*}
\par
Suppose that $\lambda-\mL_{\mY}$ does not have a bounded inverse.
Then there exist $\{u_n\}$ and $\{f_n\}$ such that
\begin{equation}
\label{eq:ass-un}
\left(\lambda-\mL_\mY\right)u_n=f_n\,,\quad
\|u_n\|_{\calX(\nu)}=1\,,\quad \lim_{n\to\infty}\|f_n\|_{\calX(\nu)}=0\,.
\end{equation}
By \cite[Proposition~C.1, (C.5)]{Miz23}, $\{u_n\}$ is bounded in a Hilbert
space $\mathcal{E}$ equipped with a norm
$\|u\|_{\mathcal{E}}:=\left(\|u\|_{\calX(\nu)}^2+\|\pd_x^2u\|_{\calX(\nu)}^2
+\|(\pd_x^{-1}\pd_y)^2u\|_{\calX(\nu)}^2\right)^{1/2}$.
Since $\mathcal{E}$ is compactly embedded into $L^2_{loc}(\R^2)$,
there exist a subsequence $\{u_n'\}$ and $u_\infty\in\calX(\nu)$ such that
$u_n'\to u_\infty$ weakly in $\mathcal{E}$ and locally in $\calX(\nu)$, and 
$(\lambda-\mL_{\mY})u_\infty=0$.
If $u_\infty\ne0$, then $\lambda$ is an eigenvalue of $\mL_{\mY}$.
\par
Suppose on the contrary, that every convergent subsequence of $\{u_n\}$ converges to
$0$ weakly in $\mathcal{E}$.  Then $u_n\to0$ locally in $\calX(\nu)$ and
there exists a subsequence $\{u_{n(k)}\}$ such that
\begin{equation}
  \label{eq:unk-0}
\lim_{k\to\infty}\|\chi_{0,k}u_{n(k)}\|_{\calX(\nu)}=0\,.
\end{equation}
Let
$\tu^1_k=\chi_{1,k}u_{n(k)}$ and $\tu^{ij}_k=\chi_{ij,k}u_{n(k)}$.
Then
\begin{gather}
\label{eq:tu-1}  
(\lambda-\mL_0)\tu^1_k=
  \chi_{1,k}f_{n(k)}-[\mL_0,\chi_{1,k}]u_{n(k)}
-\frac{3}{2}\chi_{1,k}\pd_x(u_{\mY}u_{n(k)})=:\tilde{f}^1_k\,,
\\ \label{eq:tu-ij}
\begin{aligned}
  (\lambda-\mL_{ij})\tu^{ij}_k=&
  \chi_{ij,k}f_{n(k)}-[\mL_{ij},\chi_{ij,k}]u_{n(k)}
  \\ &-\frac{3}{2}\chi_{ij,k}\pd_x\{(u_{\mY}-u_{ij})u_{n(k)}\}
  =:\tilde{f}^{ij}_k\,,  
\end{aligned}
\end{gather}
and
\begin{equation}
  \label{eq:tfk-lim}
\lim_{k\to\infty}(\|\tilde{f}^1_k\|_{\calX(\nu)}
+\sum_{1\le i<j\le3}\|\tilde{f}^{ij}_k\|_{\calX(\nu)})=0\,.  
\end{equation}
Indeed, by \cite[(C.6)]{Miz23}, 
\begin{align*}
& \left\|[\mL_0,\chi_{1,k}]u_{n(k)}\right\|_{\calX(\nu)}
+\left\|[\mL_{ij},\chi_{ij,k}]u_{n(k)}\right\|_{\calX(\nu)}
 \lesssim k^{-1}\|u_{n(k)}\|_{\mathcal{E}}\lesssim k^{-1}\,,
\end{align*}
and by \eqref{eq:Y-13}--\eqref{eq:Y-12}, there exists a $c>0$ such that
\begin{align*}
& \left\|\chi_{1,k}\pd_x(u_{\mY}u_{n(k)})\right\|_{\calX(\nu)}
+\sum_{1\le i<j\le 3}\left\|\chi_{ij,k}\pd_x\{(u_{\mY}-u_{ij})u_{n(k)}\}\right\|_{\calX(\nu)}
\\ \lesssim &
e^{-ck}\|u_{n(k)}\|_{\mathcal{E}}\lesssim e^{-ck}\,.
\end{align*}
Since
$\lambda\in\cap_{1\le i<j\le 3}\rho(\mL_{ij})\cap \rho(\mL_0)$,
it follows from \eqref{eq:tu-1}--\eqref{eq:tfk-lim} that as $k\to\infty$,
\begin{equation*}
  \|\tu^1_k\|_{\calX(\nu)}+\|\tu^{12}_k\|_{\calX(\nu)}+\|\tu^{13}_k\|_{\calX(\nu)}+\|\tu^{23}_k\|_{\calX(\nu)}\to0\,.
\end{equation*}
Combining the above with \eqref{eq:unk-0}, we have
$ \lim_{k\to\infty}\|u_{n(k)}\|_{\calX(\nu)}=0$, which contradicts to \eqref{eq:ass-un}.
Thus we prove that either $\lambda$ is an eigenvalue of $\mL_\mY$ or 
$\lambda-\mL_{\mY}$ has a bounded inverse on
$\overline{\operatorname{Range}(\lambda-\mL_{\mY})}$.
This completes the proof of Lemma~\ref{lem:spec1}.
\end{proof}

To investigate existence of eigenvalues in Section~\ref{sec:nonexistence},
we will prove that unstable eigenfunctions of $\mL_\mY$ and $\mL_\mY^*$
should be strongly localized.
\begin{lemma}
  \label{lem:imev}
  Let $\a$, $\nu$, $\mL_{23}$ and $\mL_{\mY}$ be as in Lemma~\ref{lem:spec1} and
  $\mL_0$ be a closed operator on $\calX(\nu)$.
Assume that $\lambda\in\rho(\mL_{23})\cap \rho(\mL_0)$.
\begin{itemize}
\item[\rm{(i)}]
If $\mL_{\mY} u=\lambda u$ for $u\in\calX(\nu)$,
then there exists an $\a'>0$ such that
$e^{\a'|z_{13}|}u\in \calX_{13}(\nu)$ and $e^{\a'|z_{12}|}u\in\calX_{12}$.
\item[\rm{(ii)}]
If $\mL_{\mY}^*u=\lambda u$ for $u\in\calX(\nu)^*$,
then there exists an $\a'>0$ such that
$e^{\a'|z_{13}|}u\in \calX_{13}(\nu)^*$ and $e^{\a'|z_{12}|}u\in\calX_{12}^*$.
\end{itemize}
\end{lemma}

\begin{proof}[Proof of Lemma~\ref{lem:imev}]
Let $\chi_{2,R}(y)=\chi_0(-y/R)$, $\chi_{3,R}(y)=1-\chi_{2,R}(y)$,
$\tchi_{2,R}(y)=\chi_{2,R}(y+R)$ and
$\tchi_{3,R}(y)=\chi_{3,R}(y-R)$. Then
$\chi_{2,R}(y)=1$ if $y\ge -R$, $\chi_{2,R}(y)=0$ if $y\le -2R$,
and $\chi_{2,R}\tchi_{2,R}=\chi_{2,R}$,
$\chi_{3,R}\tchi_{3,R}=\chi_{3,R}$.
Let
  \begin{align*}
W_1(\lambda)=& \tchi_{2,R}(y)(\lambda-\mL_0)^{-1}\chi_{2,R}(y)
+\tchi_{3,R}(y)(\lambda-\mL_{23})^{-1}\chi_{3,R}(y)\,.
  \end{align*}
Then
\begin{equation}
  \label{eq:imev-6}
W_1(\lambda)(\lambda-\mL)=I+B_1+G_1\,,
\end{equation}
where $B_1=B_{11}+B_{12}$, $G_1=G_{11}+G_{12}$ and
\begin{align*}
  B_{11}=& \tchi_{2,R}(y)(\lambda-\mL_0)^{-1}[\mL,\chi_{2,R}(y)]\,,
  \\
B_{12}=& \tchi_{3,R}(y) (\lambda-\mL_{23})^{-1}[\mL, \chi_{3,R}(y)]\,,
\end{align*}
\begin{align*}
  G_{11}=&
\frac32\tchi_{2,R}(y)(\lambda-\mL_0)^{-1}\pd_x(u_{\mY}\chi_{2,R}(y)\cdot)
  \\
G_{12}=  & 
\frac32\tchi_{3,R}(y)(\lambda-\mL_{23})^{-1}
\pd_x\{(u_{\mY}-u_{23})\chi_{3,R}(y)\cdot\} \,.
\end{align*}
In view of the proof of  Lemma~\ref{lem:intersection}
and the proof of  \cite[Proposition~C.1]{Miz23},
there exists an $\a_1>0$ such that
if $i$, $j\ge0$ and $i+j\le 2$ or $(i,j)=(-2,2)$,
\begin{equation}
\label{eq:mL0-inv}  
\sup_{\a'\in[-\a_1,\a_1]}
\left\|\pd_x^i\pd_y^j(\lambda-\mL_0)^{-1}\right\|_{B(\calX^{\a+\a'}(\nu))}<\infty\,,
\end{equation}
 and
$\|B_{11}\|_{B(\calX^{\a+\a'}(\nu))}=O(R^{-1})$.
Since $\chi_{3,R}(y)=\tchi_{3,R}(y)=0$ for $y\ge0$,
we have as in \cite[Appendix~C]{Miz23},
\begin{gather*}
\|B_{12}\|_{B(\calX^{\a+\a'})(\nu)}
=\|e^{\a' z_{12}}B_{22}e^{-\a' z_{12}}\|_{\calX_{12}}=O(R^{-1})\,.
\end{gather*}
By \eqref{eq:Y-13}--\eqref{eq:Y-12},
\begin{gather*}
u_{\mY}\lesssim
\begin{cases}
& \min\{e^{2\sqrt{c_{12}}z_{13}},e^{-2\sqrt{c_{23}}z_{13}}\}\text{ for $Y\ge0$,}
\\ &  e^{-\sqrt{c_{12}}|z_{12}|}+e^{-\sqrt{c_{23}}|z_{23}|}\text{ for $Y\le0$.} 
\end{cases}
\end{gather*}
Combining the above with \eqref{eq:mL0-inv} and the fact that
$\supp\chi_{2,R}\subset[-2R,\infty)$, we have
 $\|G_{11}u\|_{\calX^{\a+\a'}}(\nu)\lesssim \|u\|_{\calX(\nu)}$.
\par
By \cite[Lemma~C.4]{Miz23},
$\lambda-\mL_{23}$ is invertible on $e^{-\a'z_{12}}\calX_{12}$ provided $\a'$
is sufficiently close to $0$.
Since $\supp\chi_{3,R}\subset\supp\tchi_{3,R}\subset  \R_-$ and
\begin{gather*}
  u_\mY-u_{23}= O(e^{-2\sqrt{c_{12}}z_{12}}) \quad\text{if $z_{12}\ge0$,}
\\
|u_\mY-u_{23}|\le  u_\mY+u_{23} \lesssim e^{2\min\{\sqrt{c_{12}},
\sqrt{c_{23}}\}z_{12}}
\quad\text{if $z_{12}\le 0$ and $y\le0$,}
\end{gather*}
we have 
$\|G_{12}u\|_{\calX^{\a+\a'}(\nu)}
=\|e^{\a'z_{12}}G_{12}u\|_{\calX_{12}}\lesssim \|u\|_\calX$.
\par
Suppose that $\mL_{\mY} u=\lambda u$ for $u\in\calX$ and
$\lambda\in\rho(\mL_0)\cap\rho(\mL_{23})$. Then
$$ (I+B_1)u+G_1u=0\,.$$
If $R>0$ is sufficiently large and $\a'$ is sufficiently close to $0$,
then $I+B_1$ is invertible and
$$u=-(I+B_1)^{-1}G_1u\in \calX^{\a-\a'}(\nu)\cap \calX^{\a+\a'}(\nu)\,.$$
\par
Since $\rho(\mL_0)=\rho(\mL_0^*)$ and $\rho(\mL_{23})=\rho(\mL_{23}^*)$,
we can prove the part (2) in the same way.
This completes the proof of Lemma~\ref{lem:imev}.
\end{proof}
\par
Next, we will find eigenfunctions of  $\mL_\mY^*$
that can be written as $\Phi^1(-i\beta)/\tau_\mY$ and
$\Phi^1_1\Phi^{1,*}(-i\beta)$.
We remark that $\calX(\nu)^*=\calX_{13}(\nu)^*+\calX_{12}^*$ 
if $\a>0$ and $\nu\ge0$ and that the weight function
of $\calX(\nu)^*$ localizes perturbations around $[2,3]$-soliton.
\par
Let $D_1(\nu)$ be a domain of $\C$ enclosed by a line $\Re\beta=\k_1-\a$ and
two curves
\begin{gather*}
\mathcal{C}_{1\ell}=
\{\beta_{12}^-(\eta_R)\mid  -\eta_{12}^*(\a)\le \eta_R\le \eta_{12}^*(\a)\}\,,
\\
\mathcal{C}_{1r}(\nu)=
\{\beta_{13}^-(\eta_R+i\nu)\mid  -\eta_{13}^*(\a,\nu)\le \eta_R\le \eta_{13}^*(\a,\nu)\}\,.
\end{gather*}

Let $D_2$ be a domain enclosed by two curves
\begin{gather*}
\mathcal{C}_{2\ell}=
\{\beta_{12}^+(\eta_R)\mid  -\eta_{12}^*(\a)\le \eta_R\le \eta_{12}^*(\a)\}\,,
\\
\mathcal{C}_{2r}=
\{\beta_{23}^-(\eta_R-i\eta_I^{23})\mid  -\eta_{\#,-}(\a)\le \eta_R\le
 \eta_{\#,-}(\a)\}\,,
\end{gather*}
where $\eta_I^{23}=(\k_1-\k_3)\a$, and  $D_3(\nu)$ be a domain enclosed by
a line $\Re \beta=\k_3+\a$ and two curves
\begin{gather*}
\mathcal{C}_{3\ell}=
\{\beta_{23}^+(\eta_R-i\eta_I^{23})\mid 
 -\eta_{*,-}(\a)\le \eta_R\le  \eta_{*,-}(\a)\}\,,
\\
\mathcal{C}_{3r}(\nu)=
\{\beta_{13}^+(\eta_R-i\nu)\mid  -\eta_{13}^*(\a,-\nu)\le \eta_R\le \eta_{13}^*(\a,-\nu)\}\,.
\end{gather*}
\begin{remark}
\label{rem:boundary}
We have
\begin{gather}
\label{eq:curve-endpoints1}
\left\{
  \begin{aligned}
& \beta_{12}^{-}(\pm\eta_{12}^*(\a))=\k_1-\a\mp i\sqrt{\a(\a+\k_2-\k_1)}\,,
\\ &
\beta_{13}^{-}(\pm\eta_{13}^*(\a,\nu)+i\nu)
=\k_1-\a\mp i\sqrt{\a(\a+\k_3-\k_1)+\nu}\,,
\end{aligned}\right.
\\ \label{eq:curve-endpoints2}
\beta_{12}^+(\pm\eta_{12}^*(\a))=
\beta_{23}^-(\mp\eta_{\#,-}(\a)-i\eta_I^{23})=\k_2+\a
\pm i\sqrt{\a(\a+\k_2-\k_1)}\,,
\\
\label{eq:curve-endpoints3}
\left\{
  \begin{aligned}
& \beta_{13}^+(\pm\eta_{13}^*(\a,-\nu))
=\k_3+\a\pm i\sqrt{\a(\a+\k_3-\k_1)-\nu}\,,
\\
& \beta_{23}^+(\pm\eta_{*,-}(\a)-i\eta_I^{23})=\k_3+\a
\pm i\sqrt{\a(\a+2\k_3-2a_{12})}\,.
  \end{aligned}\right.
\end{gather}  
\par
If $\nu\ge0$, then the curve $\beta_{12}^-(\eta)$ $(\eta\in\R)$ is to
the left of the curve $\beta_{13}^-(\eta)$ $(\eta\in\R)$. The two curves
meet only at $\beta=\k_1$ if $\nu=0$.  If $\nu>0$, then
$\mathcal{C}_{1\ell}\cap\mathcal{C}_{1r}(\nu)=\emptyset$.  The two
curves $\mathcal{C}_{1\ell}$ and $\mathcal{C}_{1r}(\nu)$ are to the
right of the line $\Re\beta=\k_1-\a$.  The curve $\mathcal{C}_{2\ell}$
is in a half plane $\Re \beta\le \k_2+\a$, and the curve
$\mathcal{C}_{2r}$ is in a half plane $\Re \beta\ge \k_2+\a$. The two
curves meet at $\k_2+\a\pm i\sqrt{\a(\a+\k_2-\k_1)}$.
\par
Suppose that $\nu$ satisfies \eqref{ass:X13nu}. 
Then $\beta_{13}^-(i\nu)<\k_2$ and
$D_1(\nu)\cap D_2=\emptyset$
because $\beta_{13}^-(i\nu)$ is the rightmost point of $D_1(\nu)$ and
$\k_2$ is the leftmost point of $D_2$.
Moreover, we have $\eta_{13}^*(\a,-\nu)>0$ and $\mathcal{C}_{3r}(\nu)\ne\emptyset$.
\par
The curve $\mathcal{C}_{3r}(\nu)$ is to the right of the curve $\mathcal{C}_{3\ell}$,
and $\mathcal{C}_{3r}$ and $\mathcal{C}_{3\ell}$ do not intersect.
We have $D_2\cap D_3(\nu)=\emptyset$ because 
$\beta_{23}^-(-i\eta_I^{23})<\beta_{23}^+(-i\eta_I^{23})$ and 
the rightmost point of $D_2$ is $\beta_{23}^-(-i\eta_I^{23})$
and the leftmost point of $D_3(\nu)$ is $\beta_{23}^+(-i\eta_I^{23})$.
The two curves $\mathcal{C}_{3r}(\nu)$ and $\mathcal{C}_{3\ell}$ are to the left
of the line $\Re\beta=\k_3+\a$.
The sets  $D_1(\nu)$ and $D_3(\nu)$ are monotone increasing in $\nu\ge0$.
\end{remark}

Let
\begin{gather*}
V(\beta)=e^{-\lambda(\beta)t} \Phi^1(-i\beta)\Phi^{1,*}_1\,,\quad
\widehat{V}(\beta)=e^{\lambda(\beta)t}(\beta-\k_1)
\Phi^1_1\Phi^{1,*}(-i\beta)
\end{gather*}
If $\lambda\in \lambda(D_1(\nu))\cup -\lambda(D_2)\cup-\lambda(D_3(\nu))$,
then $\lambda$ is an eigenvalue of $\mL_\mY^*$.
\begin{lemma}
\label{lem:adjointef}
 Assume \eqref{ass:alpha23} and \eqref{ass:X13nu}.
Then $V(\beta)\in \calX(\nu)^*$ if and only if $\beta\in D_2\cup D_3(\nu)$, and
$\widehat{V}(\beta)\in \calX(\nu)^*$ if and only if $\beta\in D_1(\nu)$, and
\begin{equation*}
\{\lambda(\beta)\mid \beta\in D_1(\nu)\}\cup  
\{-\lambda(\beta)\mid \beta\in D_2\cup D_3(\nu)\}\subset \sigma_r(\mL_\mY)\,.
\end{equation*}
\end{lemma}
\begin{remark}
\label{rem:adjointef}
By the definitions,
\begin{gather*}
\lambda_{ij}^+(\eta)=\lambda(\beta_{ij}^-(\eta))\,,\quad
\lambda_{ij}^-(\eta)=-\lambda(\beta_{ij}^+(-\eta))\,,\quad
\tilde{\lambda}_{ij}^-(\eta)=-\lambda(\beta_{ij}^-(-\eta))\,,  
\\
\{\lambda_{12}^+(\eta_R)\}_{|\eta_R|\le \eta_{12}^*(\a)}
\cup \{\lambda_{13}^+(\eta_R+i\nu)\}_{|\eta_R|\le \eta_{13}^*(\a,\nu)}
\subset \pd\left(\lambda(D_1(\nu))\right)\,,
\\
\{\lambda_{12}^-(\eta_R)\}_{|\eta_R|\le \eta_{13}^*(\a)} \cup
\{\tilde{\lambda}_{23}^-(\eta_R+i\eta_I^{23})\}_{|\eta_R|\le \eta_{\#,-}(\a)}
\subset \pd\left(-\lambda(D_2)\right)\,,
\\
\{\lambda_{13}^-(\eta_R+i\nu)\}_{|\eta_R|\le \eta_{13}^*(\a,-\nu)}
\cup \{\lambda_{23}^-(\eta_R+i\eta_I^{23})\}_{|\eta_R|\le \eta_{*,-}(\a)}
\subset \pd\left(-\lambda(D_3(\nu))\right)\,,
\end{gather*}
where $\pd(A)$ denotes the boundary of a set $A\subset \C$. 
\par

Since
\begin{equation}
  \label{eq:6}
\left\{
  \begin{aligned}
& \theta_j-\theta_i=\sqrt{c_{ij}}(X+2a_{ij}Y)=\sqrt{c_{ij}}z_{ij}\,,
\\ &
 \beta x+\beta^2y-\beta^3t-\frac12(\theta_i+\theta_j)-\lambda(\beta) t
= (\beta-a_{ij})X+\left(\beta^2-\frac{\k_i^2+\k_j^2}{2}\right)Y\,,
  \end{aligned}\right.
\end{equation}
it follows from \eqref{eq:tauXY} that $V(\beta)$ and $\widehat{V}(\beta)$
depend only on $X$ and $Y$. Thus by Lemma~\ref{lem:prodJdJ},
$\mL_\mY^*V(\beta)=-\lambda(\beta)V(\beta)$
and $\mL_\mY^*\widehat{V}(\beta)=\lambda(\beta)\widehat{V}(\beta)$.

Moreover,
\begin{gather}
\label{eq:Vbeta=g*}
\frac{\k_1-\k_j}{\beta_{1j}^+(-\bar{\eta})-\k_1}\widehat{V}(\beta_{1j}^-(-\bar{\eta}))=g_{1j}^{*,+}(\eta)
\,,\enskip
V(\beta_{1j}^+(\bar{\eta}))=g_{1j}^{*,-}(\eta)
\quad\text{for $j=2$, $3$,}
\\  
V(\beta_{23}^+(\bar{\eta}))=g_{23}^{*,-}(\eta)\,,
\quad
V(\beta_{23}^-(\bar{\eta}))=\tg_{23}^{*,-}(\eta)\,.
\end{gather}
\end{remark}
\begin{proof}[Proof of Lemma~\ref{lem:adjointef}]
Suppose that
\begin{gather}
  \label{eq:1}
\Re\beta\in (\k_1+\a,\k_3+\a)\,,\quad\Re(\beta-\k_1)(\beta-\k_3)<\nu\,,
\\   \label{eq:1'}
\left\{
  \begin{aligned}
&  \Re(\beta-\k_1)(\beta-\k_2)>0\quad\text{if
$\Re\beta\le \k_2+\a$,}
\\ &
\Re \{(\beta-\k_2)(\beta-\k_3)\}>\eta_I^{23}
\quad\text{if $\Re\beta\ge\k_2+\a$.}
  \end{aligned}\right.
\end{gather}
By \eqref{eq:tauXY} and \eqref{eq:6},
\begin{gather}
\label{eq:1-1}
V(\beta)=
O\left(e^{(\beta-a_{13})z_{13}}\sech\sqrt{c_{13}}z_{13}
e^{(\beta-\k_1)(\beta-\k_3)Y}\right)\,,
\\ \label{eq:1-2}
V(\beta)=
\left\{
  \begin{aligned}
& O\left(e^{(\beta-a_{12})z_{12}}\sech\sqrt{c_{12}}z_{12} e^{(\beta-\k_1)(\beta-\k_2)Y}\right)\,,
\\  &
O\left(e^{(\beta-a_{23})z_{23}}\sech\sqrt{c_{23}}z_{23} 
e^{(\beta-\k_2)(\beta-\k_3)Y}\right)\,.
  \end{aligned}\right.
\end{gather}
Combining \eqref{eq:1} with \eqref{eq:1-1}
and \eqref{eq:1'} with \eqref{eq:1-2}, we have
$V(\beta)\in\calX(\nu)^*$ for $\beta$ satisfying \eqref{eq:1} and \eqref{eq:1'}.
\par
Since
\begin{gather}
  \label{eq:judge1}
\Re(\beta_{jk}^\pm(\eta)-\k_j)(\beta_{jk}^\pm(\eta)-\k_k)=-\Im\eta\,,
\\   \label{eq:judge2}
\lim_{\beta\to\pm\infty}(\beta-\k_j)(\beta-\k_k)=\infty\,, 
\end{gather}
we have
\begin{equation*}
\Re(\beta-\k_1)(\beta-\k_3)<\nu  
\end{equation*}
when $\beta$ is to the left of the curve $\beta_{13}^+(\eta_R-i\nu)$
and to the right of the curve $\beta_{13}^-(\eta_R-i\nu)$,
\begin{equation*}
  \Re(\beta-\k_1)(\beta-\k_2)>0
\end{equation*}
when $\beta$ is to the right of the curve $\beta_{12}^+(\eta_R)$
or to the left of the curve $\beta_{12}^-(\eta_R)$, and
\begin{align*}
  \Re(\beta-\k_2)(\beta-\k_3)>\eta_I^{23}
\end{align*}
when $\beta$ is to the right of the curve
$\beta_{23}^+(\eta_R-i\eta_I^{23})$
or to the left of the curve $\beta_{23}^-(\eta_R-i\eta_I^{23})$.
Combining these with \eqref{eq:curve-endpoints2} and
\eqref{eq:curve-endpoints3}, we have $\beta\in D_2\cup D_3(\nu)$ if and only if
$\beta$ satisfies \eqref{eq:1} and \eqref{eq:1'}.
\par
Next, we will prove that $\beta$ satisfies \eqref{eq:1} and \eqref{eq:1'} 
if $V(\beta)\in \calX(\nu)^*$. 
By \eqref{eq:6} and \eqref{eq:prod-PhiPhi*j}, we have
$V(\beta)\sim e^{(\beta-\k_1)X}$ as $X\to -\infty$,
$V(\beta)\sim e^{(\beta-\k_3)X}$ as $X\to \infty$,
\begin{gather}
 \label{eq:11-1}
V(\beta)\sim e^{(\beta-\k_1)(\beta-\k_3)Y}
\quad\text{if $Y\ge0$ and $z_{13}=0$,}
\\  \label{eq:11-2}
V(\beta)\sim e^{(\beta-\k_1)(\beta-\k_2)Y}
\quad\text{if $Y\le0$ and $z_{12}=0$,}
\\ \notag
e^{-\a z_{12}}V(\beta)\sim e^{\{(\beta-\k_2)(\beta-\k_3)-\eta_I^{23}\}Y}
\quad\text{if $Y\le0$ and $z_{23}=0$.}
\end{gather}
Thus we have \eqref{eq:1} and
\begin{gather}
\label{eq:1-4}
\Re(\beta-\k_1)(\beta-\k_2)>0\,,\quad
\Re \{(\beta-\k_2)(\beta-\k_3)\}>\eta_I^{23}\,,
\end{gather}
provided $V(\beta)\in \calX(\nu)^*$.
Since 
\begin{align*}
&\Re\{(\beta-\k_2)(\beta-\k_3)-\eta_I^{23}\}
-\Re(\beta-\k_1)(\beta-\k_2)
\\=& (\k_3-\k_1)(\k_2+\a-\Re\beta) \lesseqgtr 0
\end{align*}
for $\beta$ satisfying $\Re\beta \gtreqless\k_2+\a$, we have
\eqref{eq:1'} from \eqref{eq:1-4}.
\par
By \eqref{eq:judge1} and \eqref{eq:judge2},
we have $\beta\in D_1(\nu)$ if and only if
\begin{equation}
\label{eq:1hat} 
\left\{\begin{gathered}
\Re\beta\in (\k_1-\a,a_{13}-\sqrt{c_{13}-\nu}]\,,\\
\Re(\beta-\k_1)(\beta-\k_3)+\nu>0\,,\quad 
\Re(\beta-\k_1)(\beta-\k_2)<0\,.
\end{gathered}\right.  
\end{equation}
We will prove that $\widehat{V}(\beta)\in\calX(\nu)^*$
for $\beta$ satisfying \eqref{eq:1hat}.
If $Y\ge 0$ and $\theta_3\ge \theta_2$,
\begin{align}
\widehat{V}(\beta)\simeq &
e^{-\beta X-\beta^2Y}\frac{e^{\theta_1+\theta_3}}{e^{\theta_1}+e^{\theta_3}}
\notag 
\\ \label{eq:10-1} \sim &
e^{-(\beta-a_{13})z_{13}}\sech\sqrt{c_{13}}z_{13} e^{-(\beta-\k_1)(\beta-\k_3)Y}\,.
\end{align}
If $Y\ge 0$ and $\theta_3\le \theta_2$, then $\theta_1\ge\theta_2\ge\theta_3$,
$z_{13}\le0$ and
\begin{align*}
\widehat{V}(\beta)\simeq &
e^{(\k_2-\beta) X+(\k_2^2-\beta^2)Y}
\\ = &
e^{(\k_2-\beta)z_{13}} e^{-\{(\beta-\k_1)(\beta-\k_3)+(\k_2-\k_1)(\k_3-\k_2)\}Y}\,.
\end{align*}
If $Y\le 0$ and $\theta_2\ge \theta_3$,
\begin{equation}
\label{eq:10-3}
\widehat{V}(\beta)\simeq 
e^{-(\beta-a_{12})z_{12}}\sech\sqrt{c_{12}}z_{12} e^{-(\beta-\k_1)(\beta-\k_2)Y}\,.
\end{equation}
If $Y\le 0$ and $\theta_2\le \theta_3$, then $\theta_3\ge\theta_2\ge\theta_1$,
$z_{12}\ge0$ and
\begin{equation}
\label{eq:10-4}
\widehat{V}(\beta)\simeq 
e^{(\k_1-\beta)z_{12}-(\beta-\k_1)(\beta-\k_2)Y}\,.
\end{equation}
Combining the above with \eqref{eq:1hat}, 
we have $\widehat{V}(\beta)\in \calX(\nu)^*$.
Moreover, we have $\beta\in D_1(\nu)$ if $\widehat{V}(\beta)\in \calX(\nu)^*$.
This completes the proof of Lemma~\ref{lem:adjointef}.
\end{proof}
By Lemma~\ref{lem:prodJdJ},
$\pd_XV(\beta)$ and $\pd_X\widehat{V}(\beta)$ are formally eigenfunctions of $\mL_\mY$
but they do not belong to $\calX(\nu)$.
\begin{lemma}
  \label{lem:no-ef}
Assume \eqref{ass:alpha23} and $\nu\ge0$.
Then $\pd_XV(\beta)\not\in \calX(\nu)$ and 
$\pd_X\widehat{V}(\beta)\not\in \calX(\nu)$ for every $\beta\in\C$.
\end{lemma}
\begin{proof}
Since $\calX(\nu)\subset \calX$ for $\nu\ge0$, it suffices to prove
Lemma~\ref{lem:no-ef} for $\nu=0$.
The proof follows the line of that of Lemma~\ref{lem:adjointef}.
If $\pd_XV(\beta)\in \calX$, then
\begin{gather}
  \label{eq:1-3}
\Re\beta\in (\k_1-\a,\k_3-\a)\,,\quad\Re(\beta-\k_1)(\beta-\k_3)<0\,,
\\   \label{eq:1-3'}
\Re(\beta-\k_1)(\beta-\k_2)>0\quad\text{if
$\Re\beta\le\k_2-\a$,}
\\   \label{eq:1-3''}
\Re \{(\beta-\k_2)(\beta-\k_3)+\eta_I^{23}\}>0
\quad\text{if $\Re\beta\ge\k_2-\a$.}
\end{gather}
By \eqref{eq:1-3}, $\beta$ is to the right of $\beta_{13}^-(\eta_R)$.
If $\beta$ satisfies \eqref{eq:1-3'}, then
$\beta$ is to the left of $\beta_{12}^-(\eta_R)$
because $\Re\beta_{12}^+(\eta_R)\ge\k_2$ for any $\eta_R$.
The curve $\beta_{13}^-(\eta_R)$ lies to the right of the curve
$\beta_{12}^-(\eta_R)$, and the two curves meet only at $\kappa_1$.
Thus there is no $\beta$ satisfying both \eqref{eq:1-3} and \eqref{eq:1-3'}.
\par
If $\beta$ satisfies \eqref{eq:1-3''}, then $\beta$ lies to the left of
$\beta_{23}^-(\eta_R+i\eta_I^{23})$ or to the right of
$\beta_{23}^+(\eta_R+i\eta_I^{23})$ 
and either $\Re \beta<\k_2-\a$ or $\Re\beta>\k_3+\a$ because 
$\Re\gamma_{23}(\eta_R+i\eta_I^{23})>\sqrt{c_{23}}+\a$ for every $\eta_R$.
Thus there is no $\beta$ satisfying both \eqref{eq:1-3} and \eqref{eq:1-3''}.
\par
If $\pd_X\widehat{V}(\beta)\in \calX$, 
then by \eqref{eq:10-1}--\eqref{eq:10-4}, 
\begin{gather}
\label{eq:10-5}
\Re\beta\in (\k_1+\a,\k_2+\a)\,,
\\ \label{eq:10-6}
\Re(\beta-\k_1)(\beta-\k_3)>0\,,\quad
\Re(\beta-\k_1)(\beta-\k_2)<0\,.
\end{gather}
It follows from \eqref{eq:10-6} that $\beta$ is to the left of the curve
$\beta_{13}^-(\eta_R)$ or to the right of the curve $\beta_{13}^+(\eta_R)$,
and either $\Re\beta\le\beta_{13}^-(0)=\k_1<\k_1+\a$ or
$\Re\beta\ge\beta_{13}^+(0)=\k_3>\k_2+\a$. Thus
there is no $\beta\in\C$ satisfying both \eqref{eq:10-5} and \eqref{eq:10-6}.
This completes the proof of Lemma~\ref{lem:no-ef}.
\end{proof}

Let $\calX_\mY(\nu)$ be a subspace of $\calX(\nu)$ satisfying secular conditions:
\begin{gather*}
\calX_\mY(\nu)=\left\{f\in\calX(\nu)\,\Bigm|\,
  \begin{aligned}
&  \la f,\widehat{V}(\beta)\ra=0\text{ for
$\beta\in D_1(\nu)$},
\\ &
\la f,V(\beta)\ra=0
\text{ for $\beta\in D_2\cup D_3(\nu)$}
  \end{aligned}\right\}\,.
\end{gather*}
We will prove that $f\in\calX_\mY(\nu)$ if $f\in\calX(\nu)$ is orthogonal to
six families of adjoint continuous eigenfunctions.
For the purpose, we will show that $\widehat{V}(\beta)$ and
$V(\beta)$ satisfy the Hardy condition (see \cite[p.~22]{Bell}).
\begin{lemma}
  \label{lem:orth-analytic}
Assume \eqref{ass:alpha23} and \eqref{ass:X13nu}. Let $f\in\calX(\nu)$,
$h_1(\beta)=\la\widehat{V}(\beta),\bar{f}\ra$
and $h_2(\beta)= \la V(\beta),\bar{f}\ra$.
Then $h_1$ is analytic on $D_1(\nu)$ and $h_2$  is analytic on $D_2\cup D_3(\nu)$.
Moreover, for $\eta_1\in(0,\eta_{12}^*(\a))$,
$\eta_2\in(0,\eta_{13}^*(\a,\nu))$ and $\eta_3\in(0,\eta_{13}^*(\a,-\nu))$,
\begin{gather}
\lim_{\eta_I\to+0}\left\|h_1(\beta_{12}^-(\eta_R+i\eta_I))
-\la f,g_{12}^{*,+}(\eta_R)\ra\right\|_{L^2(-\eta_1,\eta_1)}=0\,,
\\
\lim_{\eta_I\to\nu-0}\left\|h_1(\beta_{13}^-(\eta_R+i\eta_I))
-\la f,g_{13}^{*,+}(\eta_R)\ra\right\|_{L^2(-\eta_2,\eta_2)}=0\,,
\\ \label{eq:12+convergence}
\lim_{\eta_I\to+0}\left\|h_2(\beta_{12}^+(-\eta_R-i\eta_I))
-\la f,g_{12}^{*,+}(\eta_R)\ra
\right\|_{L^2(-\eta_1,\eta_1)}=0\,,
\\
\lim_{\eta_I\to\nu-0}\left\|h_2(\beta_{13}^+(-\eta_R-i\eta_I))
-\la f,g_{13}^{*,+}(\eta_R)\ra
\right\|_{L^2(-\eta_3,\eta_3)}\,,
\end{gather}
and for $\eta_4\in(0,\eta_{\#,-}(\a))$,
and $\eta_5\in(0,\eta_{*,-}(\a))$,
\begin{gather*}
\lim_{\eta_I\to\eta_I^{23}+0}\left\|h_2(\beta_{23}^-(-\eta_R-i\eta_I))
-\la f,\tg_{23}^{*,-}(\eta_R+i\eta_I^{23})\ra\right\|_{L^2(-\eta_4,\eta_4)}=0\,,
\\
\lim_{\eta_I\to\eta_I^{23}+0}\left\|h_2(\beta_{23}^+(-\eta_R-i\eta_I))
-\la f,g_{23}^{*,-}(\eta_R+i\eta_I^{23})\ra\right\|_{L^2(-\eta_5,\eta_5)}=0\,.
\end{gather*}
\end{lemma}
\begin{proof}
If $\beta\in D_2\cup D_3(\nu)$, then $\beta$ satisfies \eqref{eq:1} and
\eqref{eq:1-4}. 
Combining \eqref{eq:1} and \eqref{eq:1-4} with
\eqref{eq:1-1} and \eqref{eq:1-2}
and using Lebesgue's dominated convergence theorem, we see
that $h_2(\beta)$ is analytic on $D_2\cup D_3(\nu)$.
\par
Next, we will prove \eqref{eq:12+convergence}.
By \eqref{eq:Vbeta=g*},
\begin{align*}
h_2(\beta_{12}^+(-\eta))=&
\int_{\R^2} V(\beta_{12}^+(-\eta))f(X,Y)\,dXdY
= \la \bar{f}, g_{12}^{*,-}(\bar{\eta})\ra\,.
\end{align*}
Let $\eta=\eta_R+i\eta_I$ and $\eta_R\in [-\eta_1,\eta_1]$. 
If $\eta_1>0$ is sufficiently small, then $V(\beta_{12}^+(-\eta))f$ is
uniformly integrable on $\R^2_+$ and
$$\int_{\R^2_+}V(\beta_{12}^+(-\eta))f\,dXdY$$
is smooth in $\eta_I$.
Indeed, we have from \eqref{eq:1-1}
\begin{align*}
& e^{-\a z_{13}-\nu Y}|V(\beta_{12}^+(-\eta)|
 \\ & \lesssim 
e^{(\Re\beta_{12}^+(-\eta)-a_{13}-\a)z_{13}}\sech\sqrt{c_{13}}z_{13} 
e^{\Re\{(\beta_{12}^+(-\eta)-\k_1)(\beta_{12}^+(-\eta)-\k_3)-\nu\}Y}\quad\text{on $\R^2_+$,}
\end{align*}
and 
$|\Re\beta_{12}^+(\bar{\eta})-a_{13}-\a|< \sqrt{c_{13}}$,
\begin{align*}
\Re (\beta_{12}^+(-\eta)-\k_1)(\beta_{12}^+(-\eta)-\k_3)
=&  \eta_I-(\k_3-\k_2)(\Re \beta_{12}^+(-\eta)-\k_1)
\\ \le & \eta_I-(\k_3-\k_2)(\k_2-\k_1)<0\,.
\end{align*}
\par
On the other hand, using \eqref{eq:6} and \eqref{eq:prod-PhiPhi*j},
we have for $\eta=\eta_R+i\eta_I$,
\begin{align*}
&\int_{\R^2_-} V(\beta_{12}^+(-\eta))f(X,Y)\,dXdY
\\ =&
\int_{\R^2_-}
f(X,Y)e^{-iY\eta}\pd_x\left\{
e^{\gamma_{12}(-\eta)z_{12}}\frac{e^{(\theta_1+\theta_2)/2}}{\tau_\mY}
\right\}\,dXdY\,,
\\=&
\int_{\R^2_-} f(X,Y)\pd_x\{W_{12}\pd_x^{-1}\overline{g_{12}^{0,*}(z_{12},-\eta)}\}
e^{-iY\eta}\,dz_{12}dY
\\=&
\int_\R\left\{ (\mathcal{F}_yf_1)(z_{12},\eta)
\overline{g_{12}^{0,*}(z_{12},-\eta)}
+
\int_\R (\mathcal{F}_yf_2)(z_{12},\eta)
\pd_x^{-1}\overline{g_{12}^{0,*}(z_{12},-\eta)}\right\}
\,dz_{12}\,,
\end{align*}
where  $W_{12}=(e^{\theta_1}+e^{\theta_2})/\tau_\mY$,
$\mathbf{1}_{(-\infty,0]}$ is the characteristic function of $(-\infty,0]$
and
\begin{gather*}
f_1(z_{12},Y)=\sqrt{2\pi}\mathbf{1}_{(-\infty,0]}(Y)f(X,Y)W_{12}\,,
\\
f_2(z_{12},Y)=\sqrt{2\pi}\mathbf{1}_{(-\infty,0]}(Y)f(X,dYy)\pd_xW_{12}\,.
\end{gather*}
Since  $W_{12}$ and its derivatives are uniformly bounded,
we have $f_1$, $f_2\in\calX_{12}$ for $f\in\calX$.
By the Plancherel theorem,
\begin{equation*}
\|e^{\a z_{12}}\mF_yf_j(z_{12},\eta_R+i\eta_I)\|_{L^2(\R_{z_{12}\times R_{\eta_R}})}
=\|e^{y\eta_I}f_j\|_{\calX_{12}}\quad
\text{for $j=1$, $2$.}
\end{equation*}
Since $f_1=f_2=0$ for $y>0$, it follows from
Lebesgue's dominated convergence theorem that
\begin{equation*}
\lim_{\eta_I\to+0}\left(\|e^{y\eta_I}f_1-f_1\|_{\calX_{12}}
+\|e^{y\eta_I}f_2-f_2\|_{\calX_{12}}\right)=0\,.
\end{equation*}
\par

Combining the above with
\begin{equation*}
\sum_{j=0,\,1}\left\|e^{-\a z_{12}}\pd_x^{-j}
\{g_{12}^{0,*}(\cdot,-\eta)-g_{12}^{0,*}(\cdot,-\eta_R)\}
\right\|_{L^2(\R)}=O(\eta_I)\,,
\end{equation*}
we have
\begin{equation*}
\lim_{\eta_I\to+0}\left\|
\int_{\R^2_-}\left\{V(\beta_{12}^+(-\eta))-V(\beta_{12}^+(-\eta_R))\right\}
f(X,Y)\,dXdY\right\|_{L^2(-\eta_1,\eta_1)}=0\,.
\end{equation*}
Thus we prove \eqref{eq:12+convergence}.
Using \eqref{eq:1-1}, \eqref{eq:1-2} and \eqref{eq:10-1}--\eqref{eq:10-4},
we can prove the other cases in the same way.
\end{proof}

Thanks to the identity theorem and Lusin-Provalov type theorems,
we see that $f$ is orthogonal to every adjoint eigenfunctions of $\mL_\mY$
if and only if $f$ is orthogonal to three families of adjoint continuous eigenfunctions of
$\mL_\mY$. Let $\C_\pm=\{z\in\C\mid \pm\Im z>0\}$.

\begin{lemma}
  \label{lem:orth-condition-identity}
  Assume \eqref{ass:alpha23} and \eqref{ass:X13nu}.  Let
  $f\in\calX(\nu)$ and let $\Gamma_1^\pm\subset\overline{D_1(\nu)}$ be
  nonconstant smooth curves satisfying $\Gamma_1^\pm\cap \C_\pm\ne\emptyset$,
 and let $\Gamma_2\subset \overline{D_2}$ and
  $\Gamma_3 \subset \overline{D_3(\nu)}$ be nonconstant smooth curves.
\begin{itemize}
\item[(i)]
If $h_1=0$ on $\Gamma_1^\pm$, then $h_1=0$ on $\overline{D_1(\nu)}\cap\C_\pm$ and
  \begin{gather*}
\la f,g_{12}^{*,+}(\eta_R)\ra=0
\text{ for $\mp\eta_R\in[0,\eta_{12}^*(\a))$,}
\\
\la f,g_{13}^{*,+}(\eta_R+i\nu)\ra=0
\text{ for $\mp\eta_R\in[0,\eta_{13}^*(\a,\nu))$.}
  \end{gather*}
If $\nu>0$, then $h_1=0$ on $\overline{D_1(\nu)}$.
\item[(ii)]
If $h_2=0$ on $\Gamma_2$, then $h_2=0$ on $\overline{D_2}$ and
\begin{gather*}
\la f,g_{12}^{*,-}(\eta_R)\ra=0
\text{ for $\eta_R\in(-\eta_{12}^*(\a),\eta_{12}^*(\a))$,}
\\
\la f,\tg_{23}^{*,-}(\eta_R+i\eta_I^{23})\ra=0
\text{ for $\eta_R\in(-\eta_{\#,-}(\a),\eta_{\#,-}(\a))$.}
\end{gather*}

\item[(iii)] 
If $h_2=0$ on $\Gamma_3$, then $h_2=0$ on $\overline{D_3(\nu)}$ and
\begin{gather*}
\la f,g_{13}^{*,-}(\eta_R+i\nu)\ra=0
\text{ for $\eta\in(-\eta_{13}^*(\a,-\nu),\eta_{13}^*(\a,-\nu))$,}
\\
\la f,g_{23}^{*,-}(\eta_R+i\eta_I^{23})\ra=0
\text{ for $\eta_R\in(-\eta_{*,-}(\a),\eta_{*,-}(\a))$.}
\end{gather*}
\end{itemize}
\end{lemma}

\begin{proof}[Proof of Lemma~\ref{lem:orth-condition-identity}]
Let  $0<\eta_{12}^1<\eta_{12}^2<\eta_{12}^*(\a))$,
$0<\eta_{13}^1<\eta_{13}^2<\eta_{13}^*(\a,\nu)$ and
$C_j=\{\beta_{1j}^-(\eta)\mid \eta_{1j}^1\le \eta\le\eta_{1j}^2\}$ 
for $j=2$, $3$.  Let $A_1$ be a simply connected closed domain with smooth boundary
such that $|\Gamma_1^+\cap A_1|>0$ and $A_1\setminus D_1(\nu)\subset C_2\cup C_3$.
By Lemma~\ref{lem:orth-analytic},
$h_1$ is analytic on $\operatorname{int}(A_1)$ and
converges to its boundary value in $L^2(\partial A_1)$.
Using \cite[Theorem~17.18]{Rudin} and
the Riemann mapping theorem \cite[Theorem~8.2]{Bell} as well as the identity theorem,
we have $h_1=0$ on $A_1$ if $h_1|_{\Gamma_1^+}=0$, and it follows that 
Thus we have $h_1\equiv0$ on $\overline{D_1(\nu)}\cap\overline{\C_+}$.
Hence
it follows from Remark~\ref{rem:adjointef} that
$\la f,g_{12}^{*,+}(\cdot)\ra=0$ on $(-\eta_{12}^*(\a),0])$
and $\la f,g_{13}^{*,+}(\cdot+i\nu)\ra=0$ on $(-\eta_{13}^*(\a,\nu),0]$.
We can prove the rest in the same way.
\end{proof}

Lemma~\ref{lem:orth-condition-identity} and Remark~\ref{rem:wavemotion}
tell us that $f\in\calX_\mY(\nu)$ if $f$ is orthogonal to adjoint continuous
eigenfunctions associated with outgoing (or incoming) 
modulations of line solitons.
\begin{corollary}
  \label{cor:orth-condition-identity}
Assume \eqref{ass:alpha23} and \eqref{ass:X13nu}.
Let $\eta_0$ be a positive number.
\begin{itemize}
\item[\rm{(i)}]
If $f\in\calX(\nu)$ and  for $\eta_R\in[-\eta_0,\eta_0]$,
\begin{equation*}
\la f, g_{12}^{*,+}(\eta_R)\ra= \la f, g_{12}^{*,-}(\eta_R)\ra=
\la f,g_{13}^{*,-}(\eta_R+i\nu)\ra=0\,,
\end{equation*}
then $f\in \calX_\mY(\nu)$.
\item[\rm{(ii)}] 
If $f\in\calX(\nu)$ and for $\eta_R\in[-\eta_0,\eta_0]$,
\begin{equation*}
\la f,g_{13}^{*,+}(\eta_R+i\nu)\ra
=\la f, g_{23}^{*,-}(\eta_R+i\eta_I^{23})\ra
= \la f, \tg_{23}^{*,-}(\eta_R+i\eta_I^{23})\ra=0\,,
\end{equation*}
then $f\in \calX_\mY(\nu)$.
\end{itemize}
\end{corollary}
\bigskip

\section{Nonexistence of eigenvalues}
\label{sec:nonexistence}
Using the Darboux transformations, we will prove that
$\mL_{\mY}$ has no eigenvalue in $\calX$ and that
$\mL_\mY^*$ has no eigenvalue other than $V(\beta)$and $\widehat{V}(\beta)$.
\par
Replacing $\pd_t$ with $\lambda$
in \eqref{eq:Darboux-4a} and \eqref{eq:Darboux-4b},
we obtain the Darboux transformations for the eigenvalue problem.
For $U=U(X,Y)$ and $V=V(X,Y)$,
\begin{align*}
& \widetilde{\mathcal{M}}_{\lambda,+}(U,V)
= \nabla_u \widetilde{M}_+[u_{\mY},v_{\mY}]U
+e^{-\lambda t}\nabla_v \widetilde{M}_+[u_{\mY},v_{\mY}](e^{\lambda t}V)
\\=&         3(-\pd_X+\pd_X^{-1}\pd_Y+2v_{\mY})U
\\ \quad &
+4(\lambda \pd_X^{-1}-b_1-b_2\pd_X^{-1}\pd_Y+\pd_X^2-3v_{\mY}\pd_X+3v_{\mY}^2)V\,,
\end{align*}
\begin{align*}
& \widetilde{\mathcal{M}}_{\lambda,-}(U,V)
  = \nabla_u \widetilde{M}_-[0,v_{\mY}]U
+e^{-\lambda t}\nabla_v \widetilde{M}_-[0,v_{\mY}](e^{\lambda t}V)
  \\=&   3(\pd_X+\pd_X^{-1}\pd_Y+2v_{\mY})U
  \\ \quad &
+ 4\left(\lambda \pd_X^{-1}-b_1-b_2\pd_X^{-1}\pd_Y+\pd_X^2+3v_{\mY}\pd_X+\frac{3}{2}u_{\mY}+3v_{\mY}^2\right)V\,,
\end{align*}
where  let $v_\mY=\pd_x\log\tau_\mY$.
Formally, $\widetilde{\mathcal{M}}_{\lambda,+}^*
=\widetilde{\mathcal{M}}_{-\bar{\lambda},-}$ since $u_{\mY}=2\pd_xv_{\mY}$.
If we substitute $v=v_{\mY}+e^{\lambda t}V(X,Y)$ into \eqref{eq:Miura-transform}
and linearize the resulting equations, we have
\begin{gather}
  \label{eq:LMiura-1}
  U_1=\nabla M_-(v_{\mY})V\,,\quad U_2=\nabla M_+(v_{\mY})V\,,
\\
\label{eq:LMiura-2}
\widetilde{\mathcal{M}}_{\lambda,-}(U_1,V)=0\,,\quad
\widetilde{\mathcal{M}}_{\lambda,+}(U_2,V)=0\,.
\end{gather}
\par

In the moving coordinates \eqref{eq:b1-b2Y}, 
the eigenvalue problem  $\mL_{\mY}U_2=\lambda U_2$ reads
\begin{equation}
  \label{eq:ef-u2}
(3\pd_Y^2-4b_2\pd_X\pd_Y+4\lambda \pd_X-\pd_X^2A)U_2=0\,,
\quad A=-\pd_X^2+4b_1-6u_{\mY}\,.
\end{equation}
The Darboux transformations \eqref{eq:LMiura-1} and \eqref{eq:LMiura-2}
map a solution of \eqref{eq:ef-u2} to that of
\begin{equation}
  \label{eq:ef-u1}
(3\pd_Y^2-4b_2\pd_X\pd_Y+4\lambda \pd_X+\pd_X^4-4b_1\pd_X^2)U_1=0\,.
\end{equation}
\begin{lemma}
  \label{lem:U1-U2}
If $U_1$ and $U_2$ satisfy \eqref{eq:LMiura-1} and \eqref{eq:LMiura-2}
and $U_2$ is a solution of \eqref{eq:ef-u2}, then $U_1$ is a solution of
\eqref{eq:ef-u1}.
\end{lemma}

Let $A_1=-\pd_X-2v_{\mY}$, $A_3=\pd_X-2v_{\mY}$ and
\begin{gather*}
A_2(\lambda)= -4\left\{\pd_X^2+(3v_{\mY}-b_2)\pd_X-b_1+\lambda\pd_X^{-1}
  +3v_{\mY}^2+\frac{3}{2}u_{\mY}-2b_2v_{\mY}\right\}\,,\\  
A_4(\lambda)= -4\left\{\pd_X^2-(3v_{\mY}-b_2)\pd_X-b_1+\lambda\pd_X^{-1}
  +3v_{\mY}^2-2b_2v_{\mY}\right\}\,,
\\
 \mathcal{A}_1(\lambda)=\begin{pmatrix}
A_1 & I \\ A_2(\lambda) & O
\end{pmatrix}\,,\quad
\mathcal{B}_1=\begin{pmatrix}I & O \\   -3A_1-4b_2  & 3I\end{pmatrix}\,,
\\
 \mathcal{A}_2(\lambda)=\begin{pmatrix}
A_3 & I \\ A_4(\lambda) & O
\end{pmatrix}\,, \quad
\mathcal{B}_2=\begin{pmatrix}I & O \\   -3A_3-4b_2  & 3I\end{pmatrix}\,,
\\
\mathbf{U_i}=\begin{pmatrix}  U_{i1} \\ U_{i2}\end{pmatrix}\,,
\quad \mathbf{V}=\begin{pmatrix}  V_1 \\ V_2\end{pmatrix}\,.
\end{gather*}
Darboux transformations \eqref{eq:LMiura-1} and \eqref{eq:LMiura-2} can be
rewritten as
\begin{gather}
  \label{eq:calA-B}
\mathcal{A}_1(\lambda)\mathbf{V}=\mathcal{B}_1\mathbf{U_1}\,,
\quad
\mathcal{A}_2(\lambda)\mathbf{V}=\mathcal{B}_2\mathbf{U_2}\,,
\end{gather}
where $U_{i1}=U_i$, $V_1=V$ and $U_{i2}=\pd_X^{-1}\pd_YU_i$,
$V_2=\pd_X^{-1}\pd_YV$.
Formally, $A_1^*=A_3$ and $A_2(\lambda)^*=A_4(-\bar{\lambda})$.
For the relationship between $A_2(\lambda)$, $A_4(\lambda)$
and the Lax operator for the KP-II, see Lemma~\ref{rem:Lax-Darboux}.

The following lemma tells us that
Darboux transformations \eqref{eq:LMiura-1} and \eqref{eq:LMiura-2}
are equivalent to \eqref{eq:calA-B} if $U_2=U_{21}$ is a solution of
\eqref{eq:ef-u2}. Note that $\ker(A_4(\lambda))=\{0\}$
by Lemmas~\ref{lem:A40-est}, \ref{lem:A4-est} and \ref{lem:A4*-est}.
\begin{lemma}
\label{lem:compatibilityU1U2}
Suppose that $\mathbf{U_j}=(U_{j1},U_{j2})^T$ $(j=1,\,2)$ and
$\mathbf{V}=(V_1,V_2)^T$ satisfy \eqref{eq:calA-B}.
Then
\begin{gather}
  \label{eq:calA-B-0}
U_{21}-U_{11}=2\pd_XV_1\,,\quad U_{22}-U_{12}=2\pd_XV_2\,,  
\\ \label{eq:calA-B-t0}
\widetilde{\mathcal{M}}_{\lambda,-}(U_{11},V_1)
=3(\pd_X^{-1}\pd_YU_{11}-U_{12})+4b_2(V_2-\pd_X^{-1}\pd_YV_1)\,,
\\ \label{eq:calA-B-t0'}
\widetilde{\mathcal{M}}_{\lambda,+}(U_{21},V_1)
=3(\pd_X^{-1}\pd_YU_{21}-U_{22})+4b_2(V_2-\pd_X^{-1}\pd_YV_1)\,.
\end{gather}
Moreover, if $\pd_YU_{21}=\pd_XU_{22}$, then
\begin{gather}
  \label{eq:calA-B-0'}
\pd_XU_{12}-\pd_YU_{11}=2\pd_X(\pd_YV_1-\pd_XV_2)\,,
\\  \label{eq:calA-B-0''}
A_4(\lambda)(\pd_YV_1-\pd_XV_2)=4(\lambda-\mL_\mY)U_{21}\,.
\end{gather}
\end{lemma}
Proof of Lemmas~\ref{lem:U1-U2} and \ref{lem:compatibilityU1U2}
will be given in Appendix~\ref{ap:lem5.1-5.2}.
\par
Now we will find solutions of $A_2(\lambda)V=0$
and $A_4(\lambda)V'=0$.
Let
\begin{gather*}
V'(\beta)=e^{\lambda(\beta)t}\tau_{\mY}\Phi^{0,*}(-i\beta)\,,
\quad
\widehat{V}_j(\beta)=(\beta-\k_2)
e^{\lambda(\beta)t}\Phi^1_j\Phi^{1,*}(-i\beta)\,,
\\
\beta_{*,\pm}=\frac{b_2\pm\sqrt{b_2^2+3b_1}}{3}\,,\quad
\beta_{\#,\pm}=\frac{b_2\pm2\sqrt{b_2^2+3b_1}}{3}\,.
\end{gather*}
Then we have the following.
\begin{lemma}
  \label{lem:kerA-solU}
\begin{itemize}
\item[\rm(i)] For any $\beta\in\C$ and $Y\in\R$,
  \begin{gather}
\label{eq:A2b}
A_2(\lambda(\beta))V(\beta)=0\,,\quad A_4(-\lambda(\beta))V'(\beta)=0\,,
\\
\label{eq:A2b'}
A_2(\lambda(\beta_{*,\pm}))\pd_\beta V(\beta_{*,\pm})=0\,,\quad
A_4(-\lambda(\beta_{*,\pm}))\pd_\beta V'(\beta_{*,\pm})=0\,.
\end{gather}
\item[\rm(ii)] 
For any $Y\in\R$ and $\beta\in\C\setminus\{\k_1,\k_2,\k_3\}$,
$V(\beta)\in L^2(\R;e^{2\a X}dX)$ if and only if $\k_1<\Re\beta+\a<\k_3$.
For any $\a$, $Y\in\R$ and $\beta\in\C$, $V'(\beta)\not\in L^2(\R;e^{2\a X}dX)$.
\item[\rm(iii)]
$A_2(0)V(\k_j)=0$ and $V(\k_j)=\Phi^1_j\Phi^{1,*}_j
\in L^2(\R;e^{2\a X}dX)$ for $j=1, 2, 3$
and any $Y\in\R$ if  $\k_1-\k_2<\a<\k_3-\k_2$.
\item[\rm(iv)]
  If $\lambda\ne\lambda(\beta_{*,\pm})$, then
\begin{equation}
  \label{eq:beta}
  \beta^3-b_2\beta^2-b_1\beta+\lambda-\k_1\k_2\k_3=0
\end{equation}
has $3$ distinctive roots. If $\lambda=\lambda(\beta_{*,\pm})$, then
\eqref{eq:beta} has a double root $\beta=\beta_{*,\pm}$ and
a simple root $\beta=\beta_{\#,\mp}$.
\item[\rm(v)]
If $\k_1-\beta_{*,+}<\a<\k_1-\beta_{\#,-}$, then
$V(\beta_{*,+})$, $\pd_\beta V(\beta_{*,+})\in L^2(\R;e^{2\a X}dX)$
and $V(\beta_{\#,-})\not\in L^2(\R;e^{2\a X}dX)$ for any $Y\in\R$.
If  $\k_3-\beta_{\#,+}<\a<\k_3-\beta_{*,-}$, then
$V(\beta_{*,-})$, $\pd_\beta V(\beta_{*,-})\in L^2(\R;e^{2\a X}dX)$
and $V(\beta_{\#,+})\not\in L^2(\R;e^{2\a X}dX)$ for any $Y\in\R$.
\end{itemize}
\end{lemma}
\begin{remark}
\label{rem:betas} 
Note that $\beta_{*,\pm}\ne \beta_{\#,\mp}$ since
$b_2^2+3b_1=\sum_{1\le i<j\le3}(\k_j-\k_i)^2/2$ and $\k_1<\k_2<\k_3$.
Moreover, 
$\beta_{\#,-}< \k_1<\beta_{*,-}<\k_2<\beta_{*,+}<\k_3<\beta_{\#,+}$,
$\beta_{\#,+}-\k_3<\beta_{*,-}-\k_1$ and $\k_1-\beta_{\#,-}<\k_3-\beta_{*,+}$.
Let $\beta_j(\lambda)$ ($j=1, 2, 3$) be
solutions of \eqref{eq:beta} satisfying
$\Re \beta_1\le \Re \beta_2\le \Re\beta_3$. Then
\begin{gather*}
\beta_1(\lambda)=\beta_2(\lambda)=\beta_{*,-}\,,\quad
\beta_3(\lambda)=\beta_{\#,+}
\quad\text{if $\lambda=\lambda(\beta_{*,-})$,}
  \\
\beta_1(\lambda)=\beta_{\#,-}\,,\quad
\beta_2(\lambda)=\beta_3(\lambda)=\beta_{*,+}
\quad\text{if $\lambda=\lambda(\beta_{*,+})$.}
\end{gather*}
\end{remark}
We may regard $V(\beta)$ and $\widehat{V}_j(\beta)$
as solutions of a $1+1$-dimensional PDE.
\begin{lemma}
 \label{lem:sol-l-ad}
$\mL_{\mY}\pd_XV(\beta)=\lambda(\beta)\pd_XV(\beta)$,
$\mL_{\mY}\pd_X\widehat{V}_j(\beta)=-\lambda(\beta)\pd_X\widehat{V}_j(\beta)$ and
  \begin{gather}
\label{eq:efVb}
(3\pd_Y^2-4b_2\pd_X\pd_Y-\pd_XA\pd_X)V(\beta)+4\lambda(\beta) \pd_XV(\beta)
=0\,,
\\ \label{eq:efhatVb}
(3\pd_Y^2-4b_2\pd_X\pd_Y-\pd_XA\pd_X)\widehat{V}_j(\beta)
-4\lambda(\beta) \pd_X\widehat{V}_j(\beta)=0\,,
\end{gather}
and $(3\pd_Y^2-4b_2\pd_X\pd_Y-\pd_XA\pd_X)V(\k_j)=0$.
In addition,
\begin{gather}
  \label{eq:efVb'}
  (3\pd_Y^2-4b_2\pd_X\pd_Y-\pd_XA\pd_X)\pd_\beta V(\beta_{*,\pm})
+4\lambda(\beta_{*,\pm})\pd_X\pd_\beta V(\beta_{*,\pm})
=0\,,\quad
\\ \label{eq:efhatVb'}
(3\pd_Y^2-4b_2\pd_X\pd_Y-\pd_XA\pd_X)\pd_\beta\widehat{V}_j(\beta_{*,\pm})
-4\lambda(\beta_{*,\pm}) \pd_X\pd_\beta\widehat{V}_j(\beta_{*,\pm})=0\,.
\end{gather}
\end{lemma}

\begin{proof}[Proof of Lemmas~\ref{lem:kerA-solU} and \ref{lem:sol-l-ad}]
First, we remark that $V(\beta)$, $\widehat{V}_j(\beta)$
and $V'(\beta)$ depend only on $X$ and $Y$.
Since $L_{\mY}\Phi^1=B_{\mY}\Phi^1=0$
and $L_0^*\Phi^{0,*}=B_0^*\Phi^{0,*}=0$,
it follows from \eqref{eq:A2-Lax} and \eqref{eq:A4-Lax} that
$A_2(\lambda(\beta))V(\beta)=0$ and $A_4(-\lambda(\beta))V'(\beta)=0$.
By \eqref{eq:b1-b2Y}, we have $\beta_j(0)=\k_j$
and $V(\k_j)=\Phi^1_j\Phi^{1,*}_j$ for $j=1$, $2$, $3$.
\par
Suppose that $\lambda=\lambda(\beta_{*,\pm})$.
Then $\beta=\beta_{*,\pm}$ is a double
root of \eqref{eq:beta} since $\lambda'(\beta_{*,\pm})=0$,
and the other root of \eqref{eq:beta} is $\beta=\beta_{\#,\mp}$.
Differentiating \eqref{eq:A2b}
with respect to $\beta$ and substituting $\beta=\beta_{*,\pm}$ into
the resulting equation, we have \eqref{eq:A2b'}.
\par
By Lemma~\ref{lem:prodJdJ},
$e^{\lambda(\beta)t}V(\beta)=\Phi^1(-i\beta)\Phi^{1,*}_1$ is
is a solution of $\pd_tw+\mL_{\mY}^*w=0$ and 
$e^{\lambda(\beta)t}\pd_XV(\beta)$ is a solution of $\pd_tv=\mathcal{L}_Yv$.
Since $V(\beta)$ depends only on $X=x-b_1t$ and $Y=y-b_2t$, we have \eqref{eq:efVb}.
Differentiating \eqref{eq:efVb} with respect to $\beta$ and substituting
$\beta=\beta_{*,\pm}$ into the resulting equation, 
we have  \eqref{eq:efVb'} because $\lambda'(\beta_{*,\pm})=0$.
Similarly, we have \eqref{eq:efhatVb} and \eqref{eq:efhatVb'}.
\par
Suppose that $\beta\in\C\setminus\{\k_1,\k_2,\k_3\}$. Then
$\Phi^1(-i\beta)=O(e^{\beta x})$ and
$\Phi^{0,*}(-i\beta)=O(e^{-\beta x})$ as $x\to\pm\infty$.
Since $\tau_{\mY}=O(e^{\k_1x})$ as $x\to-\infty$ and
$\tau_{\mY}=O(e^{\k_3x})$ as $x\to\infty$,
we have 
$e^{\a X}V(\beta)\in L^2(\R_X)$ if and only if $\k_1<\a+\Re \beta<\k_30$,
and $e^{\a X}V'(\beta)\not\in L^2(\R_X)$ for any $\a$.
Thus we prove (ii). 
We have $e^{\a X}\pd_\beta V(\beta_{*,\pm})\in L^2(\R_X)$
if and only if $\a+\beta_{*,\pm}\in (\k_1,\k_3)$ and (v)
follows from Remark~\ref{rem:betas}.
\par
Since $V(\k_j)=\Phi^1_j\Phi^{1,*}_j$,
it follows from \eqref{def:Phi1s} that for each $Y\in\R$,
\begin{equation}
  \label{eq:13}
\left\{
  \begin{gathered}
V(\k_1)=O(e^{\theta_1-\theta_3})\,,\enskip
|V(\k_2)|+|V(\k_3)|=O(e^{\theta_2-\theta_3})  \text{ as $X\to\infty$,}
\\
|V(\k_1)|+|V(\k_2)|=O(e^{\theta_2-\theta_1})\,,\enskip
V(\k_3)=O(e^{\theta_3-\theta_1})\text{ as $X\to-\infty$.}      
  \end{gathered}\right.
\end{equation}
Thus we have (iii). This completes the proof of
Lemmas~\ref{lem:kerA-solU} and \ref{lem:sol-l-ad}.
\end{proof}

We have a conserved quantity for $U$ and $U^*$ satisfying
$\mL_{\mY}U=\lambda U$ and $\mL_{\mY}^*U^*=\lambda U^*$.

\begin{lemma}
  \label{lem:bfUU*}
Let $U$ be a solution of \eqref{eq:ef-u2} and $U^*$ be a solution
of the equation
$(3\pd_Y^2-4b_2\pd_X\pd_Y-4\lambda\pd_X-\pd_XA\pd_X)U^*=0$.
Suppose that for an $\a>0$,
 \begin{gather*}
\mathbf{U}=\!(U,\pd_X^{-1}\pd_YU)^T
\in C^1(\R_Y;L^2(\R;e^{2\a X}dX))\,,
\\
\mathbf{U}^*=(-\frac{4}{3}b_2U^*+\pd_X^{-1}\pd_YU^*,U^*)^T
\in C^1(\R_Y;L^2(\R;e^{-2\a X}dX))\,.
\end{gather*}
Then
$$\la \mathbf{U},\mathbf{U^*}\ra:=\int_\R\left\{
U\left(-\frac{4}{3}b_2U^*+\pd_X^{-1}\pd_YU^*\right)
+(\pd_X^{-1}\pd_YU)U^*\right\}\,dX$$
does not depend on $Y$.
\end{lemma}
\begin{proof}
  Let $\mathbf{A}(\lambda)=
  \begin{pmatrix}
    0 & \pd_X \\ \frac{1}{3}\pd_XA-\frac{4}{3}\lambda  & \frac{4}{3}b_2\pd_X
  \end{pmatrix}$.
Since $\pd_Y\mathbf{U}=\mathbf{A}(\lambda)\mathbf{U}$ and
$\pd_Y\mathbf{U^*}=-\mathbf{A}(\lambda)^*\mathbf{U^*}$,
\begin{align*}
  \frac{d}{dY}\la \mathbf{U},\mathbf{U^*}\ra
  =& \la \pd_Y\mathbf{U},\mathbf{U^*}\ra
     +\la \mathbf{U},\pd_Y\mathbf{U^*}\ra
     =0\,.
\end{align*}
\end{proof}

We will investigate the correspondence between $\mathbf{U_1}$ and $\mathbf{U_2}$
satisfying \eqref{eq:calA-B}. To begin with,  we investigate solvability of
$A_4(0)U=F$.
\begin{lemma}
 \label{lem:A40-est}
Let $0<|\a|<\min\{\k_2-\k_1,\k_3-\k_2\}$ and $A_2(0)$ and $A_4(0)$ be closed
operators on $L^2(\R;e^{2\a X}dX)$.
Then for every $Y\in\R$,
  $\ker(A_4(0))=\ker(A_2^*(0))=\{0\}$ and
  $\ker(A_2(0))=\ker(A_4^*(0))=\spann\{V(\k_1),V(\k_2)\}$.
\\
  Moreover, if $F\in L^2(\R;e^{2\a X}dX)$ and
  \begin{equation}
    \label{eq:orthF0}
    \int_\R FV(\k_j)\,dX=0\quad\text{for $j=1$, $2$,}
  \end{equation}
then $A_4(0)U=F$ has a solution satisfying
\begin{equation}
  \label{eq:U-Fest0}
\sum_{0\le i\le2}\|\pd_X^iU(\cdot,Y)\|_{L^2(\R;e^{2\a X}dX)}\le C
 \|F\|_{L^2(\R;e^{2\a X}dX)}\,,
    \end{equation}
where $C$ is a constant that is independent of $Y$ and $F$.
\end{lemma}
\begin{remark}
  \label{rem:A40-est}
By \eqref{eq:4},  we have $V(\k_1)+V(\k_2)+V(\k_3)=0$.
\end{remark}
\begin{proof}[Proof of Lemma~\ref{lem:A40-est}]
First, we remark that $A_4(0)$ is an ordinary differential
operator of the second order.
Since $\lambda(\k_1)=\lambda(\k_2)=0$ and $A_2(0)=A_4(0)^*$ (formally),
it follows from Lemma~\ref{lem:kerA-solU} that
$V'(\k_1)$ and $V'(\k_2)$ are linearly independent solutions of $A_4(0)U=0$
and that $\ker(A_4(0))=\{0\}=\ker(A_2(0)^*)=\{0\}$.
\par
Let $k_j^*(X,Y)=V(\k_j)(X,Y)=\Phi^1_j\Phi^{1,*}_j$.
By Lemma~\ref{lem:kerA-solU},
we have $k_j^*(\cdot,Y)\in L^2(\R;e^{\pm2\a X}dX)$ and
$\ker(A_2(0))=\ker(A_4(0)^*)=\spann\{k_1^*,\,k_2^*\}$
and $\ker(A_4(0)^*)=\spann\{k_1^*,\,k_2^*\}$ since $A_2(0)=A_4(0)^*$ formally.
We remark that $\sum_{1\le j\le 3}k_j^*=0$ and that
$\operatorname{Wr}(k_1^*,k_2^*)=\operatorname{Wr}(k_2^*,k_3^*)
=\operatorname{Wr}(k_3^*,k_1^*)$ thanks to \eqref{eq:4}.
\par
Let
\begin{align*}
k_j(X,Y)=&-\frac{k_j^*(X,Y)}{4\operatorname{Wr}(k_1^*,k_2^*)(X,Y)}
=-\frac14 
\frac{e^{-(\theta_1+\theta_2+\theta_3)}\tau_{\mY}^2\Phi^1_j}{(\k_1-\k_2)(\k_2-\k_3)(\k_3-\k_1)}\,,
\end{align*}
where $\operatorname{Wr}(k_1^*,k_2^*)=k_1^*\pd_Xk_2^*-k_2^*\pd_Xk_1^*$.
Note that
\begin{equation*}
\frac{d}{dX}\operatorname{Wr}(k_1^*,k_2^*)
=(b_2-3v_\mY)\operatorname{Wr}(k_1^*,k_2^*)\,.  
\end{equation*}
Since $A_4^*(0)k_1^*=A_4^*(0)k_2^*=A_4^*(0)k_3^*=0$,
we formally have $A_4(0)k_1=A_4(0)k_2=A_4(0)k_3=0$.
\par

By the variation of constants formula,
\begin{align*}
U_1(X,Y)=& k_1(X,Y)\int^\infty_X k_2^*(X',Y)F(X',Y)\,dX'
\\ & -k_2(X,Y)\int^\infty_X k_1^*(X',Y)F(X',Y)\,dX'\,,
\end{align*}
\begin{align*}
U_2(X,Y)=& k_1(X,Y)\int_{-\infty}^X k_3^*(X',Y)F(X',Y)\,dX'
\\ & +k_3(X,Y)\int^\infty_X k_1^*(X',Y)F(X',Y)\,dX'\,,
\end{align*}
\begin{align*}
U_3(X,Y) =& k_3(X,Y)\int_{-\infty}^X k_2^*(X',Y)F(X',Y)\,dX'
\\ & -k_2(X,Y)\int_{-\infty}^X k_3^*(X',Y)F(X',Y)\,dX'
\end{align*}
are solutions of  $A_4(0)U(\cdot,Y)=F(\cdot,Y)$.
Indeed,
\begin{align}
\label{eq:k1*-est}
k_1^*(X',Y)=& O\left(\min\left(\frac{e^{\theta_1'}}{\tau_{\mY}'},
\frac{e^{\theta_2'}+e^{\theta_3'}}{\tau_{\mY}'}\right)\right)
  \\=& O\left(\min(e^{\theta_1'-\theta_2'},e^{\theta_1'-\theta_3'},
e^{\theta_2'-\theta_1'}+e^{\theta_3'-\theta_1'}) \right)\,,
\notag
\end{align}
\begin{align}
\label{eq:k2*-est}  
  k_2^*(X',Y)=& O\left(\min\left(\frac{e^{\theta_2'}}{\tau_{\mY}'},
\frac{e^{\theta_1'}+e^{\theta_3'}}{\tau_{\mY}'}\right)\right)
  \\=& O\left(\min(e^{\theta_2'-\theta_1'},e^{\theta_2'-\theta_3'},
e^{\theta_1'-\theta_2'}+e^{\theta_3'-\theta_2'} )\right)\,, \notag
\end{align}
\begin{align}
\label{eq:k3*-est}
k_3^*(X',Y)=& O\left(\min\left(\frac{e^{\theta_3'}}{\tau_{\mY}'},
\frac{e^{\theta_1'}+e^{\theta_2'}}{\tau_{\mY}'}\right)\right)
  \\=& O\left(\min(e^{\theta_3'-\theta_1'},e^{\theta_3'-\theta_2'},
e^{\theta_1'-\theta_3'}+e^{\theta_2'-\theta_3'}) \right)\,,\notag
\end{align}
where $\theta_j'=\k_jx'+\k_j^2y-\k_j^3t$, $X'=x'-b_1t$
and $z_{ij}'=X'+2a_{ij}Y$, and $U_1$, $U_2$ and $U_3$ are finite
provided $F(\cdot,Y)\in L^2(\R;e^{2\a X}dX)$.
\par
Since $k_1^*+k_2^*+k_3^*=0$, $k_1+k_2+k_3=0$ and
\begin{equation*}
\int_{-\infty}^X k_j^*(X',Y)F(X',Y)\,dX'=-\int^\infty_X k_j^*(X',Y)F(X',Y)\,dX'
\end{equation*}
for $F$ satisfying \eqref{eq:orthF0}, we have
$U_1=U_2=U_3$.

\par
First, we will estimate $U_1(X,Y)$ on
$E_1=\{(X,Y)\mid \theta_3\ge \max(\theta_1,\theta_2)\}$.
If $(X,Y)\in E_1$,
\begin{align*}
& k_1=O\left(\tau_{\mY}(e^{-\theta_2}+e^{-\theta_3})\right)
 =O(e^{\theta_3-\theta_2})\,,
\quad
k_2=O\left(\tau_{\mY}(e^{-\theta_1}+e^{-\theta_3})\right)
     =O(e^{\theta_3-\theta_1})\,.
\end{align*}
Since $k_1^*= O(e^{\theta_1'-\theta_3'})$ and $k_2^*= O(e^{\theta_2'-\theta_3'})$
by \eqref{eq:k1*-est} and \eqref{eq:k2*-est},
\begin{align*}
\left|k_2(X,Y)k_1^*(X',Y)\right| \lesssim &
e^{\theta_3-\theta_1-(\theta_3'-\theta_1')}
 \le   C_3e^{2\sqrt{c_{13}}(X-X')}\,,
\\
\left|k_1(X,Y)k_2^*(X',Y)\right| \lesssim &
e^{\theta_3-\theta_2-(\theta_3'-\theta_2')}\le  C_3e^{2\sqrt{c_{23}}(X-X')}\,,
\end{align*}
and
\begin{equation*}
  \|U_1(\cdot,Y)\|_{L^2(E_{1Y};e^{2\a X}dX)}\le C_4
  \|F(\cdot,Y)\|_{L^2(\R;e^{2\a X}dX)}\,,
\end{equation*}
where $E_{1Y}=\{X\mid (X,Y)\in E_1\}$ and
$C_3$ and $C_4$ are constants that do not depend on $X$, $Y$ and $F$.
\par
Next, we will estimate $U_2(X,Y)$ on
$E_2=\{(X,Y)\mid \theta_2\ge \max(\theta_1,\theta_3)\}$.
Suppose that $(X,Y)\in E_2$. Then
\begin{equation}
  \label{eq:k1,k3-est2}
k_1(X,Y)=O(e^{\theta_2-\theta_3})\,,\quad k_3(X,Y)=O(e^{\theta_2-\theta_1})\,.
\end{equation}
Since $k_1^*(X',Y)=O(e^{\theta_1'-\theta_2'})$ and $k_3^*(X',Y)=O(e^{\theta_3'-\theta_2'})$
by \eqref{eq:k1*-est} and \eqref{eq:k3*-est},
\begin{align*}
  |k_3(X,Y)k_1^*(X',Y)|=&O(e^{(\theta_2-\theta_1)-(\theta_2'-\theta_1')})
 \le   C_3'e^{2\sqrt{c_{12}}(X-X')}\,,
\\
  |k_1(X,Y)k_3^*(X',Y)|=&O(e^{(\theta_2-\theta_3)-(\theta_2'-\theta_3')})
  \le   C_3'e^{-2\sqrt{c_{23}}(X-X')}\,,
\end{align*}
and
\begin{equation*}
  \|U_2(\cdot,Y)\|_{L^2(E_{2Y};e^{2\a X}dX)}\le C_4'
  \|F(\cdot,Y)\|_{L^2(\R;e^{2\a X}dX)}\,,
\end{equation*}
where $E_{2Y}=\{X\mid (X,Y)\in E_2\}$ and $C_3'$ and $C_4'$ are constants
that do not depend on $X$, $Y$ and $F$.
\par
Finally, we will estimate $U_3(X,Y)$ on
$E_3=\{(X,Y)\mid \theta_1\ge\max(\theta_2,\theta_3)\}$.
Suppose that $(X,Y)\in E_3$. Then
\begin{gather}
  \label{eq:k2,k3-est2}
  k_2(X,Y)=O(e^{\theta_1-\theta_3})\,,\quad
  k_3(X,Y)=O(e^{\theta_1-\theta_2})\,.
\end{gather}
Since $k_2^*(X',Y)=O(e^{\theta_2'-\theta_1'})$ and
$k_3^*(X',Y)=O(e^{\theta_3'-\theta_1'})$ by \eqref{eq:k2*-est} and \eqref{eq:k3*-est},
we have
\begin{gather*}
\left|k_2(X,Y)k_3^*(X',Y)\right| \lesssim 
e^{(\theta_1-\theta_3)-(\theta_1'-\theta_3')}
 \le   C_3''e^{2\sqrt{c_{13}}(X'-X)}\,,
\\
\left|k_3(X,Y)k_2^*(X',Y)\right| \lesssim 
e^{(\theta_1-\theta_2)-(\theta_1'-\theta_2')}
\le  C_3''e^{2\sqrt{c_{12}}(X'-X)}\,,
\end{gather*}
and
\begin{equation*}
  \|U_3(\cdot,Y)\|_{L^2(E_{3Y};e^{2\a X}dX)}\le C_4''
  \|F(\cdot,Y)\|_{L^2(\R;e^{2\a X}dX)}\,,
\end{equation*}
where $E_{3Y}=\{X\mid (X,Y)\in E_3\}$ and $C_3''$ and $C_4''$ are constants
that do not depend on $X$, $Y$ and $F$.
Since $U_1=U_2=U_3$ and $\R=\cup_{1\le j\le 3}E_{jY}$ for every $Y\in \R$,
we have \eqref{eq:U-Fest0} with $i=0$. We can estimate $\pd_XU_j(X,Y)$
in the same way. Combining these with the fact that $U_j$ are solutions of
$A_4(0)U_j=F$, we have \eqref{eq:U-Fest0}.
Thus we complete the proof.
\end{proof}

Next, we will investigate the correspondence between
$\mathbf{U_1}$ and $\mathbf{U_2}$ satisfying \eqref{eq:calA-B}
when $\lambda=0$.
Let $V_1=V(\k_1)$, $V_3=V(\k_3)$ and
\begin{gather*}
\mathbf{W_1}(0)=\left(\pd_XV_1,\pd_YV_1\right)^T\,,\quad
\mathbf{W_1^*}(0)=
\left(-\frac{4}{3}b_2V_1+\pd_X^{-1}\pd_YV_1,V_1\right)^T\,,
\\
\mathbf{W_2}(0)=\left(\pd_XV_3,\pd_YV_3\right)^T\,,\quad
\mathbf{W_2^*}(0)
=\left(-\frac{4}{3}b_2V_3+\pd_X^{-1}\pd_YV_3,V_3\right)^T\,.
\end{gather*}

\begin{lemma}
\label{lem:calA20-est}
Assume that $0<|\a|<\min\{\k_2-\k_1,\k_3-\k_2\}$. Let $\mathcal{A}_1(0)$,
$\mathcal{A}_2(0)$, $\mathcal{B}_1$ and $\mathcal{B}_2$ be closed operators on
$L^2(\R;e^{2\a X}dX)$.
Then $\ker(\mathcal{A}_2(0))=\ker(\mathcal{B}_1)=\ker(\mathcal{B}_2)=\{0\}$ and
$\ker(\mathcal{A}_1(0))=\spann\{\pd_X^{-1}\mathbf{W_1}(0), \pd_X^{-1}\mathbf{W_2}(0)\}$.
\par
Moreover, if $\la \mathbf{U_2},\mathbf{W_1^*(0)}\ra=\la \mathbf{U_2},\mathbf{W_2^*(0)}\ra=0$,
there exists a unique solution $\mathbf{V}$ of
$\mathcal{A}_2(0)\mathbf{V}=\mathcal{B}_2\mathbf{U_2}$ satisfying
\begin{equation}
\label{eq:calA12-est}
  \|\mathbf{V}(\cdot,Y)\|_{(H^2\times H^1)(\R;e^{2\a X}dX)}
  \le C\|\mathbf{U_2}(\cdot,Y)\|_{(H^1\times L^2)(\R;e^{2\a X}dX)}\,,
\end{equation}
where $C$ is a constant that does not depend on $Y$ and $\mathbf{U_2}$.
\end{lemma}
 \begin{lemma}
\label{lem:kermap0U1-U2}
Let $\a$, $\mathcal{A}_1(0)$, $\mathcal{A}_2(0)$, $\mathcal{B}_1$ and
$\mathcal{B}_2$ be  as in Lemma~\ref{lem:calA20-est}. Suppose that
$\mathcal{A}_1(0)\mathbf{V}=\mathcal{B}_1\mathbf{U_1}$,
$\mathcal{A}_2(0)\mathbf{V}=\mathcal{B}_2\mathbf{U_2}$.
Then $\mathbf{U_1}=0$ if $\mathbf{U_2}=0$ and
$\mathbf{U_2}\in \spann\{\mathbf{W_1}(0),\mathbf{W_2}(0)\}$ if $\mathbf{U_1}=0$.
\end{lemma}

\begin{proof}[Proof of Lemma~\ref{lem:calA20-est}]
By the definitions, it is clear that
$\ker(\mathcal{B}_1)=\ker(\mathcal{B}_2)=\{0\}$.
If $\mathbf{V}=(V_1,V_2)^T\in\ker(\mathcal{A}_2(0))$,
then $V_2=-A_3V_1$ and $A_4(0)V_1=0$.
Since $\ker(A_4(0))=\{0\}$ by Lemma~\ref{lem:A40-est},
we have $V_1=V_2=0$.
\par
If $\mathbf{V}=(V_1,V_2)^T\in\ker(\mathcal{A}_1(0))$,
then $V_2=-A_1V_1$ and $A_2(0)V_1=0$.
Since $\ker(A_2(0))=\spann\{V(\k_1),V(\k_2)\}$ by Lemma~\ref{lem:A40-est}
and
\begin{equation}
  \label{eq:A1-Vk}
(A_1+\pd_X^{-1}\pd_Y)V(\k_j)=\nabla M_-(v_{\mY})V(\k_j)=0  
\end{equation}
by Lemma~\ref{cl:kerM-P}, we have
$\ker(\mathcal{A}_1(0))
=\spann\{\pd_X^{-1}\mathbf{W_1}(0),\pd_X^{-1}\mathbf{W_2}(0)\}$.
\par
Suppose that $\mathbf{U_2}=(U_{21},U_{22})^T$ and $\mathbf{V}=(V_1,V_2)^T$
satisfy $\mathcal{A}_2(0)\mathbf{V}=\mathcal{B}_2\mathbf{U_2}$
and that $\la\mathbf{U_2},\mathbf{W_j^*}(0)\ra
=0$ for $j=1$, $2$. Then
\begin{equation}
  \label{eq:pfcalA12-est2}
V_2=U_{21}-A_3V_1\,,\quad A_4(0)V_1=-(3A_3+4b_2)U_{21}+3U_{22}\,.  
\end{equation}
For $j=1$ and $3$,  we have $\mathcal{B}_2^*(0, V(\k_j))^T=3\mathbf{W_{(j+1)/2}}^*$ from
\eqref{eq:A1-Vk} and
\begin{align}
    \label{eq:pfcalA12-est3}
\la -(3A_3+4b_2)U_{21}+3U_{22},V(\k_j)\ra=&
3\la \mathbf{U_2}, \mathbf{W_{(j+1)/2}^*}\ra=0\,.
\end{align}
Combining \eqref{eq:pfcalA12-est2} and \eqref{eq:pfcalA12-est3} with
Lemma~\ref{lem:A40-est} and Remark~\ref{rem:A40-est}, we have \eqref{eq:calA12-est}.
Thus we complete the proof.
\end{proof}

\begin{proof}[Proof of Lemma~\ref{lem:kermap0U1-U2}]
Since $\ker(\mathcal{A}_2(0))=\{0\}$
and $\ker(\mathcal{B}_1)=\{0\}$ by Lemma~\ref{lem:calA20-est},
we have $\mathbf{V}=0$ and $\mathbf{U_1}=0$ if $\mathbf{U_2}=0$.
\par
  Suppose that $\mathbf{U_1}=0$. Then
  $\mathcal{A}_1(0)\mathbf{V}=0$ and it follows from
  Lemma~\ref{lem:calA20-est} that there exist $c_1$, $c_2\in\C$
such that
$\mathbf{V}=c_1\pd_X^{-1}\mathbf{W_1}(0)+c_2\pd_X^{-1}\mathbf{W_2}(0)$.
By \eqref{eq:A1-Vk} and the fact that $\mathcal{A}_1(0)\mathbf{V}=0$,
\begin{align*}
  \mathcal{A}_2(0)\mathbf{V}
  =& \sum_{j=1,2}c_j   \begin{pmatrix}
      (A_3-A_1)V(\k_j) \\ (A_4(0)-A_2(0))V(\k_j) \end{pmatrix}
  \\ =&
\sum_{j=1,2} c_j
\begin{pmatrix}
  2\pd_xV(\k_{2j-1}) \\
\{8(3v_{\mY}-b_2)\pd_X+6u_{\mY}\}V(\k_{2j-1})
\end{pmatrix}
\\=& 2\mathcal{B}_2(c_1\mathbf{W_1}+c_2\mathbf{W_2})\,.
\end{align*}
Since $\ker(\mathcal{B}_2)=\{0\}$,
$\mathbf{U_2}=2(c_1\mathbf{W_1}+c_2\mathbf{W_2})$.
This completes the proof of Lemma~\ref{lem:kermap0U1-U2}.
\end{proof}

Next, we will investigate solvability of $A_4(\lambda)U=F$ when $\lambda\ne0$.
\begin{lemma}
  \label{lem:A4-est}
Let $\a\in\R\setminus\{0\}$ and $A_4(\lambda)$ be a closed operator on $L^2(\R;e^{2\a X}dX)$. 
Assume that $\lambda\in\C\setminus\{0,-\lambda(\beta_{*,\pm})\}$ and that
$\Re\beta_j(-\lambda)-\a\ne \k_1$, $\k_2$, $\k_3$ for $j=1$, $2$, $3$.
Then $\ker(A_4(\lambda))=\{0\}$. Moreover, if
  \begin{equation}
\label{eq:secular-A4}
    \int_\R FV(\beta_j(-\lambda))\,dX=0\quad
    \text{when $\beta_j(-\lambda)\in(\k_1+\a,\k_3+\a)$,}
  \end{equation}
then $A_4(\lambda)U=F$ has a unique solution satisfying 
\begin{equation*}
\sum_{0\le i\le2}\|\pd_X^iU(\cdot,Y)\|_{L^2(\R;e^{2\a X}dX)}\le C
 \|F\|_{L^2(\R;e^{2\a X}dX)}\,,
    \end{equation*}
where $C$ is a constant that is independent of $Y$ and $F$.
\end{lemma}

\begin{lemma}
  \label{lem:A4*-est}
  \begin{itemize}
  \item[\rm{(i)}]
Let $\a\in(\beta_{\#,-}-\k_1,\beta_{*,+}-\k_2)\setminus\{0\}$ and
$A_4(-\lambda(\beta_{*,+}))$ be a closed operator on $L^2(\R;e^{2\a X}dX)$.
Then $\ker(A_4(-\lambda(\beta_{*,+})))=\{0\}$. Moreover, if
  \begin{equation}
\label{eq:ass-orth1}
\int_\R FV(\beta_{*,+})\,dX=\int_\R F\pd_\beta V(\beta_{*,+})\,dX=0\,,
  \end{equation}
then $A_4(-\lambda(\beta_{*,+}))U=F$ has a unique solution satisfying 
\begin{equation*}
\sum_{0\le i\le2}\|\pd_X^iU(\cdot,Y)\|_{L^2(\R;e^{2\a X}dX)}\le C
 \|F\|_{L^2(\R;e^{2\a X}dX)}\,,
    \end{equation*}
where $C$ is a constant that is independent of $Y$ and $F$.
\item[\rm{(ii)}]
Let $\a\in(\beta_{*,-}-\k_2,\beta_{\#,+}-\k_3)\setminus\{0\}$ and 
$A_4(-\lambda(\beta_{*,-}))$ be a closed operator on $L^2(\R;e^{2\a X}dX)$.
Then $\ker(A_4(-\lambda(\beta_{*,-})))=\{0\}$. Moreover, if
  \begin{equation*}
\int_R FV(\beta_{*,-})\,dX=\int_R F\pd_\beta V(\beta_{*,-})\,dX=0\,,
  \end{equation*}
then $A_4(-\lambda(\beta_{*,-}))U=F$ has a unique solution satisfying 
\begin{equation*}
\sum_{0\le i\le2}\|\pd_X^iU(\cdot,Y)\|_{L^2(\R;e^{2\a X}dX)}\le  C
 \|F\|_{L^2(\R;e^{2\a X}dX)}\,,
    \end{equation*}
where $C$ is a constant that is independent of $Y$ and $F$.
  \end{itemize}
\end{lemma}

\begin{proof}[Proof of Lemmas~\ref{lem:A4-est} and \ref{lem:A4*-est}]
Let $\tb_j(\lambda)=\beta_j(-\lambda)$.
Since $\pd_XA_4(\lambda)$ is an ordinary differential operator of the third order,
it follows from Lemma~\ref{lem:kerA-solU} that $\{V'(\tb_j(\lambda))\}_{1\le j\le 3}$
is a fundamental system of $\pd_XA_4(\lambda)u=0$.
We have $\ker(A_4(\lambda))=\{0\}$ and $\ker(\mathcal{A}_2(\lambda))=\{0\}$
for every $\lambda\in\C$ since any linear combinations of $V'(\beta_j(-\lambda))$
do not belong to $L^2(\R;e^{2\a X}dX)$.
\par

Let us construct a fundamental solution of $A_4(\lambda)$. Let
\begin{gather*}
a_j(\lambda)=-\frac14\prod_{1\le k\le 3,\, k\ne j}
\{\tb_k(\lambda)-\tb_j(\lambda)\}^{-1}\,,
\\   
K_0(x,\lambda)=
\begin{cases}
\sum_{\Re\tb_j(\lambda)>\k_1+\a}a_j(\lambda)e^{-\tb_j(\lambda)x}
  &\quad\text{if $x>0$,}  
\\
-\sum_{\Re\tb_j(\lambda)<\k_1+\a}a_j(\lambda)e^{-\tb_j(\lambda)x}
 & \quad\text{if $x<0$.}
\end{cases}
\end{gather*}
Then $K_0(\cdot,\lambda)$ is $C^1(\R)\cap C^\infty(\R\setminus\{0\})$ and
\begin{gather}
\notag
\pd_x^2K_0(+0,\lambda)-\pd_x^2K_0(-0,\lambda)=-\frac14\,,
\\ \label{eq:K0delta}
4(-\pd_x^3-b_2\pd_x^2+b_1\pd_x-\k_1\k_2\k_3-\lambda) K_0(x,\lambda)=\delta\,.
\end{gather}
Let $X=x-b_1t$, $X'=x'-b_1t$, $Y=y-b_2t$ and
\begin{equation*}
\tau_{Y,0}=e^{\k_1\k_2\k_3t}\tau_{\mY}=\sum_{1\le i \le 3}e^{\k_iX+\k_i^2Y}\,,\quad
\tau_{Y,0}'=e^{\k_1\k_2\k_3t}\tau_{\mY}'=\sum_{1\le i \le 3}e^{\k_iX'+\k_i^2Y}\,,
\end{equation*}
where $\tau_{\mY}'=\sum_{1\le i \le3}e^{\theta_i'}$. It follows from
\eqref{eq:A4-Lax} and \eqref{eq:K0delta} that
\begin{gather*}
K(X,X',Y,\lambda)=\frac{\tau_{Y,0}}{\tau_{Y,0}'}K_0(X-X',\lambda)
\end{gather*}
is a fundamental solution of $\pd_XA_4(\lambda)$. That is,
\begin{equation}
\label{eq:K-delta}
  A_4(\lambda)\int_\R K(X, X',Y,\lambda)\pd_{X'}f(X')\,dX'=f(X)\,.
\end{equation}
\par
Now we will estimate
\begin{multline*}
U(X,Y)=-\int_\R \pd_{X'}K(X,X',Y,\lambda)f(X')\,dX'
\\=
\begin{cases}
-\sum_{\Re\tb_j(\lambda)>\k_1+\a}a_j(\lambda)V'(\tb_j(\lambda))(X,Y)V(\tb_j(\lambda))(X',Y)
  &\quad\text{if $X>X'$,}  
\\
\sum_{\Re\tb_j(\lambda)<\k_1+\a}a_j(\lambda)V'(\tb_j(\lambda))(X,Y)V(\tb_j(\lambda))(X',Y)
 & \quad\text{if $X<X'$.}
\end{cases}
\end{multline*}
We remark that by \eqref{eq:prod-PhiPhi*j}, 
\begin{equation}
\label{eq:prod-VV'}
  \pd_{X'}\left(\frac{e^{\tb_j(\lambda)X'+\tb_j(\lambda)^2Y}}{\tau_{Y,0}'}\right)
  =V(\tb_j(\lambda))(X',Y)\,.
\end{equation}
Suppose that $Y\ge0$. Since $\theta_2\le \max(\theta_1,\theta_3)$,
\begin{align*}
0<\frac{\tau_{Y,0}}{\tau_{Y,0}'} \le &
2\frac{e^{\theta_1}+e^{\theta_3}}{e^{\theta_1'}+e^{\theta_3'}}
\\  =&
2e^{a_{13}(X-X')}
\cosh\sqrt{c_{13}}(X+2a_{13}Y)\sech\sqrt{c_{13}}(X'+2a_{13}Y)
\\ \le & 4e^{a_{13}(X-X')+\sqrt{c_{13}}|X-X'|}=
         \begin{cases}
            4e^{\k_3(X-X')} & \text{ if $X\ge X'$,}
\\          4e^{\k_1(X-X')} & \text{ if $X\le X'$.}           
         \end{cases}
\end{align*}
If $\Re\tb_j(\lambda)<\k_1+\a$,
\begin{align*}
& \sup_{Y\ge0} \left\|e^{\a X}\int_X^\infty
\pd_{X'}\left(e^{-\tb_j(\lambda)(X-X')}
  \frac{\tau_{Y,0}}{\tau_{Y,0}'}\right)f(X')\,dX'\right\|_{L^2(\R)}
\\  \lesssim &
\sup_{Y\ge 0}\left\|\int_X^\infty
e^{\{\a+\k_1-\Re\tb_j(\lambda)\}(X-X')}
e^{\a X'}|f(X')| \,dX'\right\|_{L^2(\R)}
\\ \lesssim &
\frac{1}{\a+\k_1-\Re\tb_j(\lambda)}\|e^{\a X}f\|_{L^2(\R)}\,.
\end{align*}
If $\Re\tb_j(\lambda)>\k_3+\a$,
\begin{align*}
& \sup_{Y\ge 0}\left\|e^{\a X} \int^X_{-\infty}\pd_{X'}\left(
e^{\{-\tb_j(\lambda)\}(X-X')}\frac{\tau_{Y,0}}{\tau_{Y,0}'}\right)f(X')\,dX'
\right\|_{L^2(\R)}
  \\ \lesssim &
\sup_{Y\ge 0}\left\|\int^X_{-\infty}
e^{\{\a+\k_3-\Re\tb_j(\lambda)\}(X-X')}e^{\a X'}|f(X')|\,dX'
\right\|_{L^2(\R)}
\\ \lesssim & 
 \frac{1}{\Re\tb_j(\lambda)-\k_3-\a}\|e^{\a X}f\|_{L^2(\R)}\,.
\end{align*}
Next, we consider the case where $\tb_j(\lambda)\in (\k_1+\a,\k_3+\a)$.
If $z_{13}=X+2a_{13}Y$ and $z_{13}'=X'+2a_{13}Y$ have the same sign and $|z_{13}'|\ge |z_{13}|$,
then 
$$\cosh\sqrt{c_{13}}z_{13}\sech\sqrt{c_{13}}z_{13}'\le 2e^{-\sqrt{c_{13}}|z_{13}-z_{13}'|}\,,$$
and 
\begin{equation}
  \label{eq:tau/tau'-2}
  \frac{\tau_{Y,0}}{\tau_{Y,0}'}\le \begin{cases}
  4e^{\k_3(X-X')} & \text{ if $z_{13}'\ge z_{13}\ge0$,}
\\  4e^{\k_1(X-X')} & \text{ if $z_{13}'\le z_{13}\le 0$.}
\end{cases}
\end{equation}
Combining \eqref{eq:prod-VV'} with \eqref{eq:ass-orth1} and \eqref{eq:tau/tau'-2}, we have
\begin{multline*}
\int^{X}_{-\infty}\pd_{X'}\left(e^{-\tb_j(\lambda)(X-X')}
    \frac{\tau_{Y,0}}{\tau_{Y,0}'}\right)f(X')\,dX'
\\  =-\int_X^{\infty}\pd_{X'}\left(e^{-\tb_j(\lambda)(X-X')}
    \frac{\tau_{Y,0}}{\tau_{Y,0}'}\right)f(X')\,dX'\,,
\end{multline*}
and
\begin{align}
  &  \sup_{Y\ge 0}\left\|e^{\a X}\int_X^{\infty}\pd_{X'}\left(
    e^{-\tb_j(\lambda)(X-X')}\frac{\tau_{Y,0}}{\tau_{Y,0}'}\right)
    f(X')\,dX'\right\|_{L^2(\R)}
    \notag
\\ \lesssim & \label{eq:est-compl}
\left(\frac{1}{\Re\tb_j(\lambda)-\k_1-\a}
+\frac{1}{\k_3+\a-\Re\tb_j(\lambda)}\right)\|e^{\a X}f\|_{L^2(\R)}\,.
\end{align}
\par
Next, we consider the case where $Y\le0$.
Suppose that $z_{13}\ge0$. Then $\theta_1\le \theta_3$ and
\begin{align}
0<\frac{\tau_{Y,0}}{\tau_{Y,0}'} \le &
 2\frac{e^{\theta_2}+e^{\theta_3}}{e^{\theta_2'}+e^{\theta_3'}}
\le  4e^{a_{23}(X-X')+\sqrt{c_{23}}|X-X'|}
\notag
  \\= & \label{eq:tau/tau'-23}
\begin{cases}
 4e^{\k_3(X-X')} & \text{ if $X\ge X'$,}
\\ 4e^{\k_2(X-X')} & \text{ if $X\le X'$.}  \end{cases}
\end{align}
If $\Re\tb_j(\lambda)<\k_2+\a$, then it follows from \eqref{eq:tau/tau'-23}
that
\begin{align*}
& \sup_{Y\le 0}\left\|e^{\a X} \int_X^\infty
\pd_{X'}\left(e^{\{-\tb_j(\lambda)\}(X-X')}\frac{\tau_{Y,0}}{\tau_{Y,0}'}\right)
f(X')\,dX'\right\|_{L^2(z_{13}\ge0)}
  \\ \lesssim &
\frac{1}{\k_2+\a-\Re\tb_j(\lambda)}\|e^{\a X}f\|_{L^2(\R)}\,.
\end{align*}
On the other hand, if $\Re\tb_j(\lambda)>\k_3+\a$,
\begin{align*}
& \sup_{Y\le 0}\left\|e^{\a X} \int^X_{-\infty}
\pd_{X'}\left(e^{\{-\tb_j(\lambda)\}(X-X')}\frac{\tau_{Y,0}}{\tau_{Y,0}'}\right)
f(X')\,dX'\right\|_{L^2(z_{13}\ge0)}
  \\ \lesssim &
\frac{1}{\Re\tb_j(\lambda)-\k_3-\a}\|e^{\a X}f\|_{L^2(\R)}\,.
\end{align*}
Now we consider the case where  $\Re\tb_j(\lambda)-\a\in (\k_2,\k_3)$.
Using \eqref{eq:ass-orth1} and the fact that
\begin{equation*}
  \frac{\tau_{Y,0}}{\tau_{Y,0}'}\le \begin{cases}
  4e^{\k_3(X-X')} & \text{ if $z_{23}'\ge z_{23}\ge0$ and $z_{13}\ge0$,}
\\  4e^{\k_2(X-X')} & \text{ if $z_{23}'\le z_{23}\le 0$ and $z_{13}\ge0$,}
\end{cases}
\end{equation*}
we can prove
\begin{align*}
& \sup_{Y\le 0}\left\|e^{\a X} \int^X_{-\infty}
\pd_{X'}\left(e^{\{-\tb_j(\lambda)\}(X-X')}\frac{\tau_{Y,0}}{\tau_{Y,0}'}\right)
f(X')\,dX'\right\|_{L^2(z_{13}\ge0)}
  \\ \lesssim &
\left(\frac{1}{\Re\tb_j(\lambda)-\k_2-\a}+\frac{1}{\k_3+\a-\Re\tb_j(\lambda)}\right)
\|e^{\a X}f\|_{L^2(\R)}
\end{align*}                
in the same way as \eqref{eq:est-compl}.
\par
Suppose that $z_{13}\le0$. Then $\theta_3\le \theta_1$ and
\begin{align}
0<\frac{\tau_{Y,0}}{\tau_{Y,0}'} \le &
2\frac{e^{\theta_1}+e^{\theta_2}}{e^{\theta_1'}+e^{\theta_2'}}
\le  4e^{a_{12}(X-X')+\sqrt{c_{12}}|X-X'|}
\notag
\\  = &
\label{eq:tau/tau'-12}
\begin{cases}
 4e^{\k_2(X-X')} & \text{ if $X\ge X'$,}
 \\ 4e^{\k_1(X-X')} & \text{ if $X\le X'$.}
\end{cases}
\end{align}
By \eqref{eq:tau/tau'-12}, we have for $\Re\tb_j(\lambda)<\k_1+\a$,
\begin{align*}
& \sup_{Y\le 0}\left\|e^{\a X} \int_X^\infty
\pd_{X'}\left(e^{\{-\tb_j(\lambda)\}(X-X')}\frac{\tau_{Y,0}}{\tau_{Y,0}'}\right)
f(X')\,dX'\right\|_{L^2(z_{13}\le0)}
  \\ \lesssim &
\frac{1}{\k_1+\a-\Re\tb_j(\lambda)}\|e^{\a X}f\|_{L^2(\R)}\,,
\end{align*}                
and for $\Re\tb_j(\lambda)>\k_2+\a$,
\begin{align*}
& \sup_{Y\le 0}\left\|e^{\a X} \int^X_{-\infty}
\pd_{X'}\left(e^{\{-\tb_j(\lambda)\}(X-X')}\frac{\tau_{Y,0}}{\tau_{Y,0}'}\right)
f(X')\,dX'\right\|_{L^2(z_{13}\le0)}
  \\ \lesssim &
\frac{1}{\Re\tb_j(\lambda)-\k_2-\a}\|e^{\a X}f\|_{L^2(\R)}\,.
\end{align*}                
If  $\Re\tb_j(\lambda)-\a\in (\k_1,\k_2)$, we can prove
\begin{align*}
& \sup_{Y\le 0}\left\|e^{\a X} \int^X_{-\infty}
\pd_{X'}\left(e^{\{-\tb_j(\lambda)\}(X-X')}\frac{\tau_{Y,0}}{\tau_{Y,0}'}\right)
f(X')\,dX'\right\|_{L^2(z_{13}\le0)}
  \\ \lesssim &
\left(\frac{1}{\Re\tb_j(\lambda)-\k_1-\a}+\frac{1}{\k_2+\a-\Re\tb_j(\lambda)}\right)
\|e^{\a X}f\|_{L^2(\R)}\,,  
\end{align*}
in the same way as \eqref{eq:est-compl}.
\par
Combining the above, we have
\begin{equation*}
  \sup_{Y\in\R}\|U(\cdot,Y)\|_{L^2(\R;e^{2\a X}dX)}\lesssim
  \|F\|_{L^2(\R;e^{2\a X}dX)}\,.
\end{equation*}
We can estimate
$\pd_XU(\cdot,Y)$ in the same way. Since $A_4(\lambda)U=F$ by
\eqref{eq:K-delta} and $\pd_X^{-1}$ is bounded on $L^2(\R;e^{2\a X}dX)$,
we have the bound for $\pd_X^2U(\cdot,Y)$.
Thus we prove Lemma~\ref{lem:A4-est}.
\par
By the assumptions and Remark~\ref{rem:betas},
\begin{equation*}
  \beta_{\#,-}<\k_1+\a\,,\enskip\beta_{*,+}-\a\in(\k_2,\k_3)\,,
  \enskip \beta_{\#,+}>\k_3+\a\,,\enskip \beta_{*,-}-\a\in(\k_1,\k_2)\,,
\end{equation*}
and we can prove Lemma~\ref{lem:A4*-est} in the same way as
Lemma~\ref{lem:A4-est}.
This completes the proof of Lemmas~\ref{lem:A4-est} and \ref{lem:A4*-est}.
\end{proof}

Now we will investigate the correspondence between
$\mathbf{U_1}$ and $\mathbf{U_2}$ satisfying \eqref{eq:calA-B}
when $\lambda\ne0$.

First, we will investigate $\ker(\mathcal{A}_1(\lambda))$.
For $\lambda\in\C\setminus\{0, \lambda(\beta_{*,\pm})\}$, let
\begin{equation*}
\mathbf{W_j}(\lambda)=
(\pd_XV(\beta_j(\lambda)),\pd_YV(\beta_j(\lambda)))^T\,.
\end{equation*}
For $\lambda=\lambda(\beta_{*,\pm})$, let
\begin{gather*}
\mathbf{W_1}(\lambda)=(\pd_XV(\beta_{*,\pm}),\pd_YV(\beta_{*,\pm}))^T\,, 
\quad
\mathbf{W_2}(\lambda)=
(\pd_X\pd_\beta V(\beta_{*,\pm}),\pd_Y\pd_\beta V(\beta_{*,\pm}))^T\,.
\end{gather*}

\begin{lemma}
  \label{lem:kercalA1}
Let $\mathcal{A}_1(\lambda)$ be closed operators
on $L^2(\R;e^{2\a X}dX)$.
\begin{itemize}
\item[\rm(i)]
Let $\lambda\in\C\setminus\{0,\lambda(\beta_{*,\pm})\}$, $\a\R\setminus\{0\}$ and
$\Re\beta_j(\lambda)+\a\ne \k_j$ for $j=1$, $2$, $3$. Then
  $\ker(\mathcal{A}_1(\lambda))
  =\spann\{\pd_x^{-1}\mathbf{W_j}(\lambda)\mid
  \k_1<\Re\beta_j(\lambda)+\a<\k_3\}$.
\item[\rm(ii)]
  Let $\lambda=\lambda(\beta_{*,+})$ and
  $\a\in(\k_1-\beta_{*,+},\k_1-\beta_{\#,-})\setminus\{0\}$.
  Then $\ker(\mathcal{A}_1(\lambda))
  =\spann\{\pd_x^{-1}\mathbf{W_1}(\lambda),\,
  \pd_x^{-1}\mathbf{W_2}(\lambda)\}$.
\item[\rm(iii)]
  Let $\lambda=\lambda(\beta_{*,-})$ and
$\a\in(\k_3-\beta_{\#,+},\k_3-\beta_{*,-})\setminus\{0\}$.
  Then $\ker(\mathcal{A}_1(\lambda))
  =\spann\{\pd_x^{-1}\mathbf{W_1}(\lambda),\,
  \pd_x^{-1}\mathbf{W_2}(\lambda)\}$.
\end{itemize}
\end{lemma}
\begin{proof}
  Suppose $\lambda\in\C\setminus\{0, \lambda(\beta_{*,\pm})\}$.
  By Lemma~\ref{lem:kerA-solU} (ii),
  $\mathbf{W_j}(\lambda)\in L^2(\R;e^{2\a X}dX)$ if and only if
  $\a+\Re\beta_j(\lambda)\in (\k_1,\k_3)$.  Since $\pd_XA_2(\lambda)$
  is an ordinary differential operator of the third order and $\pd_X$
  is invertible on $L^2(\R;e^{2\a X}dX)$, it follows from
  \eqref{eq:A2b} that $\{V(\beta_j(\lambda))\}_{1\le j\le 3}$ is a
  fundamental system of $\pd_XA_2(\lambda)V=0$ and that
  $\ker(A_2(\lambda(\beta_{*,\pm})))= \spann\{V(\beta_j(\lambda))\mid
  \k_1<\beta_j(\lambda)+\a<\k_3\}$.  By Lemma~\ref{cl:kerM-P},
\begin{align*}
A_1V(\beta_j(\lambda))+\pd_X^{-1}\pd_YV(\beta_j(\lambda))=&  
\nabla M_-(v_\mY)\pd_XV(\beta_j(\lambda))=0\,.
\end{align*}
Thus we prove (i).
\par
If $\lambda=\lambda(\beta_{*,\pm})$, it follows from Lemma~\ref{lem:kerA-solU} (i)
that $\{V(\beta_{*,\pm}),\pd_\beta V(\beta_{*,\pm}),V(\beta_{\#,\mp})\}$
is a fundamental system of $\pd_XA_2(\lambda)V=0$.
By Lemma~\ref{lem:kerA-solU} (v) and Remark~\ref{rem:betas}, we have
$V(\beta_{\#,\mp})\not\in L^2(\R;e^{2\a X}dX)$ and 
$\mathbf{W_1}(\lambda)$, $\mathbf{W_2}(\lambda)\in L^2(\R;e^{2\a X}dX)$
if $\lambda=\lambda(\beta_{*,+})$ and $\k_1-\beta_{*,+}<\a<\k_1-\beta_{\#,-}$
or if $\lambda=\lambda(\beta_{*,-})$ and $\k_3-\beta_{\#,-}<\a<\k_3-\beta_{*,-}$.
Thus we can prove (ii) in the same way as (i).
\end{proof}

Next, we will show that $\mathcal{A}_2(\lambda)\mathbf{V}=\mathcal{B}_2\mathbf{U_2}$
is solvable for $\mathbf{V}$ if $\mathbf{U_2}$ satisfies secular conditions.

For $\lambda\in\C\setminus\{0, -\lambda(\beta_{*,\pm})\}$, let $W_j^*(\lambda)=V(\tb_j(\lambda))$
and 
\begin{equation}
 \label{eq:defbfWj}
\mathbf{W_j^*}(\lambda)=
(-\frac{4}{3}b_2W^*_j(\lambda)+\pd_X^{-1}\pd_YW^*_j(\lambda),
W^*_j(\lambda))^T\,.
\end{equation}
For $\lambda=-\lambda(\beta_{*,\pm})$, let $W_1^*(\lambda)=V(\beta_{*,\pm})$,
$W_2^*(\lambda)=\pd_\beta V(\beta_{*,\pm})$ and $\mathbf{W_j^*}(\lambda)$ be as
\eqref{eq:defbfWj}.

\begin{lemma}
\label{lem:ImcalA2}
Assume that $\a\in\R\setminus\{0\}$. Let $\mathcal{A}_2(\lambda)$ and
$\mathcal{B}_2$ be closed operators on $L^2(\R;e^{2\a X}dX)$ and let
$\mathbf{U_2}\in(H^1\times L^2)(\R;e^{2\a X}dX)$.
\begin{itemize}
\item[\rm(i)]
Suppose that $\lambda\in\C\setminus\{0,-\lambda(\beta_{*,\pm})\}$ and
that $\Re\tb_i(\lambda)\ne \a+\k_j$ for $i$, $j=1$, $2$, $3$.
Then $\ker(\mathcal{A}_2(\lambda))=\{0\}$. Moreover, 
if $\la \mathbf{U_2},\mathbf{W_i}^*(\lambda)\ra=0$ for every $i\in\{1,2,3\}$
satisfying $\Re\tb_i(\lambda)-\a\in(\k_1,\k_3)$, then
$\mathcal{A}_2(\lambda)\mathbf{V}=\mathcal{B}_2\mathbf{U_2}$
has a unique solution satisfying
\begin{equation}
\label{eq:calA12-est-3}
  \|\mathbf{V}(\cdot,Y)\|_{(H^2\times H^1)(\R;e^{2\a X}dX)}
  \le C\|\mathbf{U_2}(\cdot,Y)\|_{(H^1\times L^2)(\R;e^{2\a X}dX)}\,,
\end{equation}
where $C$ is a constant that does not depend on $Y$ and $\mathbf{U_2}$.
\item[\rm(ii)]

  Let $\lambda=-\lambda(\beta_{*,+})$ and
$\a\in(\beta_{\#,-}-\k_1,\beta_{*,+}-\k_2)\setminus\{0\}$. 
Then $\ker(\mathcal{A}_2(\lambda))=\{0\}$. Moreover, if
$\la \mathbf{U_2},\mathbf{W_1}^*(\lambda)\ra
=\la \mathbf{U_2},\mathbf{W_2}^*(\lambda)\ra=0$,
$\mathcal{A}_2(\lambda)\mathbf{V}=\mathcal{B}_2\mathbf{U_2}$
has a unique solution satisfying \eqref{eq:calA12-est-3},
where $C$ is a constant that does not depend on $Y$ and $\mathbf{U_2}$.

\item[\rm(iii)]
 Let $\lambda=-\lambda(\beta_{*,-})$ and
$\a\in(\beta_{*,-}-\k_2,\beta_{\#,+}-\k_3)\setminus\{0\}$.
Then $\ker(\mathcal{A}_2(\lambda))=\{0\}$. Moreover, if
$\la \mathbf{U_2},\mathbf{W_1}^*(\lambda)\ra
=\la \mathbf{U_2},\mathbf{W_2}^*(\lambda)\ra=0$,
then $\mathcal{A}_2(\lambda)\mathbf{V}=\mathcal{B}_2\mathbf{U_2}$
has a unique solution satisfying \eqref{eq:calA12-est-3},
where $C$ is a constant that does not depend on $Y$ and $\mathbf{U_2}$.
\end{itemize}
\end{lemma}
\begin{proof}
Using Lemmas~\ref{lem:A4-est} and \ref{lem:A4*-est}, we can prove
Lemma~\ref{lem:ImcalA2} in the same way as Lemma~\ref{lem:calA20-est}.
\end{proof}

Next, we will investigate the kernel of the map $\mathbf{U_2}\mapsto\mathbf{U_1}$
defined by \eqref{eq:calA-B}.
\begin{lemma}
\label{lem:kermapU1-U2}
Let $\mathcal{A}_1(\lambda)$,
$\mathcal{A}_2(\lambda)$, $\mathcal{B}_1$ and $\mathcal{B}_2$
be as in
Lemma~\ref{lem:calA20-est}.
Assume that $\mathbf{U_1}$, $\mathbf{U_2}$ and $\mathbf{V}$
satisfy \eqref{eq:calA-B}.
\begin{itemize}
\item[\rm{(i)}]
Let $\lambda\in\C\setminus\{0,\lambda(\beta_{*,\pm})\}$, $\a\in\R\setminus\{0\}$ and
$\Re\beta_i(\lambda)+\a\ne\k_j$ for $i$, $j=1$, $2$, $3$.
If $\mathbf{U_1}=0$, then
$\mathbf{U_2}\in \spann\{\mathbf{W_j}(\lambda)\mid
  \k_1<\Re\beta_j(\lambda)+\a<\k_3\}$.
\item[\rm{(ii)}]
Let $\lambda=\lambda(\beta_{*,+})$ and
$\a\in(\k_1-\beta_{*,-},\k_1-\beta_{\#,-})\setminus\{0\}$.
If $\mathbf{U_1}=0$, then $\mathbf{U_2}\in \spann\{\mathbf{W_1}(\lambda),\,
\mathbf{W_2}(\lambda)\}$.
  
\item[\rm{(iii)}]
Let $\lambda=\lambda(\beta_{*,-})$ and
$\a\in(\k_3-\beta_{\#,+},\k_3-\beta_{*,-})\setminus\{0\}$.
If $\mathbf{U_1}=0$, then $\mathbf{U_2}\in \spann\{\mathbf{W_1}(\lambda),\,
\mathbf{W_2}(\lambda)\}$.
\end{itemize}
\end{lemma}

\begin{proof}[Proof of Lemma~\ref{lem:kermapU1-U2}]
If $\mathbf{U_1}=0$, we have $\mathcal{A}_1(\lambda)\mathbf{V}=0$.
Suppose that $\lambda\in\C\setminus\{0,\lambda(\beta_{*,\pm})\}$ and that
$\Re\beta_i(\lambda)+\a\ne\k_j$ for $i$, $j=1$, $2$, $3$.
Then by Lemma~\ref{lem:kercalA1},
  \begin{gather*}
\mathbf{V}=\sum_{\k_1<\a+\Re\beta_j(\lambda)<\k_3}
c_j\pd_X^{-1}\mathbf{W_j}(\lambda)\,,
    \\
\mathcal{A}_2(\lambda)\mathbf{V}
=\sum_{\k_1<\a+\Re\beta_j(\lambda)<\k_3}
c_j
\begin{pmatrix}
  2\pd_xV(\beta_j(\lambda)) \\
\{8(3v_{\mY}-b_2)\pd_X+6u_{\mY}\}V(\beta_j(\lambda))
\end{pmatrix}\,,
\end{gather*}
where $c_j$ are constants.
Since $(A_1+\pd_X^{-1}\pd_Y)V(\beta_j(\lambda))=\nabla M_-(v_{\mY})V(\beta_j(\lambda))=0$
by Lemma~\ref{cl:kerM-P}, we have
$\mathcal{A}_2(\lambda)\mathbf{V}=2\mathcal{B}_2\pd_X\mathbf{V}$.
Since $\ker(\mathcal{B}_2)=\{0\}$, it follows that
$\mathbf{U_2}=2\pd_X\mathbf{V}$. Thus we prove (i).
We can prove (ii) and (iii) in the same way.
Thus we complete the proof.
\end{proof}

Now we will prove $\sigma_p(\mL_\mY)\cap\rho(\mL_{0})\cap\rho(\mL_{23})=\emptyset$.
\begin{proposition}
  \label{prop:noev}
Assume that 
\begin{equation}
  \label{eq:ass-alpha-1}
0<\a< \min\{\k_1-\beta_{\#,-},\beta_{\#,+}-\k_3,\beta_{*,+}-\k_2,
\k_2-\k_1\}\,.
\end{equation}
Let $\mL_0$ be a closed operator on $\calX$ and $\mL_{23}$ be a closed
operator on $\calX_{23}$.  If
$\lambda\in \rho(\mL_{0})\cap\rho(\mL_{23})$ and $U_2\in \calX$ is a
solution of \eqref{eq:ef-u2}, then $U_2=0$.
\end{proposition}

\begin{proof}[Proof of Proposition~\ref{prop:noev}]
 By Lemma~\ref{lem:kerA-solU} and Remark~\ref{rem:betas}, we have
 $V(\k_1)$, $V(\k_2)$, $V(\k_3)$, $V(\beta_{*,\pm})\in L^2(\R:e^{\pm2\a X}dX)$
 and $V(\beta_{\#,\pm})\not\in L^2(\R:e^{\pm2\a X}dX)$ for $\a$ satisfying
 \eqref{eq:ass-alpha-1}.
 \par
 Let $\mathbf{U_2}=(U_2,\pd_X^{-1}\pd_YU_2)^T$.
By \cite[Proposition~C.1]{Miz23},
$\pd_X^i\pd_Y^jU_2\in\calX$ for every $i\in\Z$ and $j\ge0$,
$e^{\a X}\mathbf{U_2}(X,0)\in  L^2(\R)$ and
\begin{gather}
 \label{eq:U2+infty}
\lim_{Y\to\infty}e^{2\a a_{13}Y}\|e^{\a X}\mathbf{U}_2(\cdot,Y)
\|_{L^2}=0\,,
\\ \label{eq:U2-infty}
\lim_{Y\to-\infty}e^{2\a a_{12}Y}\|e^{\a X}\mathbf{U}_2(\cdot,Y)
\|_{L^2}=0\,.
\end{gather}
\par
First, we will prove that $\lambda=0$ is not an eigenvalue.
Since
\begin{equation*}
|V(\k_1)|+|V(\k_3)|=O(\min\{e^{-(\k_3-\k_2)z_{13}},e^{(\k_2-\k_1)z_{13}}\})
\quad\text{for $Y\ge0$,}
\end{equation*}
\begin{equation*}
  \sup_{Y\ge0}e^{-2\a a_{13}Y}(\|e^{-\a X}\mathbf{W_1^*}(0)\|_{L^2}
  +\|e^{-\a X}\mathbf{W_2^*}(0)\|_{L^2})<\infty\,.
\end{equation*}
Hence it follows from Lemma~\ref{lem:bfUU*} that for every $Y\in\R$,
\begin{equation}
  \label{eq:37}
\la \mathbf{U_2},\mathbf{W_1^*}(0)\ra=
\la \mathbf{U_2},\mathbf{W_2^*}(0)\ra=0\,.
\end{equation}
By \eqref{eq:37} and Lemma~\ref{lem:calA20-est}, there exist
$\mathbf{U_1}$ and $\mathbf{V}$ satisfying \eqref{eq:calA-B} and
\begin{equation}
  \label{eq:U1-U2bound}
  \|\mathbf{U_1}(\cdot,Y)\|_{(H^1\times L^2)(\R;e^{2\a X}dX)}
  \le C\|\mathbf{U_2}(\cdot,Y)\|_{(H^1\times L^2)(\R;e^{2\a X}dX)}\,,
\end{equation}
where $C$ is a constant independent of $Y$.
Since $\ker(A_4(0))=\{0\}$, it follows from
Lemmas~\ref{lem:U1-U2} and \ref{lem:compatibilityU1U2} that
$\mathbf{U_1}=(U_1,\pd_X^{-1}\pd_YU_1)^T$ and
$U_1$ is a solution of \eqref{eq:ef-u1} with $\lambda=0$.
Moreover, $U_1\in\calX$ by \eqref{eq:U1-U2bound}.
Since $0\in \rho(\mL_0)$ by Lemma~\ref{lem:intersection},
we have $U_1=0$ and $\mathbf{U_1}=0$. Hence it follows from
Lemma~\ref{lem:kermap0U1-U2} that there exist $c_1$, $c_2\in\C$ such that
$\mathbf{U_2}=c_1\mathbf{W_1}(0)+c_2\mathbf{W_2}(0)$.
Since $\mathbf{W_1}(0)$ and $\mathbf{W_2}(0)$ do not decay in the transverse
direction (see Lemma~\ref{lem:Vjk-profile})
and $\mathbf{U_2}\in\calX$, we have $c_1=c_2=0$ and $U_2=0$.
\par
Next, we will consider the case where $\lambda\in
\C\setminus\{0,\lambda(\beta_{*,\pm}),-\lambda(\beta_{*,\pm})\}$.
By Lemma~\ref{lem:imev}, we may assume that 
$\Re\tb_i(\lambda)\ne \k_j+\a$ for $i$, $j=1$, $2$, $3$
if $\lambda\in \rho(\mL_{0})\cap\rho(\mL_{23})$.
\par
Suppose that $\Re \tb_j(\lambda)-\a \in(\k_1,\k_3)$.
By \eqref{eq:1-1}, we have for $Y\ge0$,
$$\|e^{-\a z_{13}}\mathbf{W_j^*}(\lambda)\|_{L^2}
=O(e^{\Re(\tb_j(\lambda)-\k_1)(\tb_j(\lambda)-\k_3)Y})\,.$$
If $\Re(\tb_j(\lambda)-\k_1)(\tb_j(\lambda)-\k_3)\le0$,
it follows from Lemma~\ref{lem:bfUU*} and \eqref{eq:U2+infty}
that
\begin{equation}
  \label{eq:orthU2-Wj*}
  \la \mathbf{U_2},\mathbf{W_j^*}(\lambda)\ra=0
  \quad\text{for every $Y\in\R$.}
\end{equation}
\par
Suppose that
\begin{equation}
  \label{eq:tb1tb3>0}
\Re(\tb_j(\lambda)-\k_1)(\tb_j(\lambda)-\k_3)>0\,.  
\end{equation}
In view of  \eqref{eq:U2-infty},
it suffices to show
 $\lim_{Y\to-\infty}\|e^{-\a z_{12}}V(\tb_j(\lambda))\|_{L^2(\R)}=0$
to prove \eqref{eq:orthU2-Wj*}.
If $\Re \tb_j(\lambda)-\a \in(\k_1,\k_2)$,
then by \eqref{eq:1-2} and \eqref{eq:tb1tb3>0},
\begin{gather*}
\|e^{-\a z_{12}}V(\tb_j(\lambda))\|_{L^2(\R)}
=O\left(e^{\Re(\tb_j(\lambda)-\k_1)(\tb_j(\lambda)-\k_2)Y}\right)\,,  
\\
 \Re(\tb_j(\lambda)-\k_1)(\tb_j(\lambda)-\k_2)  
  >  (\k_3-\k_2)(\Re\tb_j(\lambda)-\k_1)
  > \a(\k_3-\k_2)>0\,.
\end{gather*}
If $\Re \tb_j(\lambda)-\a \in(\k_2,\k_3)$, then by \eqref{eq:tb1tb3>0},
\begin{align*}
 \Re(\tb_j(\lambda)-\k_2)(\tb_j(\lambda)-\k_3)
  > & (\k_2-\k_1)\{\k_3-\Re\tb_j(\lambda)\}\,.
\end{align*}
Using \eqref{eq:1-2} and the fact that
$-\a z_{12}=-\a z_{23}+\a(\k_3-\k_1)Y$, we have as $Y\to-\infty$,
\begin{align*}
\|e^{-\a z_{12}}V(\tb_j(\lambda))\|_{L^2}\lesssim &
e^{\{\a(\k_3-\k_1)+\Re(\tb_j-\k_2)(\tb_j-\k_3)\}Y}
\\ \lesssim & e^{(\k_2-\k_1)\{\a+\k_3-\Re\tb_j(\lambda)\}Y}\to0\,.
\end{align*}
\par
By \eqref{eq:orthU2-Wj*} and Lemma~\ref{lem:ImcalA2}, there exist
$\mathbf{U_1}\in\calX$ and $\mathbf{V}$ satisfying \eqref{eq:ef-u1},
\eqref{eq:calA-B} and \eqref{eq:U1-U2bound}.
By Lemmas~\ref{lem:U1-U2} and \ref{lem:compatibilityU1U2},
$U_1$ is a solution of \eqref{eq:ef-u1}.
Taking the Fourier transform of \eqref{eq:ef-u1},
we have $\widehat{U_1}(\xi,\eta)=0$ for \textit{a.e.}
$(\xi,\eta)\in\R^2$ and $\mathbf{U_1}=0$.
\par
Lemma~\ref{lem:kermapU1-U2} implies
$\mathbf{U_2}\in \spann\{\mathbf{W_j}(\lambda)\mid
  \k_1<\Re\beta_j(\lambda)+\a<\k_3\}$.
In view of Lemma~\ref{lem:no-ef} and its proof,
$\mathbf{W_j}(\lambda)\not\in\calX$ and does not decay
in the transverse direction. More precisely,
$\beta_j(\lambda)$ violates at least one of \eqref{eq:1-3}--\eqref{eq:1-3''}.
Combining this fact with \eqref{eq:11-1}, \eqref{eq:11-2}
and the asymptotic relation
$e^{\a z_{12}}V(\beta)\sim e^{\{(\beta-\k_2)(\beta-\k_3)+\eta_I^{23}\}Y}$
for $Y\le0$ when $z_{23}=0$, we obtain
  \begin{equation*}
 \spann\{\mathbf{W_j}(\lambda)\mid  \k_1<\Re\beta_j(\lambda)+\a<\k_3\}
\cap \calX=\{0\}\,,
\end{equation*}
whence  $U_2=0$.
We can prove that $U_2=0$ in the other cases in the same way.
Thus we complete the proof.
\end{proof}

Since the weight function of $\calX(\nu)^*$ localizes some perturbations
around $[2,3]$-soliton, $\sigma_r(\mL_\mY)\ne\emptyset$
as we see in Lemma~\ref{lem:adjointef}.
Using Lemmas \ref{lem:A40-est}--\ref{lem:kermapU1-U2},
we will determine $\sigma_r(\mL_\mY)$.

\begin{proposition}
  \label{prop:spec-res}
Let $\a_0$ be a sufficiently small positive number.
Assume that $\a\in(0,\a_0)$ and  \eqref{ass:X13nu}.
Let $\mL_{\mY}$ and $\mL_0$ be closed operators on $\calX(\nu)$ and
$\mL_{23}$ be a closed operator on $\calX_{23}$.
If $\lambda_0\in \sigma_r(\mL_\mY)\cap \rho(\mL_0)\cap \rho(\mL_{23})$,
then $\lambda_0=\lambda(\beta)$ for $\beta\in D_1(\nu)$
or $\lambda_0=-\lambda(\beta)$ for $\beta\in D_2\cup D_3(\nu)$. 
Moreover, if $\mL_\mY^*U^*=\bar{\lambda_0}U^*$ and $U^*\in\calX(\nu)^*$,
we have the following.
\begin{itemize}
\item[\textrm{(i)}]
If $\lambda_0 \in \lambda(D_1(\nu))\setminus -\lambda(D_2\cup D_3(\nu))$,
there exist $\beta\in D_1(\nu)$ and $c\in \C$ such that $U^*=c\,\widehat{V}(\beta)$.
\item[\textrm{(ii)}]
If $\lambda_0 \in -\lambda(D_2\cup D_3(\nu))\setminus \lambda(D_1(\nu))$,
\begin{equation*}
U^*\in \spann\{V(\beta)\mid \text{$\lambda_0=-\lambda(\beta)$
and $\beta\in D_2\cup D_3(\nu)$.}\}
\end{equation*}
\item[\textrm{(iii)}]
If $\lambda_0 \in -\lambda(D_2\cup D_3(\nu))\cap \lambda(D_1(\nu))$, there exist
$c\in\C$ and $\beta\in D_1(\nu)$ such that
\begin{equation*}
U^*-c\,\widehat{V}(\beta)\in \spann\{V(\beta)\mid
\text{$\lambda_0=-\lambda(\beta)$ and $\beta\in D_2\cup D_3(\nu)$.}\}  
\end{equation*}
\end{itemize}
Especially, $0\not\in\sigma_r(\mL_\mY)$ if $\nu=0$ and
$0\in\sigma_r(\mL_\mY)$ and $\ker(\mL_\mY^*)=\spann\{V(\k_3)\}$ if $\nu>0$.
\end{proposition}
\begin{proof}
Suppose that  $\lambda_0\in \sigma_r(\mL_\mY)$.
Then $\mL_\mY^* U^*=\bar{\lambda_0}U^*$ for a nonzero $U^*\in\calX(\nu)^*$ and
\begin{equation*}
  \mL_0\pd_x^{-1}U^*=\frac{3}{2}u_\mY U^*
-\pd_x^{-1}\left(\frac32u_\mY U^*+\bar{\lambda_0}U^*\right)
\in \calX(\nu)^*\,.
\end{equation*}
Since $\calX(\nu)^*=\calX_{12}^*+\calX_{13}(\nu)^*$,
it follows from the proof of \cite[Proposition~C.1]{Miz23} that
for a $C>0$,
\begin{equation}
\label{eq:PropC1}
  \|\pd_x^2u\|_{\calX(\nu)^*}+\|\pd_yu\|_{\calX(\nu)^*}\lesssim
\|(\mL_0-C)u\|_{\calX(\nu)^*}\lesssim
\|u\|_{\calX(\nu)^*}+\|\mL_0 u\|_{\calX(\nu)^*}\,.
\end{equation}
Combining the above, we have $\pd_xU^*\in \calX(\nu)^*$.
Since $\pd_x\mL_\mY^*=-\mL_\mY\pd_x$, it follows that
$(\mL_\mY+\lambda_0)\overline{\pd_xU^*}=0$ and that $-\lambda_0$ is an
eigenvalue of $\mL_\mY$ on $\calX^{-\a}(-\nu)=\calX^\a(\nu)^*$.
\par
Let $\mathbf{U_2}=(U_2,\pd_X^{-1}\pd_YU_2)^T$ and
\begin{equation}
  \label{eq:adevp}
\mL_\mY U_2+\lambda_0 U_2=0\,,\quad U_2\in\calX^{-\a}(-\nu)\,.
\end{equation}
Then $e^{-\a X}\mathbf{U_2}(X,0)\in  L^2(\R)$ and 
\begin{gather}
\label{eq:U2+infty'}
\lim_{Y\to\infty}e^{-(2\a a_{13}+\nu)Y}\|e^{-\a X}\mathbf{U_2}(\cdot,Y)
\|_{L^2}=0\,,
\\ \label{eq:U2-infty'}
\lim_{Y\to-\infty}e^{-2\a a_{12}Y}\|e^{-\a X}\mathbf{U_2}(\cdot,Y)
\|_{L^2}=0\,.
\end{gather}
\par
By Lemma~\ref{lem:imev} (ii), we may assume that
$\Re\beta_j(\lambda_0)+\a\ne \k_j$ for $j=1$, $2$, $3$.
Let us consider the case where
$\lambda_0\in\C\setminus\{0,\lambda(\beta_{*,\pm})\}$ and
$\Re\beta_j(\lambda_0)+\a\in(\k_1,\k_3)$.
\par
Suppose that $\Re(\beta_j(\lambda_0)-\k_1)(\beta_j(\lambda_0)-\k_3)+\nu\le 0$.
By \eqref{eq:1-1},
\begin{equation*}
\sup_{Y\ge0}\|e^{\a z_{13}+\nu Y}\mathbf{W_j^*}(-\lambda_0)\|_{L^2(\R)}<\infty\,,
\end{equation*}
and it follows from \eqref{eq:U2+infty'} and Lemma~\ref{lem:bfUU*} that
\begin{equation}
  \label{eq:3}
\la \mathbf{U_2}, \mathbf{W_j^*}(-\lambda_0)\ra=0\,.
\end{equation}
\par

Suppose that $\beta_j(\lambda_0)$ satisfies
\begin{equation}
  \label{eq:5}
\Re(\beta-\k_1)(\beta-\k_3)+\nu>0\,.
\end{equation}
By \eqref{eq:1-2},
$\sup_{Y\le0}\|e^{\a z_{12}}\mathbf{W_j^*}(-\lambda_0)\|_{L^2(\R)}<\infty$,
and \eqref{eq:3} follows from \eqref{eq:U2-infty'} and Lemma~\ref{lem:bfUU*}
unless $\beta=\beta_j(\lambda_0)$ satisfies one of the following conditions:
\begin{gather}
\label{eq:notzero-1}
\Re(\beta-\k_1)(\beta-\k_2)<0\,,
\quad \Re\beta+\a\in(\k_1,\k_2)\,,
\\
\label{eq:notzero-2}
\Re(\beta-\k_2)(\beta-\k_3)+\eta_I^{23}<0\,,
\quad \Re\beta+\a\in(\k_2,\k_3)\,.
\end{gather}
By the definition, we have $\beta\in D_1(\nu)$ if and only if $\beta$ satisfies
\eqref{eq:5} and \eqref{eq:notzero-1}.
There is no $\beta$ satisfying both \eqref{eq:5} and \eqref{eq:notzero-2}
because either $\Re\beta <\beta_{13}^-(i\nu)<\k_2-\a$ or
$\Re\beta>\beta_{13}^+(i\nu)>\k_3-\a$ for $\beta$ satisfying \eqref{eq:5}.
\par
Suppose that $\lambda_0=\lambda(\beta_{*,\pm})$.
Since $\beta_{*,\pm}\in(\k_1,\k_3)$ by Remark~\ref{rem:betas},
it follows from \eqref{eq:1-1} that
\begin{equation*}
  \lim_{Y\to\infty}e^{(2\a a_{13}+\nu)Y}
(\|e^{\a X}V(\beta_{*,\pm})\|_{L^2}+
\|e^{\a X}\pd_\beta V(\beta_{*,\pm})\|_{L^2})=0
\end{equation*}
if $\a>0$ is sufficiently small.
Combining the above with \eqref{eq:U2+infty'} and Lemma~\ref{lem:bfUU*},
we have $\la \mathbf{U_2}, \mathbf{W_1^*}(-\lambda_0)\ra=
\la \mathbf{U_2}, \mathbf{W_2^*}(-\lambda_0)\ra=0$.
\par

Suppose that $\lambda_0\not\in\{0,\lambda(\beta_{*,\pm})\}$ and
$\beta_j(\lambda_0)\not\in D_1(\nu)$ for all
$\beta_j(\lambda_0)+\a\in(\k_1,\k_3)$ or $\lambda_0=\lambda(\beta_{*,\pm})$.
By Lemmas~\ref{lem:ImcalA2} and \ref{lem:kermapU1-U2},
we have for $\lambda_0\ne -\lambda(\beta_{*,\pm})$,
\begin{gather*}
\mathbf{U_2}\in \spann\{\mathbf{W_j}(-\lambda_0)\mid
  \k_1<\Re\tb_j(\lambda_0)-\a<\k_3\}\,,
\\
\mathbf{W_j}(-\lambda_0)=
(\pd_XV(\tb_j(\lambda_0)),\pd_YV(\tb_j(\lambda_0)))^T\,,
\end{gather*}
and for $\lambda_0= -\lambda(\beta_{*,\pm})$,
$\mathbf{U_2}\in \spann\{\mathbf{W_1}(-\lambda_0),\,
\mathbf{W_2}(-\lambda_0)\}$ and
\begin{gather*}
\mathbf{W_1}(-\lambda_0)=(\pd_XV(\beta_{*,\pm}),\pd_YV(\beta_{*,\pm}))^T\,, 
\\
\mathbf{W_2}(-\lambda_0)=
(\pd_X\pd_\beta V(\beta_{*,\pm}),\pd_Y\pd_\beta V(\beta_{*,\pm}))^T\,.
\end{gather*}
By Lemma~\ref{lem:adjointef},
$\mathbf{W_j}(-\lambda_0)\in\calX(\nu)^*$ if and only if
$\tb_j(\lambda_0)\in D_2\cup D_3(\nu)$.
Since $\beta_{*,-}<\k_2$ and $\beta_{*,+}\in
(\beta_{23}^-(-i\eta_I^{23}),\beta_{23}^+(-i\eta_I^{23}))$ 
for sufficiently small $\a>0$,
we have $\beta_{*,\pm}\not\in D_2\cup D_3(\nu)$.
In view of the asymptotics of $V(\beta)$ in the neighborhood
of $[j,k]$-soliton such as \eqref{eq:1-1} and \eqref{eq:1-2} obtained
from \eqref{eq:prod-PhiPhi*j},
$$\spann\{\mathbf{W_j}(-\lambda_0)\mid \mathbf{W_j}(-\lambda_0)\not\in\calX(\nu)^*\}
\cap \calX(\nu)^*=\emptyset\,.$$
Thus we prove (ii).
\par
Using Rouch\'e's theorem,
we will prove that if $\lambda_0\in \lambda(D_1(\nu))$,
then the equation $\lambda(\beta)=\lambda_0$ has at most one solution
in the half plane $\{z\in \C \mid \Re z \le a_{12}\}$.
Let $D_R=\{x+iy\mid -R< x < a_{12},\, -R< y< R\}$
and $R$ is a positive number such that $\overline{D_1(\nu)}\subset D_R$ and
\begin{equation*}
\min_{\pd D_R}|\lambda(\beta)|
=|\lambda(a_{12})|=\frac{1}{4}(\k_2-\k_1)^2(\k_3-a_{12})\,.
\end{equation*}
By the maximum principle and Remarks~\ref{rem:boundary} and \ref{rem:adjointef},
\begin{gather*}
\sup_{\beta\in  D_1(\nu)}|\lambda(\beta)|=\max_{\beta\in \pd  D_1(\nu)}|\lambda(\beta)|
=\max(M_1,M_2,M_3)\,,
\end{gather*}
where $M_1=\max|\lambda(\mathcal{C}_{1r}(\nu))|$,
$M_2=\max|\lambda(\mathcal{C}_{1\ell})|$ and
\begin{gather*}
M_3=\max\{|\lambda(\k_1-\a+i\eta_I)|\mid \sqrt{\a(\a+\k_2-\k_1)}\le |\eta_I|
\le \sqrt{\a(\a+\k_3-\k_1)+\nu}\}\,.
  \end{gather*}

Since $|\lambda_{13}^+(\eta_R+i\nu)|$ and $|\lambda_{12}^+(\eta_R)|$ are even
and monotone increasing for $\eta_R\ge0$
(see Lemma~\ref{lem:lambda13'} and its proof),
we have 
$M_1=|\lambda_{13}^+(\pm\eta_{13}^*(\a,\nu)+i\nu)|$
and $M_2=|\lambda_{12}^+(\eta_{12}^*(\a))|$.
Since $\lambda_{13}^+(\eta)=i\eta\{\k_2-\beta_{13}^-(\eta)\}$ by
\eqref{def:lambdajk}, it follows from
\eqref{eq:curve-endpoints1}, \eqref{ass:alpha23} and \eqref{ass:X13nu} that
\begin{align*}
 M_1^2=&
\left[(\k_3-\k_1+2\a)^2\{\a(\k_3-\k_1+\a)+\nu\}+\nu^2\right]
\\ & \qquad\times
\{(\k_2-\k_1+\a)^2+\a(\k_3-\k_1+\a)+\nu\}
\\ <& 27(\k_3-\k_1)^4(\k_2-\k_1)\a\,.
\end{align*}
Since the line $\Re\beta=\k_1-\a$ intersects $\mathcal{C}_{1r}(\nu)$ and
$\mathcal{C}_{1\ell}$ at $\beta_{13}^-(\pm\eta_{13}^*(\a,\nu)+i\nu)$ and
$\beta_{12}^-(\pm\eta_{12}^*(\a))$, respectively, and 
$|\lambda(\k_1-\a+i\eta_I)|$ is even and monotone increasing for $\eta_I\ge0$,
we have $\sup_{\beta\in  D_1(\nu)}|\lambda(\beta)|=M_1$.
Thus we have $\max_{\overline{D_1(\nu)}}|\lambda(\beta)|<\min_{\pd D_R}|\lambda(\beta)|\,.$ provided $\a>0$ is sufficiently small.
Since $\beta=\k_1$ is the unique solution of $\lambda(\beta)=0$ that belongs to $D_R$
and $\k_1\in \overline{D_1(\nu)}$, it follows from Rouch\'e's theorem that
if $\lambda_0=\lambda(\beta_0)$ for a $\beta_0\in \overline{D_1(\nu)}$, then
$\beta_1(\lambda_0)=\beta_0$ and $\beta_2(\lambda_0)$, $\beta_3(\lambda_0)\not\in D_R$.
\par

Suppose that $\beta_0:=\beta_1(\lambda_0)\in D_1(\nu)$. 
By Lemmas~\ref{lem:adjointef} and \ref{lem:sol-l-ad},
\begin{equation*}
\mL_\mY^*\widehat{V}(\beta_0)
=\lambda_0\widehat{V}(\beta_0)\,,
\quad \widehat{V}(\beta_0)\in \calX(\nu)^*\,,
\end{equation*}
and $\widehat{V}(\beta_0)$ is a solution of \eqref{eq:efhatVb}
with $\beta=\beta_0$.
Let
$$\mathbf{\widetilde{W_j}}(\lambda)
=(\pd_X\widehat{V}(\beta_1(\lambda)),
\pd_Y\widehat{V}(\beta_1(\lambda)))^T\,.$$
By the definitions, we have $W_1^*(-\lambda_0)=V(\beta_0)$ and
\begin{align*}
 \la \mathbf{\widetilde{W_j}}(\lambda_0), \mathbf{W_1^*}(-\lambda_0)\ra
=&  -\frac{4}{3}b_2
\la \pd_X\widehat{V}(\beta_0), V(\beta_0)\ra
\\ & -\la \widehat{V}(\beta_0),\pd_YV(\beta_0)\ra
+\la \pd_Y\widehat{V}(\beta_0),V(\beta_0)\ra\,.
\end{align*}
By Lemma~\ref{lem:bfUU*},
$\la \mathbf{\widetilde{W_j}}(\lambda_0), \mathbf{W_1^*}(-\lambda_0)\ra$
is independent of $Y$.
Since $e^{\theta_2}/\tau_\mY$ decays exponentially as $Y\to\infty$,
we have
\begin{gather*}
V(\beta)\simeq 
\frac{1}{2}e^{(\beta-\k_1)(\beta-\k_3)Y}V_\infty(\beta)\,,
\quad 
\widehat{V}(\beta)\simeq 
-\frac{\sqrt{c_{13}}}{\beta-\k_3}
e^{-(\beta-\k_1)(\beta-\k_3)Y}\widehat{V}_\infty(\beta)\,,
\\
V_\infty(\beta)=e^{(\beta-a_{13})z_{13}}
\sech\sqrt{c_{13}}z_{13}\left(\beta-a_{13}
-\sqrt{c_{13}}\tanh\sqrt{c_{13}}z_{13}\right)\,,
\\
\widehat{V}_\infty(\beta)=e^{-(\beta-a_{13})z_{13}}
\sech\sqrt{c_{13}}z_{13}\left(\beta-a_{13}
+\sqrt{c_{13}}\tanh\sqrt{c_{13}}z_{13}\right)\,,
\end{gather*}
and
\begin{align*}
 \la \mathbf{\widetilde{W_j}}(\lambda_0), \mathbf{W_j^*}(-\lambda_0)\ra
=&   
\frac{\sqrt{c_{13}}}{\beta_0-\k_3}
\left(\frac{2}{3}b_2-2a_{13}\right)
\la \pd_X\widehat{V}_\infty(\beta_0),V_\infty(\beta_0)\ra
\\ &
+\sqrt{c_{13}}(\beta_0-\k_1)
\la \widehat{V}_\infty(\beta_0),V_\infty(\beta_0)\ra\,.
\end{align*}
Here we use $\pd_YV_\infty(\beta_0)=2a_{13}\pd_XV_\infty(\beta_0)$
and $\pd_Y\widehat{V}_\infty(\beta_0)=2a_{13}\pd_X\widehat{V}_\infty(\beta_0)$.
By straightforward computations,
\begin{gather*}
\la \widehat{V}_\infty(\beta_0),V_\infty(\beta_0)\ra
= \frac{2}{\sqrt{c_{13}}}
\left\{(\beta_0-\k_1)(\beta_0-\k_3)+\frac{2}{3}c_{13}\right\}\,,
\\
\la \pd_X\widehat{V}_\infty(\beta_0),V_\infty(\beta_0)\ra
=  (\beta_0-a_{13})
\left(\frac{4}{3}\sqrt{c_{13}}-
\left\la \widehat{V}_\infty(\beta_0),V_\infty(\beta_0)\right\ra
\right)\,,
\end{gather*}
 and
 \begin{equation*}
   \la \mathbf{\widetilde{W_j}}(\lambda_0), \mathbf{W_j^*}(-\lambda_0)\ra
=-\frac{2}{3}\lambda'(\beta_0)(\beta_0-\k_1)\ne 0\,.
 \end{equation*}
Note that  $\k_1\not\in D_1(\nu)$ and that $\beta_{*,\pm}\not\in D_1(\nu)$
thanks to \eqref{eq:1hat}  and Remark~\ref{rem:betas}
provided $\a>0$ is sufficiently small.
Hence it follows that if $U_2$ is a solution of \eqref{eq:adevp}
and $\lambda=\lambda(\beta)$ for a $\beta\in D_1(\nu)$, then 
there exists a constant $c$ such that
$\la \mathbf{U_2}-c\mathbf{\widetilde{W_j}}(\lambda),
\mathbf{W_1^*}(-\lambda)\ra =0$ for every $Y\in\R$.
Applying Lemmas~\ref{lem:ImcalA2} and \ref{lem:kermapU1-U2},
we have
\begin{equation*}
\mathbf{U_2}-c\mathbf{\widetilde{W_j}}(\lambda)\in 
 \spann\{\mathbf{W_j}(-\lambda)\mid
  \k_1<\Re\tb_j(\lambda)-\a<\k_3\}\,.
\end{equation*}
Combining the above with Lemma~\ref{lem:adjointef}, we have (i) and (iii)
for $\lambda\ne0$.
\par
Finally, we consider the case where $\lambda=0$.
If $\nu=0$, then by using \eqref{eq:U2+infty'} and the fact that $\k_2, \k_3\not\in D_2\cup D_3(0)$,
we can prove $0\not\in\sigma_r(\mL_\mY)$ in the same way as Proposition~\ref{prop:noev}.

Suppose that $\nu>0$. By Lemma~\ref{lem:orth-bfUU*0}, there exist constants $c_3, c_4\in\C$ such that
\begin{equation*}
  \mathbf{\widetilde{U_2}}:=\mathbf{U_2}-c_3\mathbf{W_3}(0)-c_4\mathbf{W_4}(0)
\perp \mathbf{W_1^*}(0)\,,\; \mathbf{W_2^*}(0)\,.
\end{equation*}
Applying Lemmas~\ref{lem:calA20-est} and
\ref{lem:kermap0U1-U2}, we have
$\mathbf{\widetilde{U_2}}\in  \spann\{\mathbf{W_1}(0),\mathbf{W_2}(0)\}$.
Since $\k_3\in D_3(\nu)$, it follows from Lemmas~\ref{lem:adjointef} and
\ref{lem:v1-v12-v23} that
$\spann\{\mathbf{W_j}(0)\}_{j=1}^4\cap \calX(\nu)^*=\spann\{\mathbf{W_2}(0)\}$.
This completes the proof of Proposition~\ref{prop:spec-res}.
\end{proof}
\bigskip

\section{Proof of Theorem~\ref{thm:linearstability}}
\label{sec:proof-main}
Let $\wmL$ be the restriction of $\mL_\mY$ to $\calX_\mY(\nu)$.
First, we will investigate invertibility of $\lambda-\wmL$ 
for $\lambda\not\in\sigma(\mL_{12})\cap\sigma(\mL_{13})$.
\begin{lemma}
  \label{lem:non-degeneracy}
Let $\a_0$ be a sufficiently small positive number.
Assume  \eqref{ass:X13nu} and that $\a\in(0,\a_0)$.
If $\lambda\in(\rho(\mL_{12})\cup\rho(\mL_{13}))\cap
\rho(\mL_0)\cup \rho(\mL_{23})$, then
$\lambda-\wmL$ has a bounded inverse on $\calX_\mY$.
\end{lemma}

Unlike in \cite[Lemmas~6.24 and 6.25]{Miz23}, we need to prove this
case separately from Lemma~\ref{lem:spec1} because spectral
projections from $\calX$ to $\spann\{g_{1j}(\eta)\}$ are not
well-defined (see Remark~\ref{rem:cef}).  To prove
Lemma~\ref{lem:non-degeneracy}, we need to resolve degeneracy of
$g_{ij}^{*,\pm}(\eta)$ at $\eta=0$.  Let
 \begin{gather}
\label{def:r12}
r_1=\frac{\k_3-\k_2}{\k_3-\k_1}\,,\quad
r_2=\frac{\k_2-\k_1}{\k_3-\k_1}\,,
\quad
\delta_{ij}(\eta)=\gamma_{ij}(\eta)-\sqrt{c_{ij}}\,,
\\ \notag
\delta_{13}^\#(\eta)=(\k_3-\k_2)\delta_{13}(\eta)-ir_1\eta\,,
\quad
\delta_{12}^\#(\eta)=ir_1\eta-(\k_3-\k_2)\delta_{12}(r_2\eta)\,,
\end{gather}
and let
\begin{align*}
& g^{1,*}(X,Y,\eta)=\frac{1}{\eta}\{g_{13}^{*,+}(X,Y,\eta)-g^{2,*}(X,Y,\eta)\}\,,
\\
& g^{2,*}(X,Y,\eta)=g_{12}^{*,-}(X,Y,\eta')+g_{13}^{*,-}(X,Y,\eta)\,,
\\ 
& \tg^{1,*}(X,Y,\eta)=\frac{1}{\eta}
\{g_{12}^{*,+}(X,Y,r_2\eta)-\tg^{2,*}(X,Y,\eta)\}\,,
\\
& \tg^{2,*}(X,Y,\eta)
=g_{12}^{*,-}(X,Y,r_2\eta)+g_{13}^{*,-}(X,Y,\eta'')\,,
\end{align*}
where $\eta'=r_2\eta+i\delta_{13}^\#(\eta)$ and
$\eta''=\eta+i\delta_{12}^\#(\eta)$.
We choose this specific form, taking into account \eqref{eq:4},
$g_{jk}^{*,+}(X,Y,0)=-\Phi^1_j\Phi^{1,*}_j$ and
$g_{jk}^{*,-}(X,Y,0)=\Phi^1_k\Phi^{1,*}_k$.
Lemma~\ref{lem:g*-growth'} tells us that 
$g^{1,*}(\eta)$ and $g^{2,*}(\eta)$ behave like $g^{0,*}_{13,1}(\eta)$ and
$g^{0,*}_{13,2}(\eta)$ as $Y\to\infty$ and that
$\tg^{1,*}(\eta)$ and $\tg^{2,*}(\eta)$ behave like $g^{0,*}_{12,1}(\eta)$ and
$g^{0,*}_{12,2}(\eta)$ as $Y\to-\infty$.

\begin{proof}[Proof of Lemma~\ref{lem:non-degeneracy}]
  If
  $\lambda\in\rho(\mL_{13})\cap \rho(\mL_{12})\cap \rho(\mL_{23})
  \cap\rho(\mL_0)$, then Lemma~\ref{lem:non-degeneracy} follows from
  Lemma~\ref{lem:spec1} and Propositions~\ref{prop:noev} and
  \ref{prop:spec-res}.  By Proposition~\ref{lem:wmL1-decay} and
  Lemma~\ref{lem:intersection}, we only have to consider the cases
  where $\lambda$ is a continuous eigenvalue associated with
  $[1,2]$-soliton or $[1,3]$-soliton.
\par
Suppose that $\tilde{\eta}\in(-\eta_{13}^*(\a),\eta_{13}^*(\a))$ and that
$\lambda_{13}^+(\tilde{\eta})\in\rho(\mL_{12})$. Let $f\in\calX_\mY$.
By Propositions~\ref{prop:noev} and \ref{prop:spec-res},
there exist $\lambda_n\in \rho(\mL)$ and $u_n\in \calX$ such that
\begin{gather*}
\left(\lambda_n-\mL_{\mY}\right)u_n=f\,,\quad
\lim_{n\to\infty}\lambda_n=\lambda_{13}^+(\tilde{\eta})\,.
\end{gather*}
Since $\lambda_n\in\rho(\mL)$ and $f\in \calX_\mY$, we have $u_n\in\calX_\mY$
and it follows from Lemma~\ref{lem:orth-analytic} that
\begin{equation}
\label{eq:un-orth}
\la u_n,g^{1,*}(\eta)\ra=\la u_n,g^{2,*}(\eta)\ra=0
\text{ for $\eta\in(-\eta_{13}^*(\a),\eta_{13}^*(\a))$.}
\end{equation}
Note that $g^{2,*}(\eta)$ is well-defined 
for $\eta\in(-\eta_{13}^*(\a),\,\eta_{13}^*(\a))$ because
$r_2\eta_{13}^*(\a)<\eta_{12}^*(\a)$.
\par
If $\liminf_{n\to\infty}\|u_n\|_\calX<\infty$, then there exist a subsequence
$\{u_n'\}$ and a $u_\infty\in\calX_\mY$ such that $u_n'$ converges weakly to  $u_\infty$
in $\calX$ and 
\begin{equation*}
  \left(\lambda_{13}^+(\tilde{\eta})-\mL_{\mY}\right)u_\infty=f\,.
\end{equation*}
\par
Suppose that $\lim_{n\to\infty}\|u_n\|_\calX=\infty$.
Let $\tilde{u}_n=u_n/\|u_n\|_\calX$ and $\tilde{f}_n=f_n/\|u_n\|_\calX$.
Then
\begin{gather*}
\left(\lambda_n-\mL_{\mY}\right)\tilde{u_n}=\tilde{f}_n\,,\quad
\|\tilde{u}_n\|_\calX=1\,,\quad \lim_{n\to\infty}\|\tilde{f}_n\|_\calX=0\,,
\\
\lim_{n\to\infty}\lambda_n=\lambda_{13}^+(\tilde{\eta})\,.
\end{gather*}
As in the proof Lemma~\ref{lem:spec1},
we can find a subsequence $\{u_n'\}$ of $\{\tilde{u}_n\}$ and $u_\infty\in\calX$ such that
$u_n'\to u_\infty$ weakly in $\mathcal{E}$, $u_n'\to u_\infty$ locally
in $\calX$, and 
\begin{equation*}
  \left(\lambda_{13}^+(\tilde{\eta})-\mL_{\mY}\right)u_\infty=0\,.
\end{equation*}
By Proposition~\ref{prop:noev}, we have $u_\infty=0$ and
$\{u_n'\}$ converges to $0$ weakly in $\mathcal{E}$. Thus
there exists a subsequence $\{\tilde{u}_{n(k)}\}$ satisfying
\eqref{eq:unk-0}--\eqref{eq:tfk-lim}. Since
$\lambda_{13}^+(\tilde{\eta})\in\rho(\mL_0)\cap\rho(\mL_{12})\cap\rho(\mL_{23})$,
it follows from \eqref{eq:tu-1}--\eqref{eq:tfk-lim} that as $k\to\infty$,
\begin{equation}
  \label{eq:tukto0}
  \|\tu^1_k\|_\calX+\|\tu^{12}_k\|_\calX+\|\tu^{23}_k\|_\calX\to0\,.
\end{equation}
By \eqref{eq:unk-0} and \eqref{eq:tukto0},
\begin{equation}
 \label{eq:tu13-1}
 \lim_{k\to\infty}\|\tu^{13}_k\|_\calX=1\,.
\end{equation}
\par
Let $\eta_0\in(|\tilde{\eta}|,\eta_{1j}^*(\a))$.
By \eqref{eq:unk-0}, \eqref{eq:un-orth} and \eqref{eq:tukto0},
\begin{equation}
  \label{eq:akto0}
a_k^j(\eta):=\la \tu_k^{13},g^{j,*}(\eta)\ra\to0
\text{ in $L^2(-\eta_0,\eta_0)$ for $j=1$, $2$ as $k\to\infty$.}
\end{equation}
\par
Let $\tu_k^{13}=v_k+w_k$ and
\begin{gather*}
a_k^{0,j}(\eta)=
\int_{\R^2}\tu_k^{13}\overline{g_{13,j}^{0,*}(z_{13},\eta)}
e^{-iy\eta}\,dXdY\,,
\\
w_k=\sum_{j=1,2}\int_{-\eta_0}^{\eta_0}
a_k^{0,j}(\eta)g_{13,j}^0(z_{13},\eta)e^{iy\eta}\,d\eta\,.
\end{gather*}
Then
\begin{gather}
\label{eq:vk}  
\left(\lambda_{13}^+(\tilde{\eta})-\mL_{13}\right)v_k=
  \tilde{f}^{13}_k-\left(\lambda_{13}^+(\tilde{\eta})-\mL_{13}\right)w_k\,,
\\  
\label{eq:vk-orth}
P_{13}(\eta_0)v_k=0\,.
\end{gather}
As $Y\to\infty$, continuous eigenfunctions of $\mL_\mY^*$ are
approximated by those of $\mL_{13}^*$.
Since $\supp\tu_k^{13}\subset\{(X,Y)\in\R^2_+\mid X^2+Y^2\ge 256k^2\}$,
it follows from Lemma~\ref{lem:g*-growth'}, \eqref{eq:7} and \eqref{eq:akto0} that
$$\|w_k\|_{\calX_{13}}\lesssim \|a_k^{0,1}\|_{L^2(-\eta_0,\eta_0)}
+\|a_k^{0,2}\|_{L^2(-\eta_0,\eta_0)}\to0\,.$$
Combining the above with Proposition~\ref{lem:wmL1-decay}, \eqref{eq:tfk-lim}
\eqref{eq:vk}  and \eqref{eq:vk-orth}, we have
\begin{align*}
  \|v_k\|_{\calX_{13}} \lesssim 
 \|\tilde{f}^{13}_k\|_{\calX_{13}}+\|w_k\|_{\calX_{13}}
\to 0\,,
\end{align*}
and $\|\tu_{13}^k\|_{\calX_{13}}\le
\|v_k\|_{\calX_{13}}+\|w_k\|_{\calX_{13}}\to 0$ as $k\to\infty$.
Since $\supp\tu^{13}_k\subset \R^2_+$,
\begin{equation*}
  \|\tu^{13}_k\|_\calX=\|\tu^{13}_k\|_{\calX_{13}}\to0\quad
  \text{as $k\to\infty$,}
\end{equation*}
which contradicts \eqref{eq:tu13-1}.
Thus we prove that
$\operatorname{Range}(\lambda_{13}^+(\tilde{\eta})-\wmL)=\calX_\mY$
if $0<|\eta|<\eta_{13}^*(\a)$.
We can prove the same results for $\lambda_{13}^-(\eta)-\wmL$ and
$\lambda_{12}^\pm(\eta)-\wmL$ in the same way.
This completes the proof of Lemma~\ref{lem:non-degeneracy}.
\end{proof}
\begin{remark}
If $\tilde{\eta}=0$, then
$\lambda_{13}^\pm(0)=\lambda_{12}^\pm(0)=0$ and
the proof of Lemma~\ref{lem:non-degeneracy} fails because
we cannot exclude the possibility that
\begin{equation*}
\liminf_{k\to\infty}\|\tu^{12}_k\|_\calX>0\quad\text{and}\quad
\liminf_{k\to\infty}\|\tu^{13}_k\|_\calX>0\,.
\end{equation*}
\end{remark}

Next, we will investigate the intersection of the sets
$\{\lambda_{13}^\pm(\eta_R+i\nu)\}$ and $\{\lambda_{12}^\pm(\eta_R)\}$.
\begin{lemma}
\label{lem:lambda12cap13}
\begin{itemize}
  \item[\rm{(i)}] There exists a $\delta_2>0$ such that
the curves $\lambda_{12}^+(\eta_R)$ and $\lambda_{12}^-(\eta_R)$ have
no intersection with the curves
$\lambda_{13}^+(\eta_R)$ and $\lambda_{13}^-(\eta_R)$ 
in the region $\{\lambda\in\C \mid \Re \lambda\ge -\delta_2,\,
\lambda\ne0\}$.
  \item[\rm{(ii)}] 
Suppose that $(\k_3-\k_2)^2>(\k_3-\k_1)(\k_2-\k_1)$.
There exist positive constants $\nu_0$ and $\delta_2$ such that
if $\nu\in(0,\nu_0)$,
the curves $\lambda_{12}^+(\eta_R)$ and $\lambda_{12}^-(\eta_R)$ have
no intersection with the curves $\lambda_{13}^+(\eta_R+i\nu)$
and $\lambda_{13}^-(\eta_R+i\nu)$ in the region
$\{\lambda\in\C\mid \Re \lambda\ge -\delta_2\}$.
  \item[\rm{(iii)}] Suppose that $(\k_3-\k_2)^2<(\k_3-\k_1)(\k_2-\k_1)$.
There exist positive constants $\nu_0$ and $\delta_2$ such that
if $\nu\in(0,\nu_0)$, then in the region
$\{\lambda\in\C\mid \Re \lambda\ge -\delta_2\}$,
the curve $\lambda_{12}^\pm(\eta_R)$
has no intersection with the curve $\lambda_{13}^\pm(\eta_R+i\nu)$,
respectively, whereas $\lambda_{12}^\pm(\eta_R)$ intersects with
$\lambda_{13}^\mp(\eta_R+i\nu)$ at a pair of distinct complex conjugate points,
respectively.
\item[\rm{(iv)}] Suppose that $(\k_3-\k_2)^2=(\k_3-\k_1)(\k_2-\k_1)$.
There exist positive constants $\nu_0$ and $\delta_2$ such that
if $\nu\in(0,\nu_0)$, then in the region
$\{\lambda\in\C\mid \Re \lambda\ge -\delta_2\}$, the curve $\lambda_{12}^-(\eta_R)$
has no intersection with the curves $\lambda_{13}^\pm(\eta_R+i\nu)$,
while the curve $\lambda_{12}^+(\eta_R)$
has no intersection with the curve $\lambda_{13}^+(\eta_R+i\nu)$
and intersects with $\lambda_{13}^-(\eta_R+i\nu)$  at a pair of distinct complex conjugate points.
\end{itemize}
\end{lemma}
\begin{proof}
First, we prove (i).
By \eqref{def:lambdajk} and \cite[Lemma~7.20]{Miz23},
the real parts of $\lambda_{1j}^\pm(\eta_R)$ $(j=2, 3)$ are even in
$\eta_R\in\R$, strictly increasing for $\eta_R\ge0$ and
$\Re\lambda_{1j}^\pm\le 0$ for $\eta_R\in\R$.
Thus we only have to investigate intersections of $\lambda_{12}^\pm$
and $\lambda_{13}^\pm$ in the strip
$\{\lambda\in\C\mid -\delta_2\le \Re\lambda\le 0\}$ for sufficiently small
$\delta_2>0$.
\par
Substituting
$\gamma_{jk}(\eta)=c_{jk}^{1/2}+i\eta/(2c_{jk}^{1/2})
+\eta^2/(8c_{jk}^{3/2})
-i\eta^3/(16c_{jk}^{5/2})+O(\eta^4)$
into \eqref{def:lambdajk}, we have
\begin{gather*}
 \lambda_{12}^+(\eta)=
i \eta\left(\k_3-\k_1+\frac{\eta^2}{(\k_2-\k_1)^3}
+O(\eta^4)\right) 
-\frac{\eta^2}{\k_2-\k_1}+\frac{2 \eta^4}{(\k_2-\k_1)^5}+O(\eta^6)\,,
\\
 \lambda_{12}^-(\eta)=
i \eta\left(\k_3-\k_2-\frac{\eta^2}{(\k_2-\k_1)^3}
+O(\eta^4)\right)
 -\frac{\eta^2}{\k_2-\k_1}+\frac{2 \eta^4}{(\k_2-\k_1)^5}+O(\eta^6)\,,
\\
\lambda_{13}^+(\eta)=
i \eta\left(\k_2-\k_1+\frac{\eta^2}{(\k_3-\k_1)^3}
+O(\eta^4)\right) 
-\frac{\eta^2}{\k_3-\k_1}+\frac{2 \eta^4}{(\k_3-\k_1)^5}+O(\eta^6)\,,
\\
\lambda_{13}^-(\eta)=
-i \eta\left(\k_3-\k_2+\frac{\eta^2}{(\k_3-\k_1)^3}
+O(\eta^4)\right)
 -\frac{\eta^2}{\k_3-\k_1}+\frac{2 \eta^4}{(\k_3-\k_1)^5}+O(\eta^6)\,.
\end{gather*}
Using the change of variables $\Im\lambda_{jk}^\pm=(\k_3-\k_1)^3\mu$ and
\eqref{def:r12},  we have for small $\eta\in\R$,
\begin{align*}
(\k_3-\k_1)^{-3}\lambda_{12}^+=& i\mu-\frac{\mu^2}{r_2} 
+\frac{2(1+r_2)\mu^4}{r_2^5}+O(\mu^6)\,,
\\
(\k_3-\k_1)^{-3}\lambda_{12}^- =&  i\mu -\frac{\mu^2}{r_2r_1^2}
+ \frac{2(r_1-r_2)\mu^4}{r_1^5r_2^5}+O(\mu^6)\,,
\\
(\k_3-\k_1)^{-3}\lambda_{13}^+ =& i\mu-\frac{\mu^2}{r_2^2} 
+\frac{2(1+r_2)\mu^4}{r_2^5}+O(\mu^6)\,,
\\
(\k_3-\k_1)^{-3}\lambda_{13}^- =& i\mu-\frac{\mu^2}{r_1^2} 
+\frac{2(1+r_1)\mu^4}{r_1^5}+O(\mu^6)\,.
\end{align*}
If $\mu$ is sufficiently small and $\mu\ne0$, we have 
 $\lambda_{12}^\pm\ne \lambda_{13}^\pm$ by comparing the
quadratic terms of $\lambda_{12}^\pm$ and $\lambda_{13}^\pm$ since
$r_2\in(0,1)$.
Moreover, we have  $\lambda_{12}^\pm\ne \lambda_{13}^\mp$
for small $\mu\ne0$ if $r_1^2\ne r_2$.
\par
In the case where $r_1^2=r_2$, the quadratic terms
of $\lambda_{12}^\pm$ and $\lambda_{13}^\mp$ are equal, and
\begin{align*}
\frac{\lambda_{12}^+-\lambda_{13}^-}{(\k_3-\k_1)^3}=&
\frac{(2+2r_2)r_1^5-(2+2r_1)r_2^5}{r_1^5r_2^5}
\mu^4+O(\mu^6)\,,
\\
\frac{\lambda_{12}^--\lambda_{13}^+}{(\k_3-\k_1)^3}
=&
\frac{2r_1-2r_2-(2+2r_2)r_1^5}{r_1^5r_2^5}
\mu^4+O(\mu^6)\,.
\end{align*}
Since $r_1^2=r_2=1-r_1$ and $r_1\in(0,1)$, we have $r_1=(\sqrt{5}-1)/2$ and  
\begin{gather*}
(2+2r_2)r_1^5-(2+2r_1)r_2^5
=2r_1-2r_2-r_1^5(2+2r_2)=20-32r_1>0\,.
\end{gather*}
Thus we prove (i).
\par
Next, we will prove (ii) and (iii).
By Lemma~\ref{lem:lambda13'}, we only have to consider intersections
of the curves $\lambda_{12}^\pm(\eta_R)$ and the curves
$\lambda_{13}^\pm(\eta_R+i\nu)$ in a region
$\{\lambda \mid |\Re\lambda|<\delta_2\}$ for sufficiently small
$\delta_2>0$.
\par
For $\eta=\eta_R+i\nu$ with $\eta_R\in\R$ and $\nu>0$,
\begin{multline*}
\lambda_{13}^+(\eta)= -(\k_2-\k_1)\nu+\frac{\nu^2}{\k_3-\k_1}
\\
+i\eta_R\left\{\k_2-\k_1-\frac{2\nu}{\k_3-\k_1}
+\frac{\eta_R^2}{(\k_3-\k_1)^3}+O\left(\eta_R^4+\nu^2\right)\right\}
\\
-\frac{\eta_R^2}{\k_3-\k_1}\left\{1+\frac{3\nu}{(\k_3-\k_1)^2}\right\}
+\frac{2\eta_R^4}{(\k_3-\k_1)^5}+O\left(\eta_R^6+\nu^3\right)\,,
\end{multline*}
\begin{multline*}
\lambda_{13}^-(\eta)=(\k_3-\k_2)\nu+\frac{\nu^2}{\k_3-\k_1}
\\
-i \eta_R\left(\k_3-\k_2+\frac{2\nu}{\k_3-\k_1}+\frac{\eta_R^2}{(\k_3-\k_1)^3}
+O(\eta_R^4+\nu^2)\right)
\\
-\frac{\eta_R^2}{\k_3-\k_1}\left\{1-\frac{3\nu}{(\k_3-\k_1)^2}\right\}
+\frac{2 \eta_R^4}{(\k_3-\k_1)^5}+O(\eta_R^6+\nu^3)\,.
\end{multline*}
Let $\Im\lambda_{13}^+(\eta)=(\k_3-\k_1)^3\mu$ and
$\nu=(\k_3-\k_1)^2\tilde{\nu}$. Then
\begin{align*}
\frac{\eta_R}{(\k_3-\k_1)^2}=
\left(\frac{1}{r_2}+\frac{2\tilde{\nu} }{r_2^2}
-\frac{\mu^2}{r_2^4}+O(\mu^4+\tilde{\nu}^2)\right)\mu\,,
\end{align*}
and it follows that
\begin{align*}
\frac{\lambda_{13}^+}{(\k_3-\k_1)^3}=&  i\mu-r_2\tilde{\nu}+\tilde{\nu}^2
-\frac{\mu^2}{r_2^2}
\left(1+\frac{(4+3r_2)\tilde{\nu}}{r_2}\right)
\\ & +\frac{2+2r_2}{r_2^5}\mu^4+O(\mu^6+\tilde{\nu}^3)\,.
\end{align*}
Similarly,
\begin{align*}
\frac{\lambda_{13}^-}{(\k_3-\k_1)^3}=& i\mu+r_1\tilde{\nu}+\tilde{\nu}^2
-\frac{\mu^2}{r_1^2}
\left(1-\frac{(4+3r_1)\tilde{\nu}}{r_1}\right)
\\ & +\frac{2+2r_1}{r_1^5}\mu^4+O(\mu^6+\tilde{\nu}^3)\,.
\end{align*}
We have $\lambda_{13}^+\ne\lambda_{12}^+$ and $\lambda_{13}^-\ne\lambda_{12}^-$
if $\nu>0$ and $\mu\in\R$ are sufficiently small.
Moreover, $\lambda_{13}^+\ne\lambda_{12}^-$ and $\lambda_{13}^-\ne\lambda_{12}^+$
provided $r_1^2>r_2$.
If $r_1^2<r_2$, the equation
$\lambda_{12}^+=\lambda_{13}^-$ has roots $\mu=\pm \mu_1$ such that
\begin{gather*}
\mu_1^2=\frac{r_1^3r_2}{r_2-r_1^2}\tilde{\nu}
+O(\tilde{\nu}^2)\,,
\\
\lambda_{12}^+=\lambda_{13}^-=(\k_3-\k_1)^3\left\{
\pm i\mu_1-\frac{r_1^3}{r_2-r_1^2}\tilde{\nu}\right\}
+O(\tilde{\nu}^2)<0\,,
\end{gather*}
and the equation $\lambda_{12}^-=\lambda_{13}^+$ has roots $\mu=\pm \mu_2$ such that
\begin{gather*}
\mu_2^2=
\frac{r_2^3r_1^2}{r_2-r_1^2}
\tilde{\nu}+O(\tilde{\nu}^2)\,,
\\
\lambda_{12}^-=\lambda_{13}^+=(\k_3-\k_1)^3\left\{
\pm i\mu_2-\frac{r_2^2}{r_2-r_1^2}\tilde{\nu}\right\}
+O(\tilde{\nu}^2)<0\,.
\end{gather*}
\par
Finally, we will prove (iv).
If $r_1^2=r_2$, then
\begin{align*}
\frac{\lambda_{12}^+-\lambda_{13}^-}{(\k_3-\k_1)^3}=&
-r_1\tilde{\nu}-\tilde{\nu}^2
-\frac{4+3r_1}{r_1^3}\tilde{\nu}\mu^2
+\frac{20-32r_1}{r_1^5r_2^5}
\mu^4+O(\mu^6)\,,
\end{align*}
\begin{align*}
\frac{\lambda_{12}^--\lambda_{13}^+}{(\k_3-\k_1)^3}
=&
r_2\tilde{\nu}-\tilde{\nu}^2+\frac{4+3r_2}{r_2^3}\tilde{\nu}\mu^2
+\frac{20-32r_1}{r_1^5r_2^5}\mu^4+O(\mu^6)\,,
\end{align*}
and it follows that for small $\tilde{\nu}$, the equation
$\lambda_{12}^+=\lambda_{13}^-$ has two nonzero real solutions
$\pm \mu_1=O(\tilde{\nu}^{1/2})$ and the equation $\lambda_{12}^-=\lambda_{13}^+$
has no real solution in a neighborhood of $0$.
This completes the proof of Lemma~\ref{lem:lambda12cap13}.
\end{proof}
By Lemmas~\ref{lem:mL-12inX1}, \ref{lem:spec1}, \ref{lem:non-degeneracy}
and \ref{lem:lambda12cap13}, we have the following.

\begin{corollary}
    \label{cor:non-degeneracy}
Let $\a_0$ be a sufficiently small positive number.
If $\a\in(0,\a_0)$, there exist positive constants  $\delta_2$ and
$\delta_3$ such that $\lambda-\wmL$ has a bounded inverse on $\calX_\mY$
if $\Re\lambda\ge -\min\{\delta_2,\delta_3\a\}$ and $\lambda\ne0$.
  \end{corollary}
  \begin{remark}
Corollary~\ref{cor:non-degeneracy} tells us that either
$\lambda=0$ is a regular point or an isolated singularity of
$(\lambda-\wmL)^{-1}$ on $\calX_\mY=\calX_\mY(0)$.
Since $\wmL$ does not have eigenvalues by Proposition~\ref{prop:noev}, 
the point $\lambda=0$ is either a regular point or an essential singularity.
If $\lambda=0$ is an essential singularity of the resolvent
operator $(\lambda-\wmL)^{-1}$, the spectral projection associated with
$\lambda=0$ is infinite-dimensional (see, e.g., \cite[XIII.4, Lemma~1]{RS4}).
It remains open whether $0\in\rho(\wmL)$.    
  \end{remark}
\par
If $\nu>0$, then $\wmL$ is invertible on $\calX_\mY(\nu)$.
\begin{lemma}
    \label{lem:continuation-1}
Let $\a_0>0$ be a sufficiently small number
and $\delta_2$ be as in Lemma~\ref{lem:lambda12cap13}.
If $\a\in(0,\a_0)$, there exist positive constants $\nu_0$ and $\delta_3$
such that $\lambda-\wmL$ has a bounded on $\calX(\nu)$ if $\nu\in(0,\nu_0)$ and
$\Re\lambda \ge -\min\{\delta_2,\delta_3\a\}$.
  \end{lemma}
\begin{proof}
By Lemmas~\ref{lem:mL-12inX1} and \ref{lem:intersection},
there exists a $\delta_3>0$ such that 
$\sigma(\mL_0)\cup \sigma(\mL_{23})\subset \{\lambda\in\C\mid \Re\lambda<-\delta_3\a\}$
for every $\nu\in(0,2\sqrt{c_{12}}\a)$.
Except for intersections of $\lambda_{13}^\pm(\eta_R+i\nu)$
and $\lambda_{12}^\mp(\eta_R)$, we can prove Lemma~\ref{lem:continuation-1}
in the same way as Lemma~\ref{lem:non-degeneracy}.
\par
Let $\lambda_0$ be an intersection of the two curves
$\lambda_{13}^+(\eta_R+i\nu)$ and $\lambda_{12}^-(\eta_R)$ in the region
$\{\lambda\in\C\mid \Re\lambda\ge -\min\{\delta_2, \delta_3\a\}$.
In view of Lemma~\ref{lem:lambda12cap13}, we have
$\lambda_0=\lambda_{13}^+(\eta_1+i\nu)=\lambda_{12}^-(\eta_2)$
for $\eta_1, \eta_2\in\R\setminus\{0\}$.
\par
Propositions~\ref{prop:noev} and \ref{prop:spec-res} tell us
that for every $f\in\calX_\mY$, there exist
$\lambda_n\in \rho(\mL)$ and $u_n\in \calX_\mY$ such that
$\left(\lambda_n-\mL_{\mY}\right)u_n=f$ and
$\lim_{n\to\infty}\lambda_n=\lambda_0$.
\par
If $\liminf_{n\to\infty}\|u_n\|_\calX<\infty$, we can prove that
$(\lambda_0-\mL)u_\infty=f$ for  a $u_\infty\in \calX(\nu)$
as in the proof of Lemma~\ref{lem:non-degeneracy}.
\par
Suppose, to the contrary, that $\lim_{n\to\infty}\|u_n\|_\calX=\infty$.
Let $\tilde{u}_n=u_n/\|u_n\|_\calX$ and $\tilde{f}_n=f_n/\|u_n\|_\calX$.
Then
\begin{gather*}
\left(\lambda_n-\mL_{\mY}\right)\tilde{u_n}=\tilde{f}_n\,,\quad
\|\tilde{u}_n\|_\calX=1\,,\quad \lim_{n\to\infty}\|\tilde{f}_n\|_\calX=0\,.
\end{gather*}
As in the proof of Lemma~\ref{lem:non-degeneracy},
we can find a subsequence  $\{\tilde{u}_{n(k)}\}$ of $\{\tilde{u}_n\}$
satisfying \eqref{eq:unk-0}--\eqref{eq:tfk-lim}. Since
$\lambda_0\in\rho(\mL_0)\cap\rho(\mL_{23})$,
it follows from \eqref{eq:tu-1}--\eqref{eq:tfk-lim} that as $k\to\infty$,
\begin{equation}
  \label{eq:tukto0'}
  \|\tu^1_k\|_{\calX(\nu)}+\|\tu^{23}_k\|_{\calX(\nu)}\to0\,.
\end{equation}
By \eqref{eq:unk-0} and \eqref{eq:tukto0'},
\begin{equation}
 \label{eq:tu12+13=1}
 \lim_{k\to\infty}(\|\tu^{12}_k\|_{\calX_{12}}+
\|\tu^{13}_k\|_{\calX_{13}(\nu)})=1\,.
\end{equation}
\par
Since $\tu_n\in\calX_\mY(\nu)$, it follows from
 Lemma~\ref{lem:orth-analytic} that
\begin{gather}
\label{eq:un-orth'1}
\la \tu_n,g_{13}^{*,+}(\eta_R+i\nu)\ra=0
\text{ for $\eta_R\in(\eta_1-\eps_1,\eta_1+\eps_1)$,}
\\
\label{eq:un-orth'2}
\la \tu_n,g_{12}^{*,-}(\eta_R)\ra=0
\text{ for $\eta_R\in(\eta_2-\eps_1,\eta_2+\eps_1)$,}
\end{gather}
where $\eps_1$ is a positive number satisfying
$0<\eps_1<\min\{|\eta_1|, |\eta_2|\}$.
\par
Let \begin{gather*}
a_{12}^{k,-}(\eta_R)=\frac{1}{i\eta_R}\int_{\R^2}\tu_{12}^k
\overline{g_{12}^{*,-}(\eta_R)}\,dxdy\,,
\\
a_{13}^{k,+,+}(\eta_R)=\frac{1}{i\eta_R-\nu}\int_{\R^2}\tu_{13}^k
\overline{g_{13}^{*,+}(\eta_R+i\nu)}\,dxdy\,,
\\
a_{13}^{k,+,-}(\eta_R)=\frac{1}{i\eta_R-\nu}\int_{\R^2}\tu_{12}^k
\overline{g_{13}^{*,+}(\eta_R+i\nu)}\,dxdy\,.
\end{gather*}
Lemma~\ref{lem:g*-growth} tells us that 
$g_{12}^{*,-}(\eta_R)$ tends to $0$ as $Y\to\infty$.
Combining Lemma~\ref{lem:g*-growth} with \eqref{eq:unk-0}, \eqref{eq:tukto0'}, 
\eqref{eq:un-orth'1} and \eqref{eq:un-orth'2}, we have
\begin{equation}
\label{eq:lim-a1213}
\lim_{k\to\infty}
\left(\|a_{12}^{k,-}\|_{L^2(\eta_1-\eps_1,\eta_1+\eps_1)}
+\|a_{13}^{k,+,+}+a_{13}^{k,+,-}\|_{L^2(\eta_2-\eps_1,\eta_2+\eps_1)}\right)=0\,.
\end{equation}
\par
Let $P_{12}^-(\eta_1,\eps_1)$ and $P_{13}^+(\eta_2,\eps_1,\nu)$
be operators defined by
\begin{gather*}
P_{12}^-(\eta_1,\eps_1)f=\int_{\eta_1-\eps_1}^{\eta_1+\eps_1}
\tilde{a}_{12}(\eta)g_{12}^0(z_{12},-\eta)e^{iy\eta}\,d\eta\,,
\\
\tilde{a}_{12}(\eta)
=\frac{1}{i\eta}\int_{\R^2}f(x,y)\overline{g_{12}^{0,*}(z_{12},-\eta)}
e^{-iy\eta}\,dxdy\,,
\\
P_{13}^+(\eta_2,\eps_1,\nu)f=\int_{\eta_2-\eps_1}^{\eta_2+\eps_1}
\tilde{a}_{13}(\eta_R)g_{13}^0(z_{13},\eta_R+i\nu)e^{(-\nu+i\eta_R)y}\,d\eta_R\,,
\\
\tilde{a}_{13}(\eta_R)=\frac{1}{i\eta_R-\nu}\int_{\R^2}f(x,y)
\overline{g_{13}^{0,*}(z_{13},\eta_R+i\nu)}e^{(\nu-i\eta_R)y}\,dxdy\,.
\end{gather*}
Thanks to Lemma~\ref{lem:lambda12cap13},
the curves $\lambda_{12}^+(\eta_R)$ and $\lambda_{12}^-(\eta_R)$ 
as well as the curves $\lambda_{13}^+(\eta_R+i\nu)$ and
$\lambda_{13}^-(\eta_R+i\nu)$ do not intersect in
a sufficiently small neighborhood of $\lambda_0$.
In view of  \eqref{eq:g-g*-int} and the proof of \cite[Proposition~6.3]{Miz23},
the operator $P_{12}^-(\eta_1,\eps_1)$ is a spectral projection of
$\mL_{12}$ onto continuous eigenspaces
$$\spann\{g_{12}^0(-\eta_R)e^{iy\eta_R} \mid 
\eta_R\in[\eta_1-\eps_1,\eta_1+\eps_1]\}$$ and
that $\lambda_0-\mL_{12}$ is invertible on $(I-P_{12}^-(\eta_1,\eps_1))\calX_{12}$.
Similarly, the operator $P_{13}^+(\eta_2,\eps_1,\nu)$ is a spectral projection
of $\mL_{13}$ onto
$$\spann\{g_{13}^0(\eta_R+i\nu)e^{(-\nu+i\eta_R)y} \mid
\eta_R\in[\eta_2-\eps_1,\eta_2+\eps_1]\}$$
and $\lambda_0-\mL_{13}$ is invertible on
$(I-P_{13}^+(\eta_2,\eps_1,\nu))\calX_{13}(\nu)$.
Note that \eqref{eq:g-g*-int} holds for any $\eta\in\C$ as long as
$g_{13}^0(z,\eta)\overline{g_{13}^{0,*}(z,\pm\eta)}$ decays exponentially
as $z\to\pm\infty$.
\par
Let $\tu_k^{1j}=v_k^{1j}+w_k^{1j}$ ($j=2, 3$),
$w_k^{12}=P_{12}^-(\eta_1,\eps_1)\tu_{12}^k$
and $w_k^{13}=P_{13}^+(\eta_2,\eps_1,\nu)\tu_{13}^k$. 
Then
\begin{gather}
\label{eq:v12k}  
\left(\lambda_{n(k)}-\mL_{12}\right)v_k^{12}=
\tilde{f}^{12}_k-\left(\lambda_0-\mL_{12}\right)w_k^{12}\,,
\\ \label{eq:v13k}  
\left(\lambda_{n(k)}-\mL_{13}\right)v_k^{13}=
\tilde{f}^{12}_k-\left(\lambda_0-\mL_{13}\right)w_k^{13}\,.
\end{gather}
By Lemma~\ref{lem:g*-growth}, \eqref{eq:lim-a1213} and the fact that
$\supp\tu_k^{12}\subset\{(X,Y)\in\R^2_- \mid X^2+Y^2\ge 256k^2\}$,
\begin{equation}
\label{eq:limw12}
\lim_{k\to\infty}\|w_{12}^k\|_{\calX_{12}}=0\,.
\end{equation}
Since $\lambda_0-\mL_{12}$ is invertible on $(I-P_{12}^-(\eta_1,\eps_1))\calX_{12}$
and $\lim_{k\to\infty}\lambda_{n(k)}=\lambda_0$,
it follows from \eqref{eq:tfk-lim}, \eqref{eq:v12k} and \eqref{eq:limw12} that 
$\lim_{k\to\infty}\|v_{12}^k\|_{\calX_{12}}=0$.
Thus we prove
\begin{equation}
\label{eq:limtu12}
\lim_{k\to\infty}\|\tu_{12}^k\|_{\calX_{12}}=0\,.
\end{equation}
\par
For $\eta=\eta_R+i\nu$ with $\eta_R\in\R$ and $0<\nu<(\k_2-\k_1)(\k_3-\k_2)$,
\begin{align*}
\Re\{ i\bar{\eta}+(\k_3-\k_2)\delta_{13}(-\bar{\eta})\}
=& \nu +(\k_3-\k_2)\delta_{13}(-\bar{\eta})
\\ \ge &   \nu +(\k_3-\k_2)(i\nu)>0\,,
\end{align*}
and it follows from Lemma~\ref{lem:g*-growth} that
\begin{equation}
  \label{eq:16}
\lim_{k\to\infty}\sup_{|\eta_R-\eta_1|\le \eps_1}
\left\|\mathbf{1}_{\supp \chi_{12,n(k)}}g_{13}^{*,+}(\eta)
\right\|_{\calX(\nu)^*}=0\,.  
\end{equation}
By \eqref{eq:tu12+13=1}, \eqref{eq:lim-a1213} and \eqref{eq:16},
\begin{equation*}
\lim_{k\to\infty}
\left(\|a_{13}^{k,+,+}\|_{L^2(\eta_2-\eps_1,\eta_2+\eps_1)}
+\|a_{13}^{k,+,-}\|_{L^2(\eta_2-\eps_1,\eta_2+\eps_1)}\right)=0\,.
\end{equation*}
Combining the above with Lemma~\ref{lem:g*-growth} and the fact that
$\supp\tu_k^{13}\subset\{(X,Y)\in\R^2_+ \mid X^2+Y^2\ge 256k^2\}$, we have
\begin{equation}
  \label{eq:limw13}
\lim_{k\to\infty}\|w_{13}^k\|_{\calX_{13}(\nu)}=0\,.
\end{equation}
By \eqref{eq:tfk-lim}, \eqref{eq:v13k} and \eqref{eq:limw13}, we have
$\lim_{k\to\infty}\|v_{13}^k\|_{\calX_{13}(\nu)}=0$,
which contradicts \eqref{eq:tu12+13=1}, \eqref{eq:limtu12} and \eqref{eq:limw13}.
Thus we prove  $\operatorname{Range}(\lambda_0-\wmL)=\calX_\mY(\nu)$.

The intersections of $\lambda_{13}^-(\eta_R+i\nu)$ and
$\lambda_{12}^+(\eta_R)$ can be handled in the same way.
This completes the proof of Lemma~\ref{lem:continuation-1}.
\end{proof}

Finally, we will prove our main result.
\begin{theorem}
    \label{thm:main}
Let $\a$ be a sufficiently small positive number.
Then there exists $\nu_0>0$ such that if $\nu\in(0,\nu_0)$,
\begin{equation*}
\|e^{t\mL_{\mY}}f\|_{\calX(\nu)} \le K e^{-bt}\|f\|_{\calX(\nu)}
\quad\text{for every $f\in \calX_\mY(\nu)$ and $t\ge0$,} 
\end{equation*}
where $K$ and $b$ are positive constants that depend only on $\a$ and $\nu$.
\end{theorem}
\begin{proof}
By Lemma~\ref{lem:continuation-1}, the operator
$\lambda-\wmL$ is invertible on $\calX_\mY(\nu)$ provided
$\Re\lambda\ge -\min\{\delta_2,\delta_3\a\}$, and Theorem~\ref{thm:main}
follows from Gearhart-Pr\"{u}ss theory (\cite{Gh}, \cite[Corollary~4]{Pruss}).
Thus we complete the proof.
\end{proof}
\bigskip

\appendix

\section{The Miura transformations and the Lax pair}
\label{sec:Miura}
We recall a relationship
between the auto B\"{a}cklund transformation and the Lax pairs of
\eqref{eq:KPII}.
Suppose that $\tau$-functions
$\tau_1=\tau_1(\eps)$ and $\tau_2=\tau_2(\eps)$ 
satisfy the auto B\"acklund transformation
\begin{gather}
  \label{eq:BT-f1}
  (D_y+D_x^2)\tau_1\cdot\tau_2=0\,,
\\ \label{eq:BT-f2}
(4D_t+D_x^3-3D_xD_y)\tau_1\cdot\tau_2=0\,,
\end{gather}
and $u_i(\eps)=2\pd_x^2\log\tau_i(\eps)$ for $i=1$, $2$ and
\begin{equation*}
  D_x^iD_y^jD_t^k\tau_1\cdot\tau_2=
  \pd_{s_1}^i\pd_{s_2}^j\pd_{s_3}^k
\tau_1(x+s_1,y+s_2,t+s_3)\tau_2(x-s_1,y-s_2,t-s_3)\bigr|_{s_1=s_2=s_3=0}.
\end{equation*}
If $u_1(\eps)$ is a solution of \eqref{eq:KPII}, then $u_2(\eps)$
is a solution of \eqref{eq:KPII}, and vice versa.
See \cite{Nim-Fr}.
\par
Let $L_i(\eps)$ and $B_i(\eps)$ be Lax pairs defined by
\begin{gather*}
  L_i(\eps)=-\pd_y+\pd_x^2+u_i(\eps)\,,\\
B_i(\eps)
=4(\pd_t+\pd_x^3)+6u_i(\eps)\pd_x
+3\left(\pd_xu_i(\eps)+\pd_x^{-1}\pd_yu_i(\eps)\right)\,.
\end{gather*}
\par
Let
$h(\eps)=\tau_2(\eps)/\tau_1(\eps)$ and $v(\eps)=\pd_x\log h(\eps)$.
Then
\begin{equation}
  \label{eq:pdxyh}
\pd_xh=hv\,,\quad \pd_yh=h\pd_x^{-1}\pd_yv\,.
\end{equation}
By \eqref{eq:BT-f1} and \eqref{eq:BT-f2},
\begin{equation}
  \label{eq:kerLax}
L_1(\eps)h(\eps)=B_1(\eps)h(\eps)=0\,,\quad
L_2^*(\eps)h(\eps)^{-1}=B_2(\eps)^*h(\eps)^{-1}=0\,.
\end{equation}
See \cite{Hirota}. Moreover, we have the intertwining property
\begin{gather}
  \label{eq:intertwine-1}
h^{-1}L_2h\pd_x=\pd_xh^{-1}L_1h\,, \quad h^{-1}B_2h\pd_x=\pd_xh^{-1}B_1h\,,
\\   \label{eq:intertwine-2}
hL_1^*h^{-1}\pd_x=\pd_xhL_2^*h^{-1}\,,
\quad hB_1^*h^{-1}\pd_x=\pd_xhB_2^*h^{-1}\,.
\end{gather}
Let $U_1=\pd_\eps u_1(0)$, $U_2=\pd_\eps u_2(0)$ and
$V=\pd_\eps v(0)$. Then
\begin{equation}
  \label{eq:U-dL}
\left\{  \begin{aligned}
& U_1=\pd_\eps L_1(0)\,,\quad U_2=\pd_\eps L_2(0)\,,
\\ &    
\pd_\eps h(0)=h(0)\pd_x^{-1}V\,,\quad
\pd_\eps h(0)^{-1}=-h(0)^{-1}\pd_x^{-1}V\,.
\end{aligned}\right.
\end{equation}
Differentiating \eqref{eq:kerLax} with respect to $\eps$ and substituting
\eqref{eq:pdxyh} and \eqref{eq:U-dL} into the resulting equations, we have
\begin{gather}
\label{eq:Darboux-1}
 U_1+h(0)^{-1}L_1(0)h(0)\pd_x^{-1}V=0\,,\quad
 U_2-h(0)L_2(0)^*h(0)^{-1}\pd_x^{-1}V=0\,,\\
\label{eq:Darboux-2}
3(\pd_x+\pd_x^{-1}\pd_y+2v(0))U_1+ h(0)^{-1}B_1(0)h(0)\pd_x^{-1}V=0\,,\\
\label{eq:Darboux-2b}
 3(-\pd_x+\pd_x^{-1}\pd_y+2v(0))U_2-h(0)B_2(0)^*h(0)^{-1}\pd_x^{-1}V=0\,.
\end{gather}
Now we will translate \eqref{eq:Darboux-1}--\eqref{eq:Darboux-2b}
in terms of the Miura transformation. Let
\begin{gather*}
M_\pm[v]=(\pm\pd_x+\pd_x^{-1}\pd_y)v-v^2\,,
  \\
\widetilde{M}_\pm[u,v]=3(\mp\pd_x+\pd_x^{-1}\pd_y)u+4(\pd_x^{-1}\pd_t+\pd_x^2)v
 +6uv\mp6\pd_x(v^2)+4v^3\,.
\end{gather*}
If $v$ is a solution of the modified KP-II equation
\begin{equation}
  \label{eq:mkp}
4\pd_tv+\pd_x^3v+3\pd_x^{-1}\pd_y^2v-6v^2\pd_xv+6(\pd_xv)(\pd_x^{-1}\pd_yv)=0\,,
\end{equation}
then the Miura transformations $M_\pm$ map $v$ into solutions
of \eqref{eq:KPII} and
\begin{equation}
  \label{eq:Miura-transform}
u_\pm=M_\pm [v]\,,\quad \widetilde{M}_\pm[u_\pm,v]=0\,.  
\end{equation}
By the definitions and \eqref{eq:intertwine-1} and \eqref{eq:intertwine-2},
\begin{gather}
\label{eq:Darboux-Lax1}
\nabla_v M_-[v]=-h(0)^{-1}L_1(0)h(0)\pd_x^{-1}
=-\pd_x^{-1}h(0)^{-1}L_2(0)h(0)\,,
\\ \label{eq:Darboux-Lax2}
\nabla_v M_+[v]=h(0)L_2(0)^*h(0)^{-1}\pd_x^{-1}=\pd_x^{-1}h(0)L_1(0)^*h(0)^{-1}\,,
\\ \label{eq:Darboux-Lax3}
\nabla_v\widetilde{M}_-[u_1,v]=
h(0)^{-1}B_1(0)h(0)\pd_x^{-1}=\pd_x^{-1}h(0)^{-1}B_2(0)h(0)\,,
\\ \label{eq:Darboux-Lax4}
\nabla_v\widetilde{M}_+[u_2,v]
=-h(0)B_2(0)^*h(0)^{-1}\pd_x^{-1}=-\pd_x^{-1}h(0)B_1(0)^*h(0)^{-1}\,,
\end{gather}
and \eqref{eq:Darboux-1}--\eqref{eq:Darboux-2b} can be rewritten as
\begin{gather}
  \label{eq:Darboux-3}
U_1=\nabla M_-[v]V\,,\quad U_2=\nabla M_+[v]V\,,  
  \\  \label{eq:Darboux-4a}
  \nabla_u\widetilde{M}_-[u_1,v]U_1+\nabla_v\widetilde{M}_-[u_1,v]V=0\,,
  \\   \label{eq:Darboux-4b}
\nabla_u\widetilde{M}_+[u_2,v]U_2+\nabla_v\widetilde{M}_+[u_2,v]V=0\,.
\end{gather}
\par
By the Miura transformations \eqref{eq:BT-f1} and \eqref{eq:BT-f2},
Y-shaped solitons are connected to the null solution.
Indeed,  $\tau$-functions $\tau_1=1$ and $\tau_2=\tau_{\mY}$
satisfy \eqref{eq:BT-f1} and \eqref{eq:BT-f2},
and $v_{\mY}=\pd_x\log \tau_{\mY}$ satisfies
\begin{equation}
  \label{eq:Mpmv2}
  M_+[v_{\mY}]=u_{\mY}\,,\quad M_-[v_{\mY}]=0\,.
\end{equation}
Let $\tau_{\mY}$ be as \eqref{eq:tau-Y} and
\begin{gather*}
L_0=-\pd_y+\pd_x^2\,,\quad  L_{\mY}=-\pd_y+\pd_x^2+u_{\mY}\,,
\\
B_0=4(\pd_t+\pd_x^3)\,,\quad
B_\mY=
4(\pd_t+\pd_x^3)+6u_\mY\pd_x+3\left(\pd_xu_\mY+\pd_x^{-1}\pd_yu_\mY\right)\,.
\end{gather*}
By \eqref{eq:kerLax} and \eqref{eq:Darboux-Lax1}--\eqref{eq:Darboux-Lax4},
we have $L_0\tau_{\mY}=0$, $L_{\mY}^*\tau_{\mY}^{-1}=0$ and
\begin{align}
&  \label{eq:Miura-Lax1}
\nabla M_-[v_{\mY}]=-\tau_{\mY}^{-1}L_0\tau_{\mY}\pd_x^{-1}
=-\pd_x^{-1}\tau_{\mY}^{-1}L_{\mY}\tau_{\mY}\,,
\\ & \label{eq:Miura-Lax2}
\nabla M_+[v_{\mY}]=\tau_{\mY}L_{\mY}^*\tau_{\mY}^{-1}\pd_x^{-1}
=\pd_x^{-1}\tau_{\mY}L_0^*\tau_{\mY}^{-1}\,,
\\ \label{eq:Miura-Lax3}
& \nabla_v\widetilde{M}_-[0,v_\mY]=
\tau_\mY^{-1}B_0\tau_\mY\pd_x^{-1}=\pd_x^{-1}\tau_\mY^{-1}B_\mY\tau_\mY\,,
\\ \label{eq:Miura-Lax4}
& \nabla_v\widetilde{M}_+[u_\mY,v]
=-\tau_\mY B_\mY^*\tau_\mY^{-1}\pd_x^{-1}
=-\pd_x^{-1}\tau_\mY B_0^*\tau_\mY^{-1}\,.
\end{align}

Finally, we recall the correspondence between solutions of 
the linearized equation around $u_\mY$ and those of the linearized
equation around $0$.
\begin{lemma}
\label{cl:kerM-P}
  Let $\beta\in \C$ and $\beta'\in \C\setminus\{\k_1,\k_2,\k_3\}$. Then
  \begin{gather}
    \label{eq:M1+pdPhiPhi*}
\nabla M_+[v_{\mY}]\pd_x
\left(\Phi^0(-i\beta)\Phi^{1,*}(-i\beta')\right)
=2\pd_x\left(\Phi^1(-i\beta)\Phi^{1,*}(-i\beta')\right)\,,
\\    \label{eq:M1-pdPhiPhi*}
\nabla M_-[v_{\mY}]\pd_x
\left(\Phi^0(-i\beta)\Phi^{1,*}(-i\beta')\right)
=2\pd_x\left(\Phi^0(-i\beta)\Phi^{0,*}(-i\beta')\right)\,,
\\
\nabla M_-[v_{\mY}]\left(\Phi^1(-i\beta)\Phi^{1,*}_j\right)=0\,.
\end{gather}
\end{lemma}
We can prove Lemma~\ref{cl:kerM-P} by a straightforward computation
(see e.g. \cite[Lemmas~3.2--Corollary~3.5, Claim~4.2]{Miz23}
and \cite[Lemma~3.6 and Corollary~3.7]{Miz24}).
\bigskip

\section{Proof of Lemmas~\ref{lem:U1-U2} and
 \ref{lem:compatibilityU1U2}}
\label{ap:lem5.1-5.2}
\begin{proof}[Proof of Lemma~\ref{lem:U1-U2}]
To prove Lemma~\ref{lem:U1-U2}, we use a relationship
of the eigenvalue problems and Darboux transformations
\eqref{eq:LMiura-1} and \eqref{eq:LMiura-2}. Let
\begin{gather*}
e_1(U)=3\pd_Y^2U-4b_2\pd_X\pd_YU+4\lambda \pd_XU+\pd_X^2(\pd_X^2-4b_1)U\,,
\\  e_2(U)=e_1(U)+6\pd_X^2(u_{\mY}U)\,.
\end{gather*}
We have $u_{\mY}=2\pd_xv_{\mY}=2(\pd_x^{-1}\pd_yv_{\mY}-v_{\mY}^2)$ by
by \eqref{eq:Mpmv2},
\begin{multline*}
e_2(U_2)-e_1(U_1)= \pd_X^3\{\widetilde{\mathcal{M}}_{\lambda,+}(U_2,V)
+\widetilde{\mathcal{M}}_{\lambda,-}(U_1,V)\}
  \\  +\left\{4(\lambda\pd_X-b_1\pd_X^2-b_2\pd_X\pd_Y+\pd_X^4)
   +3\pd_X^2(u_{\mY}\cdot)+3\pd_Y^2\right\}(U_2-U_1-2\pd_XV)
  \\  
-3\pd_X^2(\pd_Y+2v_{\mY}\pd_X)(U_1+U_2-2\pd_X^{-1}\pd_YV+4v_{\mY}V)\,.
\end{multline*}
If $U_1$, $U_2$ and $V$ satisfy \eqref{eq:LMiura-1}, then
\begin{equation*}
U_2-U_1=2\pd_XV\,,\quad 
U_1+U_2=2\pd_X^{-1}\pd_YV-4v_{\mY}V\,.
\end{equation*}
By \eqref{eq:LMiura-2} and  the above,
we have $e_1(U_1)=0$ if and only if $e_2(U_2)=0$.  Thus we complete the proof.
\end{proof}

To prove Lemma~\ref{lem:compatibilityU1U2}, we need the following.
\begin{lemma}
\label{rem:Lax-Darboux}
\begin{align}
  A_2(\lambda)V=& -e^{-\lambda t}\nabla_v\left(\widetilde{M}_-[0,v_{\mY}]+4b_2M_-[v_{\mY}]\right)
(e^{\lambda t}V) \notag
  \\ \label{eq:A2-Lax}
= & -e^{-\lambda t}\pd_x^{-1}\tau_{\mY}^{-1}(B_{\mY}-4b_2L_{\mY})\tau_{\mY}(e^{\lambda t}V)\,,
\end{align}
\begin{align}
A_4(\lambda)V=& -e^{-\lambda t}\nabla_v\left(\widetilde{M}_+[u_{\mY},v_{\mY}]+4b_2M_+[v_{\mY}]\right)
(e^{\lambda t}V)
 \notag   \\= & \label{eq:A4-Lax}
e^{-\lambda t}\pd_x^{-1}\tau_{\mY}
(B_0^*-4b_2L_0^*)\tau_{\mY}^{-1}(e^{\lambda t}V)\,.
\end{align}  
\end{lemma}
Lemma~\ref{rem:Lax-Darboux} follows immediately from
\eqref{eq:Miura-Lax1}--\eqref{eq:Miura-Lax4}.

\begin{proof}[Proof of Lemma~\ref{lem:compatibilityU1U2}]
By \eqref{eq:calA-B},
\begin{gather}
\label{eq:calA-B-1}
U_{11}=A_1V_1+V_2\,,\quad
-(3A_1+4b_2)U_{11}+3U_{12}=A_2(\lambda)V_1\,,
\\ \label{eq:calA-B-2}
U_{21}=A_3V_1+V_2\,,\quad 
-(3A_3+4b_2)U_{21}+3U_{22}=A_4(\lambda)V_1\,.
\end{gather}
Subtracting the first equation of \eqref{eq:calA-B-1}
from that of \eqref{eq:calA-B-2}, we have
$U_{21}-U_{11}=(A_3-A_1)V_1=2\pd_XV_1$. By \eqref{eq:calA-B-1} and
\eqref{eq:calA-B-2},
\begin{align*}
  3(U_{22}-U_{12})=& \{A_4(\lambda)-A_2(\lambda)\}V_1
+3(A_3U_{21}-A_1U_{11})+4b_2(U_{21}-U_{11})
\\ =&
(6u_\mY-12\pd_xv_\mY)V_1+6\pd_XV_2\,.
\end{align*}
Since $u_\mY=2\pd_xv_\mY$, we have $U_{22}-U_{12}=2\pd_XV_2$.
\par
By \eqref{eq:calA-B-0},
\begin{equation*}
 \pd_Y(U_{21}-U_{11})-\pd_X(U_{22}-U_{12})=2\pd_X(\pd_YV_1-\pd_XV_2)\,.
\end{equation*}
Substituting $\pd_YU_{21}=\pd_XU_{22}$ into the above, we have
the first equation of \eqref{eq:calA-B-0'}.
By a straightforward computation, we have
\begin{align*}
\widetilde{\mathcal{M}}_{\lambda,-}(U_{11},V_1)
=& -(3A_3+4b_2)U_{11}+3U_{12}-A_2(\lambda)V_1
\\ & +
3(\pd_X^{-1}\pd_YU_{11}-U_{12})+4b_2(V_2-\pd_X^{-1}\pd_YV_1)\,,
\end{align*}
\begin{align*}
\widetilde{\mathcal{M}}_{\lambda,+}(U_{21},V_1)
=&-(3A_3+4b_2)U_{21}+3U_{22}-A_2(\lambda)V_1
\\ & +
3(\pd_X^{-1}\pd_YU_{21}-U_{22})+4b_2(V_2-\pd_X^{-1}\pd_YV_1)\,.
\end{align*}
Substituting \eqref{eq:calA-B-1} and \eqref{eq:calA-B-2} into the above,
we have \eqref{eq:calA-B-t0} and \eqref{eq:calA-B-t0'}.
\par
Differentiating the second equation of \eqref{eq:calA-B-2} with
respect to $Y$ and substituting $\pd_YU_{21}=\pd_XU_{22}$ into
the resulting equation, we have
\begin{equation}
\label{eq:cal-comp1}
\begin{split}
&  -3A_3\pd_XU_{22}-4b_2\pd_YU_{21}+6(\pd_Yv_\mY)U_{21}+3\pd_X^{-1}\pd_Y^2U_{21}
\\ = &
A_4(\lambda)\pd_YV_1+[\pd_Y,A_4(\lambda)]V_1\,.
\end{split}
\end{equation}
Since $[B_0^*,L_0^*]=0$ and $\nabla M_+(v_\mY)=\pd_X^{-1}\pd_Y+A_3$,
it follows from \eqref{eq:Miura-Lax2} and 
\eqref{eq:A4-Lax}  that $\nabla M_+(v_\mY)\pd_XA_4(\lambda)
=A_4(\lambda)\pd_X\nabla M_+(v_\mY)$ and that
\begin{equation*}
A_4(\lambda)\pd_XA_3-A_3\pd_XA_4(\lambda)=[\pd_Y,A_4(\lambda)]\,.  
\end{equation*}
Combining the above with \eqref{eq:calA-B-2}, we have
\begin{multline}
\label{eq:cal-comp1'}
3A_3\pd_XU_{22}= A_3\pd_X\left\{A_4(\lambda)V_1+(3A_3+4b_2)U_{21}\right\}
\\=
A_4(\lambda)\pd_XA_3V_1-[\pd_Y,A_4(\lambda)]V_1
+A_3\pd_X(3A_3+4b_2)U_{21}
\\= 
-A_4(\lambda)\pd_XV_2-[\pd_Y,A_4(\lambda)]V_1
+\left\{A_4(\lambda)\pd_X+A_3\pd_X(3A_3+4b_2)\right\}U_{21}\,.
\end{multline}
Using $u_\mY=2\pd_Xv_\mY$ and $M_-(v_\mY)=0$, we have
\begin{equation}
\label{eq:cal-comp2}
\begin{split}
&A_4(\lambda)\pd_X+A_3\pd_X(3A_3+4b_2)+4b_2\pd_Y-6\pd_Yv_\mY
-3\pd_X^{-1}\pd_Y^2
\\=& 
4(\mL_\mY-\lambda)-6\pd_XM_-(v_\mY)=4(\mL_\mY-\lambda)\,.
\end{split}
\end{equation}
Substituting \eqref{eq:cal-comp1'} and \eqref{eq:cal-comp2}
into \eqref{eq:cal-comp1}, we have \eqref{eq:calA-B-0''}.
Thus we complete the proof.
\end{proof}

\bigskip
\section{Adjoint continuous eigenfunctions}
\label{sec:projection}
First, we will investigate the behavior of $g_{13}^{*,\pm}(\eta)$ and
$g_{12}^{*,\pm}(\eta)$ on $\R^2_+$ and $\R^2_-$.
Let
\begin{gather}
\label{eq:defg13-s*}
g_{13}^{\#,*}(x,\eta)=\frac12C_{32}(-\bar{\eta})
\pd_x\left\{e^{\{\sqrt{c_{12}}+\delta_{13}(-\bar{\eta})\}x}
\sech\sqrt{c_{12}}x\right\}\,,
\\
\label{eq:defg12-s*}
g_{12}^{\#,*}(x,\eta)=\frac12C_{23}(-\bar{\eta})\pd_x
\left\{e^{\{\sqrt{c_{13}}+\delta_{12}(-\bar{\eta})\}x}
\sech\sqrt{c_{13}}x\right\}\,,
\end{gather}
where
\begin{align}
\notag
C_{jk}(\eta)=&\frac{(\k_k-\k_1)(\k_j-\beta_{1j}^-(\eta))}
{(\k_j-\k_1)(\k_k-\beta_{1j}^-(\eta))}\,,
\\ =& 
\label{eq:C1jk}
1+i\frac{\k_k-\k_j}{(\k_j-\k_1)^2(\k_k-\k_1)}\eta+O(\eta^2)\,.
\end{align}
We remark that $e^{-\a x}g_{13}^{\#,*}(x,\eta)\in L^2(\R)$
if $0<\a<2\sqrt{c_{12}}$ and $\eta\in(-\eta_{13}^*(\a),\eta_{13}^*(\a))$
and that $e^{-\a x}g_{12}^{\#,*}(x,\eta)\in L^2(\R)$
if $0<\a<2\sqrt{c_{13}}$ and $\eta\in(-\eta_{12}^*(\a),\eta_{12}^*(\a))$.

Let
\begin{gather*}
\calX_{\eps,+}^*=L^2(\R^2_+;e^{-2\a z_{13}+2\eps(|z_{13}|+|Y|)}dXdY)\,,
\\ \calX_{\eps,-}^*=L^2(\R^2_-;e^{-2\a z_{12}+2\eps(|z_{12}|+|Y|)}dXdY)\,,
\end{gather*}
where $(X,Y)$ is the moving coordinate given by \eqref{eq:b1-b2Y}.
\begin{lemma}
  \label{lem:g*-growth}
Assume \eqref{ass:alpha23} and \eqref{ass:X13nu}.
  \begin{itemize}
\item[\rm{(i)}] 
Let $0\le\eta_0<\eta_{13}^*(\a,-\nu)$.
There exists an $\eps>0$
such that for $\eta=\eta_R+i\nu$ with
$\eta_R\in[-\eta_0,\eta_0]$ and $j=0$, $1$,
\begin{gather*}
\left\|\pd_\eta^j\{e^{-iY\bar{\eta}}g_{13}^{*,\pm}(\eta)-
g_{13}^{0,*}(z_{13},\pm\eta)\}\right\|_{\calX_{\eps,+}^*}\le C(\eta_0)\,,
\\
\left\|\pd_\eta^j\{e^{-iY\bar{\eta}}g_{13}^{*,-}(\eta)\}
\right\|_{\calX_{\eps,-}^*}\le C(\eta_0)\,,
\\
\left\|\pd_\eta^j\{e^{-\{i\bar{\eta}+(\k_3-\k_2)\delta_{13}(-\bar{\eta})\}Y}g_{13}^{*,+}(\eta)
-g_{13}^{\#,*}(z_{12},\eta)\}
\right\|_{\calX_{\eps,-}^*}\le C(\eta_0)\,,
\end{gather*}
where $C(\eta_0)$ is a constant with $\limsup_{\eta_0\to+0}C(\eta_0)<\infty$.
\item[\rm{(ii)}]
Let $0\le\eta_0<\eta_{12}^*(\a)$.
There exists an $\eps>0$ such that for $\eta\in[-\eta_0,\eta_0]$
and $j=0$, $1$,
\begin{gather*}
\left\|\pd_\eta^j\{e^{-iY\eta}g_{12}^{*,\pm}(\eta)- g_{12}^{0,*}(z_{12},\pm\eta)\}
\right\|_{\calX_{\eps,-}^*}\le C(\eta_0)\,,
\\ 
\left\|\pd_\eta^j\{e^{-iY\eta}g_{12}^{*,-}(\eta)\}
\right\|_{\calX_{\eps,+}^*}\le C(\eta_0)\,,
\\
\left\|
\pd_\eta^j\{e^{-\{i\eta-(\k_3-\k_2)\delta_{12}(-\eta)\}Y}g_{12}^{*,+}(\eta)
-g_{12}^{\#,*}(z_{13},\eta)\}
\right\|_{\calX_{\eps,+}^*}\le C(\eta_0)\,,
\end{gather*}
where $C(\eta_0)$ is a constant with $\limsup_{\eta_0\to+0}C(\eta_0)<\infty$.
\end{itemize}
\end{lemma}

To prove Lemma~\ref{lem:g*-growth}, we need the following.
\begin{lemma}
  \label{cl:pd-1prod}
Let $\beta\in\C$ and $\beta'\in\C\setminus\{\k_1,\k_2,\k_3\}$.
\begin{align}
    \label{eq:prod-PhiPhi*}
&  \Phi^1(-i\beta)\Phi^{1,*}(-i\beta')
  \\  =& \notag
\pd_x\left\{
\Phi^0(-i\beta)\left(\frac{1}{\beta-\beta'}\Phi^{0,*}(-i\beta')
+ \Phi^{1,*}(-\beta')\right)\right\}
\quad\text{if $\beta\ne\beta'$,}
\end{align}
\begin{align}
  \label{eq:prod-PhiPhi*=}
  \Phi^1(-i\beta)\Phi^{1,*}(-i\beta)=1+\pd_x\left\{
\Phi^0(-i\beta)\Phi^{1,*}(-\beta)\right\}\,,
\end{align}
\begin{gather}
\label{eq:prod-PhiPhi*j}
  \Phi^1(-i\beta)\Phi^{1,*}_j
  =\pd_x\left(\Phi^0(-i\beta)\Phi^{1,*}_j\right)
\quad\text{for $j=1$, $2$, $3$,}
  \\
\label{eq:prod-PhijPhi*}
\Phi^1_j\Phi^{1,*}(-i\beta')=\pd_xJ_j
\quad\text{for $j=1$, $2$, $3$,}
 \end{gather}
 where
 \begin{gather*}
J_j=\Phi^{0,*}(-i\beta')\frac{1}{\tau_{\mY}}\sum_{1\le k\le 3,\,k\ne j}
\frac{(\k_k-\k_j)e^{\theta_j+\theta_k}}{(\beta'-\k_j)(\beta'-\k_k)}\,.     
\end{gather*}
\end{lemma}
Lemma~\ref{cl:pd-1prod} follows from straightforward computations.
\begin{proof}[Proof of Lemma~\ref{lem:g*-growth}]
By \eqref{def:lambdajk} and \eqref{eq:btitj}, we have for $1\le j<k\le 3$,
  \begin{align*}
\Phi^0(-i\beta_{jk}^+(\eta))=&
e^{\theta_k+\delta_{jk}(\eta)z_{jk}+iY\eta-\lambda_{jk}^-(-\eta)t}\,,
\\
\Phi^0(-i\beta_{jk}^-(\eta))=&
e^{\theta_j-\delta_{jk}(\eta)z_{jk}+iY\eta+\lambda_{jk}^+(\eta)t}\,,
\\
\Phi^{0,*}(-i\beta_{jk}^-(-\eta))=&
e^{-\theta_j+\delta_{jk}(-\eta)z_{jk}+iY\eta-\lambda_{jk}^+(-\eta)t}\,,
\\
\Phi^{0,*}(-i\beta_{jk}^+(-\eta))=&
e^{-\theta_k-\delta_{jk}(-\eta)z_{jk}+iY\eta+\lambda_{jk}^-(\eta)t}\,.
\end{align*}
Combining the above with Lemma~\ref{cl:pd-1prod}, we have
  \begin{align*}
g_{13}^{*,+}(\eta)=&
e^{iY\bar{\eta}}\pd_x\left\{e^{\delta_{13}(-\bar{\eta})z_{13}}
\frac{e^{\theta_3}}{\tau_{\mY}}
+C_{32}(-\bar{\eta})e^{\delta_{13}(-\bar{\eta})z_{13}}
\frac{e^{\theta_2}}{\tau_{\mY}}\right\}
\\ =&: e^{iY\bar{\eta}}\{g_{131}^{*,+}(\eta)+g_{132}^{*,+}(\bar{\eta})\}\,,
\end{align*}
\begin{align*}
g_{132}^{*,+}(\eta)=& C_{32}(-\bar{\eta})e^{(\k_3-\k_2)\delta_{13}(-\bar{\eta})Y}
\pd_x\left\{\frac{e^{\theta_2+\delta_{13}(-\bar{\eta})z_{12}}}{e^{\theta_1}+e^{\theta_2}}
\left(1-\frac{e^{\theta_3}}{\tau_{\mY}}\right)\right\}
\\=&
e^{(\k_3-\k_2)\delta_{13}(-\bar{\eta})Y}g_{13}^{\#,*}(z_{12},\eta)
\left(1+O\left(\frac{e^{\theta_3}}{\tau_{\mY}}\right)\right)\,,
\end{align*}
\begin{align*}
g_{12}^{*,+}(\eta)=&
e^{iY\bar{\eta}}\pd_x\left\{
\frac{e^{\theta_2+\delta_{12}(-\bar{\eta})z_{12}}}{\tau_{\mY}}
+C_{23}(-\bar{\eta})
\frac{e^{\theta_3+\delta_{12}(-\bar{\eta})z_{12}}}{\tau_{\mY}}\right\}
\\=&: e^{iY\bar{\eta}}\{g_{121}^{*,+}(\eta)+g_{122}^{*,+}(\eta)\}\,,
\end{align*}
\begin{align*}
g_{122}^{*,+}(\eta)=& C_{23}(-\bar{\eta})e^{-(\k_3-\k_2)\delta_{12}(-\bar{\eta})Y}
\pd_x\left\{\frac{e^{\theta_3+\delta_{12}(-\bar{\eta})z_{13}}}{e^{\theta_1}+e^{\theta_3}}
\left(1-\frac{e^{\theta_2}}{\tau_{\mY}}\right)\right\}
\\=&
e^{-(\k_3-\k_2)\delta_{12}(-\bar{\eta})Y}g_{12}^{\#,*}(z_{13},\eta)
\left(1+O\left(\frac{e^{\theta_2}}{\tau_{\mY}}\right)\right)\,,
\end{align*}
and
\begin{equation*}
g_{jk}^{*,-}(\eta)=e^{iY\bar{\eta}}\pd_x\left\{
    e^{\delta_{jk}(\bar{\eta})z_{jk}}\frac{e^{\theta_k}}{\tau_{\mY}}
    \right\}\,.
  \end{equation*}
Suppose that $Y\ge0$. Then $\theta_2\le \max(\theta_1,\theta_3)$.
In view of \eqref{eq:def-gij*0},
\begin{align*}
  g_{13}^{*,-}(\eta)=&  e^{iY\bar{\eta}}g_{13}^{0,*}(z_{13},-\eta)
\left(1+O\left(\frac{e^{\theta_2}}{\tau_{\mY}}\right)\right)\,,
\\
  g_{131}^{*,+}(\eta)=& g_{13}^{0,*}(z_{13},\eta)
\left(1+O\left(\frac{e^{\theta_2}}{\tau_{\mY}}\right)\right)\,,
\\
 g_{12}^{*,-}(\eta)=&  e^{iY\eta}g_{12}^{0,*}(z_{12},-\eta)
\left(1+O\left(\frac{e^{\theta_3}}{\tau_{\mY}}\right)\right)\,,
\\
g_{121}^{*,+}(\eta)=&   g_{12}^{0,*}(z_{12},\eta)
\left(1+O\left(\frac{e^{\theta_3}}{\tau_{\mY}}\right)\right)\,.
\end{align*}
Since $0\le  \Re \delta_{13}(\pm\bar{\eta})\le \Re \delta_{13}(\eta_0-i\nu)<\a$
for $\eta_R\in[-\eta_0,\eta_0]$ and
  \begin{equation}
\label{2divtau}
\frac{e^{\theta_2}}{\tau_{\mY}}\le\min_{j=1,3}e^{\theta_2-\theta_j}\le
e^{-(\k_3-\k_2)(\k_2-\k_1)Y}\min\{e^{(\k_2-\k_1)z_{13}},\,e^{-(\k_3-\k_2)z_{13}}\}\,,
\end{equation}
we have the first estimate in (i).
Applying 
\begin{equation}
   \label{3divtau}
\frac{e^{\theta_3}}{\tau_{\mY}}\le \min\{1,e^{\theta_3-\theta_2}\}
=\min\{1,\,
e^{(\k_3-\k_2)(\k_2-\k_1)Y+(\k_3-\k_2)z_{13}}\}
\end{equation}
for $Y\le 0$ and $z_{13}\le0$ and
\begin{equation*}
|e^{\{-\a+\delta_{12}(\pm\eta)\}z_{12}}|\le
e^{\{\a-\Re\delta_{12}(\eta_0)\}\{-z_{13}+\a(\k_3-\k_2)Y\}}
\end{equation*}
for $Y\le0$ and $z_{13}\ge0$,
we have the first estimate in (ii).
\par
If  $Y\le0$ and $\eta_R\in[-\eta_0,\eta_0]$,
\begin{align*}
|e^{-\a z_{12}}g_{13}^{*,-}(\eta)| \lesssim &
e^{\{\Re \delta_{13}(\bar{\eta})-\a\}z_{13}+\a(\k_3-\k_2)Y}\min(1,e^{2\sqrt{c_{13}}z_{13}})
\,.
\end{align*}
Since $0\le\Re \delta_{13}(\bar{\eta})\le\Re \delta_{13}(\eta_0-i\nu)<\a$ and
$\a<\sqrt{c_{13}}$,
\begin{equation*}
\{e^{-iY\bar{\eta}}g_{13}^{*,-}(\eta) \mid \text{$\eta=\eta_R+i\nu$, $\eta\in[-\eta_0,\eta_0]$}\}
\end{equation*}
is bounded in $\calX_{\eps,-}^*$ provided $\eps>0$ is sufficiently small.
Similarly, we can prove the uniform boundedness of
$e^{-iY\bar{\eta}}g_{131}^{*,+}(\eta)$ 
and $e^{-(\k_3-\k_2)\delta_{13}(-\bar{\eta})Y}g_{132}^{*,+}(\eta)-g_{13}^{\#,*}(z_{12},\eta)$.
Thus we prove the second and the third estimates in (i).
\par
Since $0\le \Re \delta_{12}(\eta_0)<\a$ and
$|g_{12}^{*,-}(\eta)|+|g_{121}^{*,+}(\eta)| \lesssim
e^{\Re\gamma_{12}(\eta) z_{12}}
\sech\sqrt{c_{12}}z_{12}$,
there exists an $\eps>0$ such that
\begin{align*}
& \sup_{\eta\in[-\eta_0,\eta_0]}  
\left( \|g_{12}^{*,-}(\eta)\|_{\calX_{\eps,+}^*}
+\|g_{121}^{*,+}(\eta)\|_{\calX_{\eps,+}^*}\right)
\\ & \lesssim
\left(\|e^{\{-\a+\Re\delta_{12}(\eta_0)z_{12}+\eps\}z_{12}}\|_{L^2(z_{12}\ge0)}
+\|e^{(2\sqrt{c_{12}}-\a-\eps)z_{12}}\|_{L^2(z_{12}\le0)}\right)
\\ & \phantom{e^{\{-\a+\Re\delta_{12}(\eta_0)z_{12}+\eps\}z_{12}}}
\times
\|e^{-\{(\a-\eps)(\k_3-\k_2)-\eps\}Y}\|_{L^2(Y\ge0)}
\\ & <\infty\,.
\end{align*}
Combining the above with \eqref{2divtau} and
$|g_{12}^{\#,*}(\eta)| \lesssim   
\min\{e^{2\sqrt{c_{13}}z_{13}},e^{\Re \delta_{12}(\eta_0)z_{13}}\}$,
we have the second and the third estimates in (ii).
\par
Thus we complete the proof.
\end{proof}
\begin{remark}
\label{rem:delta-sharp}
Lemma~\ref{lem:g*-growth} tells us that
$g_{13}^{*,\pm}(\eta)$ oscillate like $e^{iY\eta}$ as $Y\to+\infty$
whereas $g_{13}^{*,+}(\eta)$ oscillates like 
$e^{\{ir_2\eta+\delta_{13}^\#(-\eta)\}Y}$
and $g_{13}^{*,-}(\eta)$ vanishes as $Y\to-\infty$ and that
$g_{12}^{*,\pm}(\eta)$ oscillate like $e^{iY\eta}$ as $Y\to-\infty$ whereas
$g_{12}^{*,+}(r_2\eta)$ oscillates like 
$e^{\{i\eta+\delta_{12}^\#(\eta)\}Y}$ and
$g_{12}^{*,-}(\eta)$ vanishes as $Y\to\infty$. We have
\begin{equation}
\label{eq:delta-sharp}
\delta_{13}^\#(\eta)=\delta_{13}^{\#,0}(\eta)+O(\eta^3)\,,\quad
\delta_{12}^\#(\eta)=\delta_{12}^{\#,0}(\eta)+O(\eta^3)\,,
\end{equation}
where
\begin{gather*}
\delta_{13}^{\#,0}(\eta)=\frac{\k_3-\k_2}{(\k_3-\k_1)^3}\eta^2\,,
\quad
\delta_{12}^{\#,0}(\eta)=-\frac{\k_3-\k_2}{(\k_3-\k_1)^2(\k_2-\k_1)}\eta^2\,.
 \end{gather*}
\end{remark}

\bigskip

Let
\begin{gather*}
\tg_{jk}^{0,*}(z,\eta)=
\frac{g_{jk}^{0,*}(z,\eta)-g_{jk}^{0,*}(z,0)}{\bar{\eta}}\,,
\quad
\tg_{jk}^{\#,*}(z,\eta)=
\frac{g_{jk}^{\#,*}(z,\eta)-g_{jk}^{\#,*}(z,0)}{\bar{\eta}}\,.
\end{gather*}
By Lemma~\ref{lem:g*-growth}, we have the following.
\begin{lemma}
 \label{lem:g*-growth'}
Let $0<\a<\k_2-\k_1$ and $0\le \eta_0<\eta_{12}^*(\a)$.
There exists an $\eps>0$ such that for $\eta\in[-\eta_0,\eta_0]$,
\begin{gather*}
\left\|g^{1,*}(\eta)-2ie^{iY\eta}g_{13,1}^{0,*}(z_{13},\eta)
\right\|_{\calX_{\eps,+}^*}
+
\left\|g^{2,*}(\eta)-e^{iY\eta}g_{13}^{0,*}(z_{13},-\eta)
\right\|_{\calX_{\eps,+}^*}\le C(\eta_0)\,,
\end{gather*}
\begin{multline*}
\Bigl\|g^{1,*}(\eta)-e^{\{ir_2\eta+\delta_{13}^\#(\eta)\}Y}
\bigl\{\tg_{13}^{\#,*}(z_{12},\eta)
\\ -\left(r_2-i\eta^{-1}\delta_{13}^\#(-\eta)\right)
  \tg_{12}^{0,*}(z_{12},-\eta')\bigr\}
\Bigr\|_{\calX_{\eps,-}^*}\le C(\eta_0)\,,
\end{multline*}
\begin{gather*}
\left\|g^{2,*}(\eta)-  e^{\{ir_2\eta+\delta_{13}^\#(-\eta)\}Y}
g_{12}^{0,*}(z_{12},-\eta')\right\|_{\calX_{\eps,-}^*}\le C(\eta_0)\,,
\end{gather*}
\begin{multline*}
\Bigl\|\tg^{1,*}(\eta)-
e^{\{i\eta+\delta_{12}^\#(-\eta)\}Y}
\bigl\{r_2\tg_{12}^{\#,*}(z_{13},r_2\eta)
\\ -\left(1-i\eta^{-1}\delta_{12}^\#(-\eta)\right)
\tg_{13}^{0,*}(z_{13},-\eta'')
\bigr\}\Bigr\|_{\calX_{\eps,+}^*}\le C(\eta_0)\,,
\end{multline*}
\begin{gather*}
\left\|\tg^{2,*}(\eta)- e^{\{i\eta+\delta_{12}^\#(-\eta)\}Y}
g_{13}^{0,*}(z_{13},-\eta'')\right\|_{\calX_{\eps,+}^*}\le C(\eta_0)\,,  
\\
\left\|\tg^{1,*}(\eta)
-2ir_2e^{ir_2Y\eta}g_{12,1}^{0,*}(z_{12},r_2\eta)
\right\|_{\calX_{\eps,-}^*}\le C(\eta_0)\,,
\\
\left\|\tg^{2,*}(\eta)-e^{ir_2Y\eta}g_{12}^{0,*}(z_{12},-r_2\eta)
\right\|_{\calX_{\eps,-}^*}\le C(\eta_0)\,.
\end{gather*}
\end{lemma}
\begin{proof}
First, we remark that $\delta_{12}(0)=\delta_{13}(0)=0$, that
$g_{131}^{*,+}(0)=g_{13}^{*,-}(0)$ and that
$g_{132}^{*,+}(0)=g_{12}^{*,-}(0)$.

For $Y\ge 0$,
\begin{align*}
 \frac{g_{131}^{*,+}(\eta)-e^{-iY\eta}g_{13}^{*,-}(\eta)}{\eta}
=& 
2ig_{13,1}^{0,*}(z_{13},\eta)\left(
1+O\left(\frac{e^{\theta_2}}{\tau_\mY}\right)\right)\,,
\end{align*}
and 
\begin{gather*}
\frac{g_{131}^{*,+}(\eta)+e^{-iY\eta}g_{13}^{*,-}(\eta)}{\eta}
- 2ig_{13,1}^{0,*}(z_{13},\eta)\in \calX_{\eps,+}^*\,.
\end{gather*}
As in the proof of Lemma~\ref{lem:g*-growth},
\begin{align*}
& \frac{g_{131}^{*,+}(\eta)-e^{-iY\eta}g_{13}^{*,-}(\eta)}{\eta}
 =
\pd_x\left(
\frac{e^{\delta_{13}(\eta)z_{13}}-e^{\delta_{13}(-\eta)z_{13}}}{\eta}
\frac{e^{\theta_3}}{\tau_\mY}\right)
\in \calX_{\eps,-}^*\,.
\end{align*}
Since $\delta_{13}(-\eta)z_{13}+iY\eta=
\delta_{13}(-\eta)z_{12}+i\overline{\eta'}Y$ and
$g_{13}^{\#,*}(z_{12},0)=g_{12}^{*,0}(z_{12},0)$, we have  for $Y\ge0$,
\begin{align*}
 \frac{g_{132}^{*,+}(\eta)-e^{-iY\eta}g_{12}^{*,-}(\eta')}{\eta}
=&
\pd_x\left\{\frac{C_{32}(-\eta)e^{\delta_{13}(-\eta)z_{13}}
-e^{\delta_{12}(\overline{\eta'})z_{12}+i(\overline{\eta'}-\eta)Y}}{\eta}
\frac{e^{\theta_2}}{\tau_\mY} \right\}
\\=&
O\left((1+|z_{13}|+|Y|)\frac{e^{\theta_2}}{\tau_\mY} \right)
\in \calX_{\eps,+}^*\,,
\end{align*}
and for $Y\le 0$,
\begin{align*}
& \frac{g_{132}^{*,+}(\eta)-e^{-iY\eta}g_{12}^{*,-}(\eta')}{\eta}
\\ =&
e^{\{-ir_1\eta+\delta_{13}^\#(-\eta)\}Y}
\left\{\tg_{13}^{\#,*}(z_{12},\eta)
+\eta^{-1}\overline{\eta'}\tg_{12}^{0,*}(z_{12},-\eta')\right\}
\left(1+O\left(\frac{e^{\theta_3}}{\tau_\mY}\right) \right)\,,
\end{align*}
and
\begin{multline*}
\frac{g_{132}^{*,+}(\eta)-e^{-iY\eta}g_{12}^{*,-}(\eta')}{\eta}
\\-
e^{\{-ir_1\eta+\delta_{13}^\#(-\eta)\}Y}
\left\{\tg_{13}^{\#,*}(z_{12},\eta)
+\eta^{-1}\overline{\eta'}\tg_{12}^{0,*}(z_{12},-\eta')\right\}
\in \calX_{\eps,+}^*\,.
\end{multline*}
We can prove the rest in the same way.
\end{proof}

Finally, we give asymptotic
profiles of $\tg_{jk}^{0,*}(\eta)$ and $\tg_{jk}^{\#,*}(\eta)$ around $\eta=0$.
\begin{lemma}
\label{lem:tg-profile}
\begin{gather*}
\tg_{jk}^{0,*}(z,\eta)
 = -\frac{i}{4}\int_{-\infty}^z\pd_c\varphi_{c_{jk}}+O(\eta)\,,
\\
\tg_{13}^{\#,*}(z,\eta)=
\frac{i}{2}\frac{\k_3-\k_2}{(\k_2-\k_1)^2(\k_3-\k_1)^2}\varphi_{c_{12}}(z)
-\frac{ir_2}{4}\int_{-\infty}^z\pd_c\varphi_{c_{12}}(z)+O(\eta)\,,
\\
\tg_{12}^{\#,*}(z,\eta)=
\frac{i}{2}\frac{\k_2-\k_3}{(\k_2-\k_1)^2(\k_3-\k_1)^2}\varphi_{c_{13}}(z)
-\frac{i}{4r_2}\int_{-\infty}^z\pd_c\varphi_{c_{13}}(z)
+O(\eta)\,.
\end{gather*}
\end{lemma}
\begin{proof}
The first estimate follows from Remark~\ref{rem:gij0}.
Since $\tg_{jk}^{\#,*}(\eta)=\pd_{\bar{\eta}} g_{jk}^{\#,*}(0)+O(\eta)$,
 it follows from \eqref{eq:defg13-s*}--\eqref{eq:C1jk} that
\begin{multline*}
\tg_{13}^{\#,*}(z,\eta)= \pd_{\bar{\eta}} g_{13}^{\#,*}(z,0)+O(\eta)
\\= 
-\frac{\pd_\eta C_{32}(0)}{2}\sqrt{c_{12}}\sech^2\sqrt{c_{12}}z
-\frac{C_{32}(0)}{2}\pd_\eta\delta_{13}(0)
\pd_z(ze^{\sqrt{c_{12}}z}\sech\sqrt{c_{12}}z)+O(\eta)
\\=
\frac{i}{2}\frac{\k_3-\k_2}{(\k_2-\k_1)^2(\k_3-\k_1)^2}\varphi_{c_{12}}
-\frac{i}{4}\frac{\k_2-\k_1}{\k_3-\k_1}\int_{-\infty}^z\pd_c\varphi_{c_{12}}
+O(\eta)\,.
\end{multline*}
We use $\varphi_c(z)=2c\sech^2\sqrt{c}z$ and
$\pd_z(ze^{\sqrt{c}z}\sech\sqrt{c}z)=\sqrt{c}\int_{-\infty}^z\pd_c\varphi_c$
in the last line.
We can estimate $\tg_{12}^{\#,*}(\eta)$ in the same way.
\end{proof}
\bigskip

\section{Orthogonality relations between $\mathbf{W_j}(0)$ and
  $\mathbf{W_k^*}(0)$}

To solve $\mathcal{A}_2(0)\mathbf{V}=\mathcal{B}_2\mathbf{U_2}$
in the proof of Proposition~\ref{prop:spec-res},
we need to find solutions to \eqref{eq:efVb} with $\beta=0$ that are
linearly independent of $V(\k_1)$ and $V(\k_3)$. Let
\begin{gather*}
V_{12}=(\k_3-\k_2)\{\pd_\beta V(\k_1)-1\}
-(\k_3-\k_1)\{\pd_\beta\widehat{V}_1(\k_2)+1\}\,,
\\
V_{23}=(\k_2-\k_1)\pd_\beta V(\k_3)-(\k_3-\k_1)\pd_\beta \widehat{V}_3(\k_2)\,.
\end{gather*}
\begin{lemma}
  \label{lem:sol-l-ad-0}
If $V=V_{12}$ or $V= V_{23}$, then
\begin{equation}
  \label{eq:12}
(3\pd_Y^2-4b_2\pd_X\pd_Y-\pd_XA\pd_X)V=0\,,
\end{equation}
and $V\in L^2(\R;e^{-2\a X}dX)$ for $Y\in\R$ and $\a$ satisfying
$0<\a<\min\{\k_2-\k_1,\,\k_3-\k_2\}$.
\end{lemma}
\begin{proof}
Let
\begin{gather*}
\gamma_1(\eps)=\k_1+\eps\lambda'(\k_2)\,,\quad
  \gamma_2(\eps)=\k_2+\eps\lambda'(\k_1)\,,\\
  \overline{V}_1(\eps)=(\gamma_2(\eps)-\k_2)
  e^{\{\lambda(\gamma_2(\eps))-\lambda(\gamma_1(\eps))\}t}
  \Phi^1(-i\gamma_1(\eps))\Phi^{1,*}(-i\gamma_2(\eps))\,,
\\
\gamma_3(\eps)=\k_3+\eps\lambda'(\k_2)\,,\quad
\gamma_4(\eps)=\k_2+\eps\lambda'(\k_3)\,,
\\
\overline{V}_3(\eps)=(\gamma_4(\eps)-\k_2)
e^{\{\lambda(\gamma_4(\eps))-\lambda(\gamma_3(\eps))\}t}
\Phi^1(-i\gamma_3(\eps))\Phi^{1,*}(-i\gamma_4(\eps))\,.
\end{gather*}
It follows from \eqref{eq:6} that $\overline{V}_1(\eps)$ and $\overline{V}_3(\eps)$
depend only on $X$ and $Y$. 
Thus by Lemma~\ref{lem:prodJdJ},
\begin{gather*}
(3\pd_Y^2-4b_2\pd_X\pd_Y-\pd_XA\pd_X)\overline{V}_1(\eps)
=4\{\lambda(\gamma_2)-\lambda(\gamma_1)\}\pd_X\overline{V}_1(\eps)\,,
\\
(3\pd_Y^2-4b_2\pd_X\pd_Y-\pd_XA\pd_X)\overline{V}_3(\eps)
=4\{\lambda(\gamma_4)-\lambda(\gamma_3)\}\pd_X\overline{V}_3(\eps)\,.
\end{gather*}
Since $\lambda(\k_1)=\lambda(\k_2)=0$ and
$\gamma_1'(0)\lambda'(\k_1)=\gamma_2'(0)\lambda'(\k_2)$,
\begin{equation*}
(3\pd_Y^2-4b_2\pd_X\pd_Y-\pd_XA\pd_X)\pd_\eps\overline{V}_1(0)=  0\,.  
\end{equation*}
By the definitions,
\begin{align*}
\pd_\eps \overline{V}_1(0)=& 
\pd_\eps V(\gamma_1)|_{\eps=0}+ \pd_\eps\widehat{V}_1(\gamma_2)|_{\eps=0}
\\ =& \lambda'(\k_2)\pd_\beta V(\k_1)+
\lambda'(\k_1)\pd_\beta \widehat{V}_1(\k_2) 
\\=&
(\k_2-\k_1)\{(\k_3-\k_2)\pd_\beta V(\k_1)
-(\k_3-\k_1)\pd_\beta \widehat{V}_1(\k_2)\}
\\=& (\k_2-\k_1)(V_{12}+2\k_3-\k_1-\k_2)\,.
\end{align*}
Similarly,
\begin{align*}
 (3\pd_Y^2-4b_2\pd_X\pd_Y-\pd_XA\pd_X)\pd_\eps\overline{V}_3(0)=0\,,
\end{align*}
\begin{align*}
\pd_\eps \overline{V}_3(0)=& \lambda'(\k_2)\pd_\beta V(\k_3)
+\lambda'(\k_3)\pd_\beta \widehat{V}_3(\k_2) 
  \\=&
(\k_3-\k_2)\{(\k_2-\k_1)\pd_\beta V(\k_3)
-(\k_3-\k_1)\pd_\beta \widehat{V}_3(\k_2)\}
=(\k_3-\k_2)V_{23}\,.
\end{align*}
In view of \eqref{eq:13} and
\begin{align}
  \label{eq:11a}
\pd_\beta V(\k_j)=& \frac{e^{\theta_j}}{\tau_\mY}+(X+2\k_jY)V(\k_j)\,,
\\ \label{eq:11b}
\pd_\beta\widehat{V}_j(\k_2)=&
V(\k_j) \left(
 \frac{e^{\theta_1-\theta_2}}{\k_2-\k_1}+\frac{e^{\theta_3-\theta_2}}{\k_2-\k_3}
 -X-2\k_2Y \right)\,,
\end{align}
we have as $X\to\infty$,
\begin{gather*}
|\pd_\beta V(\k_1)|+|\pd_\beta \widehat{V}_1(\k_2)|
=O\left((1+|X|)e^{-(\k_3-\k_1)X}+e^{-(\k_2-\k_1)X}\right)\,,
\\
|\pd_\beta V(\k_3)-1|+|\pd_\beta \widehat{V}_3(\k_2)+1|
=O\left((1+|X|)e^{-(\k_3-\k_2)X}+e^{-(\k_2-\k_1)X}\right)\,,
\end{gather*}
and as $X\to-\infty$,
\begin{gather*}
|\pd_\beta V(\k_1)-1|+|\pd_\beta\widehat{V}_1(\k_2)+1|
=O\left((1+|X|)e^{(\k_2-\k_1)X}+e^{(\k_3-\k_2)X}\right)\,,
\\
|\pd_\beta V(\k_3)|+|\pd_\beta \widehat{V}_3(\k_2)|
=O\left((1+|X|)e^{(\k_3-\k_1)X}+e^{(\k_3-\k_2)X}\right)\,.
\end{gather*}
Thus we prove Lemma~\ref{lem:sol-l-ad-0}.
\end{proof}

Let
\begin{gather*}
\mathbf{W_3}(0)=\left(\pd_XV_{12},\pd_YV_{12}\right)^T\,,\quad
\mathbf{W_3^*}(0)=
\left(-\frac{4}{3}b_2V_{12}+\pd_X^{-1}\pd_YV_{12},V_{12}\right)^T\,,
\\
\mathbf{W_4}(0)=\left(\pd_XV_{23},\pd_YV_{23}\right)^T\,,\quad
\mathbf{W_4^*}(0)=
\left(-\frac{4}{3}b_2V_{23}+\pd_X^{-1}\pd_YV_{23},V_{23}\right)^T\,.
\end{gather*}

We will show orthogonality relations between $\mathbf{W_j}(0)$ and
$\mathbf{W_k^*}(0)$.
\begin{lemma}
  \label{lem:orth-bfUU*0}
Let $0<\a<\min(\k_2-\k_1\k_3-\k_2)$. Then for $j, k=1, \ldots, 4$,
$\mathbf{W_j}(0)\in C^1(\R;e^{2\a X}dX)$,
$\mathbf{W_k^*}(0)\in C^1(\R;e^{-2\a X}dX)$ and
\begin{gather*}
\la \mathbf{W_j}(0),\mathbf{W_k^*}(0)\ra=
-\la \mathbf{W_k}(0),\mathbf{W_j^*}(0)\ra\,.
\\
\la \mathbf{W_1}(0),\mathbf{W_3^*}(0)\ra\ne0\,,\quad
  \la \mathbf{W_1}(0),\mathbf{W_k^*}(0)\ra=0\text{ if $k\ne3$,}
\\
\la \mathbf{W_2}(0),\mathbf{W_4^*}(0)\ra\ne0\,, \quad
\la \mathbf{W_2}(0),\mathbf{W_k^*}(0)\ra=0\text{ if $k\ne4$.}
\end{gather*}
\end{lemma}
To prove Lemma~\ref{lem:orth-bfUU*0}, we need asymptotic profiles of
$\mathbf{W_j}$ and $\mathbf{W_k^*}$ as $Y\to\pm\infty$.
\begin{lemma}
 \label{lem:Vjk-profile}
If $(X,Y)\in\R^2_+$,
 \begin{gather}
 \label{eq:V1-1}
V(\k_1) =-\frac{\sqrt{c_{13}}}{2}\sech^2\sqrt{c_{13}}z_{13}
+O\left(e^{\min\{\theta_2-\theta_1, \theta_1+\theta_2-2\theta_3\}}\right)\,,
\\  \label{eq:V3-1}
V(\k_3)=\frac{\sqrt{c_{13}}}{2}\sech^2\sqrt{c_{13}}z_{13}
+O\left(e^{\min\{\theta_2+\theta_3-2\theta_1,\theta_2-\theta_3\}}\right)\,.
 \end{gather}
If $(X,Y)\in\R^2_-$,
\begin{gather}
\label{eq:V1-2}
V(\k_1)=\left\{
  \begin{aligned}
&  O\left(e^{\min\{\theta_1-\theta_2,\theta_1-\theta_3\}}\right)
\quad\text{for $z_{13}\ge0$,}
\\ &
 -\frac{\sqrt{c_{12}}}{2}\sech^2\sqrt{c_{12}}z_{12}
+ O\left(e^{\min\{\theta_3-\theta_1,\theta_1+\theta_3-2\theta_2\}}\right)
\quad\text{for $z_{13}\le0$,}
  \end{aligned}\right.
\\ \label{eq:V3-2}
 V(\k_3)=\left\{
  \begin{aligned}
& \frac{\sqrt{c_{23}}}{2}\sech^2\sqrt{c_{23}}z_{23}
+O\left(e^{\min\{\theta_1+\theta_3-2\theta_2,\theta_1-\theta_3\}}\right)
\quad\text{for $z_{13}\ge0$,}
\\ &
 O\left(e^{\min\{\theta_3-\theta_1,\theta_3-\theta_2\}}\right)
\quad\text{for $z_{13}\le0$,}
  \end{aligned}\right.
\\ \label{eq:V12-2}
V_{12}=\left\{
  \begin{aligned}
& 2a_{12}-2\k_3+O\left((1+|X|+|Y|)e^{\theta_1-\theta_2}\right)
\quad\text{for $z_{13}\ge0$,}
\\ &
(a_{12}-\k_3)\pd_X\{z_{12}(1+\tanh \sqrt{c_{12}}z_{12})\}
-\left(\frac{\sqrt{c_{13}}}{2}+2c_{12}^{3/2}Y\right)\sech^2\sqrt{c_{12}}z_{12}
\\ & \qquad
+O\left((1+|X|+|Y|)e^{\theta_3-\theta_2}\right)
\quad\text{for $z_{13}\le0$,}
  \end{aligned}\right.
\\
  \label{eq:V23-2}
V_{23}=\left\{
  \begin{aligned}
& (a_{23}-\k_1)\pd_X\{z_{23}(1+\tanh \sqrt{c_{23}}z_{23})\}
-\left(\frac{\sqrt{c_{13}}}{2}+2c_{23}^{3/2}Y\right)\sech^2\sqrt{c_{23}}z_{23}
\\ & \qquad 
+O\left((1+|X|+|Y|)e^{\theta_1-\theta_2}\right)
\quad\text{for $z_{13}\ge0$,}
\\ &
 O\left((1+|X|+|Y|)e^{\theta_3-\theta_2}\right)
\quad\text{for $z_{13}\le0$.}
  \end{aligned}\right.
\end{gather}
\end{lemma}
\begin{proof}
We have \eqref{eq:V1-1}--\eqref{eq:V3-2} from \eqref{def:Phi1s} and the fact that
\begin{equation}
  \label{eq:8}
  \frac{1}{\tau_\mY}=\frac{1}{e^{\theta_{i_1}+\theta_{i_2}}}
\left(1+
O(e^{\min\{\theta_{i_3}-\theta_{i_1},\,\theta_{i_3}-\theta_{i_2}\}})\right)
\quad\text{for $\{i_1,i_2,i_3\}=\{1,2,3\}$.}
\end{equation}
\par
By \eqref{eq:11a}, \eqref{eq:11b},
\eqref{eq:V1-2}, \eqref{eq:V3-2} and \eqref{eq:8},
we have for $Y\le 0$ and $z_{13}\le0$,
\begin{equation*}
|\pd_\beta V(\k_3)|+ |\pd_\beta \widehat{V}_3(\k_2)|\lesssim
(1+|X|+|Y|)e^{\theta_3-\theta_2}\,,
\end{equation*}
\begin{align*}
 \pd_\beta V(\k_1)=&       
\frac{e^{\theta_1}}{e^{\theta_1}+e^{\theta_2}}
\left(1+O(e^{\min\{\theta_3-\theta_1,\,\theta_3-\theta_2\}})\right)
+(X+2\k_1Y)V(\k_1)
  \\=&
\frac{1}{2}\left(1-\tanh\sqrt{c_{12}}z_{12}\right)
+\left(c_{12}Y-\frac{\sqrt{c_{12}}}{2}z_{12}\right)\sech^2\sqrt{c_{12}}z_{12}
\\ &       
+O\left((1+|X|+|Y|)e^{\min\{\theta_3-\theta_1,\,\theta_3-\theta_2\}}\right)\,,
\end{align*}
\begin{align*}
  \pd_\beta \widehat{V}_1(\k_2)=&
V(\k_1)\left(\frac{e^{\theta_1-\theta_2}}{\k_2-\k_1}-X-2\k_2Y\right)
+O\left(e^{\min\{\theta_3-\theta_1,\,\theta_3+\theta_1-2\theta_2\}}\right)
\\=&
\frac12\left(\tanh \sqrt{c_{12}}z_{12}-1 \right)
+\left(\frac{1}{4}+\frac{\sqrt{c_{12}}}{2}z_{12}+c_{12}Y\right)
\sech^2 \sqrt{c_{12}}z_{12}
\\ &
+O\left((1+|X|+|Y|)e^{\theta_3-\theta_2}\right)\,.
\end{align*}
Similarly, we have for $Y\le 0$ and $z_{13}\ge0$,
\begin{equation*}
|\pd_\beta V(\k_1)|+|\pd_\beta \widehat{V}_1(\k_2)|
\lesssim (1+|X|+|Y|)e^{\theta_1-\theta_2}\,,
\end{equation*}
\begin{align*}
 \pd_\beta V(\k_3)=&
\frac{1}{2}\left(1+\tanh \sqrt{c_{23}}z_{23}\right)
+\left(\frac{\sqrt{c_{23}}}{2}z_{23}+c_{23}Y\right)\sech^2Z_{23}
\\ & +O\left((1+|X|+|Y|)e^{\min\{\theta_1-\theta_2,\,\theta_1-\theta_3\}}\right)\,,
\end{align*}
\begin{align*}
\pd_\beta \widehat{V}_3(\k_2)=&
-\frac12\left(1+\tanh\sqrt{c_{23}} z_{23} \right)
+\left(\frac{1}{4}-\frac{\sqrt{c_{23}}}{2}z_{23}+c_{23}Y\right)
\sech^2\sqrt{c_{23}}z_{23}
\\ & 
+O\left((1+|X|+|Y|)e^{\theta_1-\theta_2}\right)\,.
\end{align*}
Combining the above, we have \eqref{eq:V12-2} and \eqref{eq:V23-2}.
Thus we complete the proof.
\end{proof}
By Lemma~\ref{lem:Vjk-profile}, we have the following.
\begin{lemma}
  \label{lem:v1-v12-v23}
Let $0<\a<\min\{\k_2-\k_1, \k_3-\k_2\}$ and $\nu<(\k_2-\k_1)(\k_3-\k_2)$.
Then $\spann\{V(\k_1),\,V_{12},\,V_{23}\}\cap\calX(\nu)^*=\{0\}$.
\end{lemma}
\begin{proof}
We have $\pd_\beta V(\k_3)=O(1+Y)$ on $\R^2_+$ from \eqref{eq:11a} and \eqref{eq:V3-1}.
On the other hand, it follows from
\eqref{eq:2/tau}, \eqref{eq:11b} and \eqref{eq:V3-1} that
\begin{multline*}
\pd_\beta \widehat{V}_3(\k_2)=
V(\k_3)\left(\frac{e^{\theta_1-\theta_2}}{\k_2-\k_1}
-\frac{e^{\theta_3-\theta_2}}{\k_3-\k_2}\right)+O(1+|Y|)
\\= 
\frac{\sqrt{c_{13}}}{2}\sech^2\sqrt{c_{13}}z_{13}
\left(\frac{e^{-2\sqrt{c_{12}}z_{13}}}{\k_2-\k_1}
-\frac{e^{2\sqrt{c_{23}}z_{13}}}{\k_3-\k_2}\right)e^{(\k_2-\k_1)(\k_3-\k_2)Y}+O(1+|Y|)\,.
\end{multline*}
Combining the above, we see that there exist an interval $I$ of
positive length and positive constants $C_1$, $C_2$ such that
\begin{equation*}
|V_{23}|\ge C_1e^{(\k_2-\k_1)(\k_3-\k_2)Y}-C_2\quad\text{for $z_{13}\in I$ and $Y\ge0$.}
\end{equation*}
Thus we prove $V_{23}\not\in\calX(\nu)^*$.
\par
It follows from \eqref{eq:V1-2}, \eqref{eq:V12-2}
and \eqref{eq:V23-2} that  along the half line $z_{12}=0$ and $Y\le0$,
$V_{23}$ decays to $0$  as $Y\to-\infty$ whereas $V(\k_1)$ and $V_{12}$ do not decay to $0$ as $Y\to-\infty$. Thus for any constants $(c_1,c_2)\ne(0,0)$ and $c_3$,
there exist a nondegenerate interval $I$ 
and positive constants $C_1$ and $C_2$ such that
\begin{equation*}
|c_1V(\k_1)+c_2V_{12}+c_3V_{23}|\ge C_1\quad\text{ if $z_{12}\in I$ and $Y\le -C_2$,}
\end{equation*}
and $c_1V(\k_1)+c_2V_{12}+c_3V_{23}\not\in\calX(\nu)^*$.
Thus we prove Lemma~\ref{lem:v1-v12-v23}.  
\end{proof}

Finally, we will prove Lemma~\ref{lem:orth-bfUU*0}.
\begin{proof}[Proof of Lemma~\ref{lem:orth-bfUU*0}]
 By integration by parts, we have
 $\la \mathbf{W_j},\mathbf{W_k^*}\ra+\la \mathbf{W_k},\mathbf{W_j^*}\ra
 =0$. In particular, $\la \mathbf{W_j},\mathbf{W_j^*}\ra=0$.
It follows from Lemmas~\ref{lem:sol-l-ad-0}, \ref{lem:sol-l-ad} and
\ref{lem:bfUU*} that
\begin{equation}
  \label{eq:15}
\la \mathbf{W_j},\mathbf{W_k^*}\ra
=\lim_{Y\to\pm\infty}\la \mathbf{W_j},\mathbf{W_k^*}\ra\,.    
\end{equation}
By the definitions,
\begin{multline*}
\la \mathbf{W_1},\mathbf{W_2^*}\ra
= \frac{4}{3}(b_2-3a_{13})\int_\R V(\k_1)\pd_XV(\k_3)\,dX     
  \\  +
\int_\R V(\k_1)(2a_{13}\pd_X-\pd_Y)V(\k_3)\,dX
+\int_\R V(\k_3)(\pd_Y-2a_{13}\pd_X)V(\k_1)\,dX\,.
\end{multline*}
Substituting \eqref{eq:2/tau}, \eqref{eq:V1-1} and \eqref{eq:V3-1}
into the above and letting $Y\to\infty$, we have
\begin{align*}
\lim_{Y\to\infty}\la \mathbf{W_1},\mathbf{W_2^*}\ra =& 
-\frac{c_{13}}{3}(b_2-3a_{13})\int_\R \sech^2\sqrt{c_{13}}z_{13}
\pd_X\left(\sech^2\sqrt{c_{13}}z_{13}\right)\,dX
\\=& 0\,.
\end{align*}
Combining the above with \eqref{eq:15}, we have
$\la \mathbf{W_1},\mathbf{W_2^*}\ra=0$ for every $Y\in\R$.

By \eqref{eq:V1-2}--\eqref{eq:V23-2} and \eqref{eq:15},
\begin{equation*}
\la \mathbf{W_j},\mathbf{W_k^*}\ra=0
\quad\text{for $\{j,k\}=\{1,4\}, \{2,3\}, \{3,4\}$.}
\end{equation*}
Moreover,  as $Y\to-\infty$, we have for $z_{13}\le0$,
\begin{equation}
  \label{eq:18}
(\pd_Y-2a_{12}\pd_X)V(\k_1)\simeq 0\,,\enskip
(\pd_Y-2a_{12}\pd_X)V_{12}\simeq  -2c_{12}^{3/2}\sech^2\sqrt{c_{12}}z_{12}\,,
\end{equation}
and for $z_{13}\ge0$,
\begin{equation}
  \label{eq:19}
(\pd_Y-2a_{23}\pd_X)V(\k_3)\simeq 0\,,\enskip
(\pd_Y-2a_{23}\pd_X)V_{23}\simeq -2c_{23}^{3/2}\sech^2\sqrt{c_{23}}z_{23}\,.  
\end{equation}
\begin{equation*}
\end{equation*}
By \eqref{eq:15}--\eqref{eq:19},
  \begin{align*}
& \la \mathbf{W_1},\mathbf{W_3^*}\ra
= \frac{4}{3}(b_2-3a_{12})\int_\R V(\k_1)\pd_XV_{12}\,dX
\\ & +\int_\R (\pd_Y-2a_{12}\pd_X)V(\k_1)V_{12}\,dX
 -\int_\R V(\k_1)(\pd_Y-2a_{12}\pd_X)V_{12}\,dX
\\=& \frac{4}{3}(b_2-3a_{12})\lim_{Y\to-\infty}\int_\R V(\k_1)\pd_XV_{12}\,dX
-c_{12}^2\int_\R \sech^4\sqrt{c_{12}}z_{12}\,dX
    \\=& \{(\k_3-a_{12})^2-c_{12}\}c_{12}\int_\R \sech^4\sqrt{c_{12}}z_{12}\,dz_{12}
\\ =& \frac{2}{3}(\k_3-\k_1)(\k_3-\k_2)(\k_2-\k_1)\ne0\,,
  \end{align*}
  \begin{align*}
& \la \mathbf{W_2},\mathbf{W_4^*}\ra
= \frac{4}{3}(b_2-3a_{23})\int_\R V(\k_3)\pd_XV_{23}\,dX
\\ & +\int_\R (\pd_Y-2a_{23}\pd_X)V(\k_3)V_{23}\,dX
 -\int_\R V(\k_3)(\pd_Y-2a_{23}\pd_X)V_{23}\,dX
\\=& -\frac{2}{3}(\k_3-\k_1)(\k_3-\k_2)(\k_2-\k_1)\ne0\,.
  \end{align*}
Thus we complete the proof.
\end{proof}

\end{document}